\documentclass[11pt,letterpaper]{article}
\usepackage[latin9]{inputenc}
\usepackage{amsmath}
\usepackage{amsthm}
\usepackage{amssymb}
\usepackage{graphicx}
\PassOptionsToPackage{normalem}{ulem}
\usepackage{ulem}

\makeatletter

\providecommand{\tabularnewline}{\\}

\theoremstyle{plain}
\newtheorem{thm}{\protect\theoremname}
\theoremstyle{plain}
\newtheorem{prop}[thm]{\protect\propositionname}
\theoremstyle{plain}
\newtheorem{lem}[thm]{\protect\lemmaname}

\usepackage{EMstyle}
\usepackage{enumitem}
\usepackage{breqn}
\usepackage{bbm}
\usepackage{cite}

\DeclareMathOperator{\spa}{span} 
\DeclareMathOperator{\TP}{TP}

\allowdisplaybreaks

\global\long\def\s[#1]{\textnormal{\scriptsize #1}}
\global\long\def\st[#1]{\textnormal{\tiny #1}}

\global\long\def\P{\mathbb{P}}
\global\long\def\E{\mathbb{E}}

\global\long\def\I{\mathbbm{1}}
\global\long\def\v[#1]{\mathbf{#1}} 
\global\long\def\m[#1]{\boldsymbol{#1}} 

\global\long\def\dkl{\mathrm{d_{KL}}}
\global\long\def\dchis{\mathrm{d}_{\chi^2}}

\global\long\def\dee{\partial}
\newcommand{\ind}{\perp\!\!\!\!\perp}

\global\long\def\gl[#1]{\mathsf{#1}} 

\global\long\def\r[#1]{#1}
\global\long\def\d{\mathrm{d}}

\global\long\def\eqd{\stackrel{d}{=}}

\global\long\def\dfn{:=}

\global\long\def\trre[#1,#2]{\overset{{\scriptstyle (#2)}}{#1}} 

\author{Nir Weinberger and Guy Bresler}

\providecommand{\lemmaname}{Lemma}
\providecommand{\propositionname}{Proposition}
\providecommand{\theoremname}{Theorem}

\makeatother

\providecommand{\lemmaname}{Lemma}
\providecommand{\propositionname}{Proposition}
\providecommand{\theoremname}{Theorem}

\begin{document}

\title{The EM Algorithm is Adaptively-Optimal for Unbalanced Symmetric Gaussian
Mixtures\thanks{The first author with the Department of Electrical Engineering, Technion
- Israel Institute of Technology, and the second author is with IDSS
and LIDS, Massachusetts Institute of Technology, MA, USA. Emails:
\{nirwein@technion.ac.il, guy@mit.edu\}. This work was supported by
the MIT-Technion fellowship, the Viterbi scholarship from the Technion,
MIT-IBM Watson AI Lab, and NSF CAREER award CCF-1940205.}}
\maketitle
\begin{abstract}
This paper studies the problem of estimating the means $\pm\theta_{*}\in\mathbb{R}^{d}$
of a symmetric two-component Gaussian mixture $\delta_{*}\cdot N(\theta_{*},I)+(1-\delta_{*})\cdot N(-\theta_{*},I)$
where the weights $\delta_{*}$ and $1-\delta_{*}$ are unequal. Assuming
that $\delta_{*}$ is known, we show that the population version of
the EM algorithm globally converges if the initial estimate has non-negative
inner product with the mean of the larger weight component. This can
be achieved by the trivial initialization $\theta_{0}=0$. For the
empirical iteration based on $n$ samples, we show that when initialized
at $\theta_{0}=0$, the EM algorithm adaptively achieves the minimax
error rate $\tilde{O}\Big(\min\Big\{\frac{1}{(1-2\delta_{*})}\sqrt{\frac{d}{n}},\frac{1}{\|\theta_{*}\|}\sqrt{\frac{d}{n}},\left(\frac{d}{n}\right)^{1/4}\Big\}\Big)$
in no more than $O\Big(\frac{1}{\|\theta_{*}\|(1-2\delta_{*})}\Big)$
iterations (with high probability). We also consider the EM iteration
for estimating the weight $\delta_{*}$, assuming a fixed mean $\theta$
(which is possibly mismatched to $\theta_{*}$). For the empirical
iteration of $n$ samples, we show that the minimax error rate $\tilde{O}\Big(\frac{1}{\|\theta_{*}\|}\sqrt{\frac{d}{n}}\Big)$
is achieved in no more than $O\Big(\frac{1}{\|\theta_{*}\|^{2}}\Big)$
iterations. These results robustify and complement recent results
of Wu and Zhou \cite{wu2019EM} obtained for the equal weights case
$\delta_{*}=1/2$. 
\end{abstract}

\section{Introduction}

The expectation-maximization (EM) algorithm is a heuristic formulated
in \cite{dempster1977maximum} to approximate the maximum likelihood
estimator (MLE) in parametric models $(X,S)\sim P_{\theta}(x,s)$
when $X$ is observed, but $S$ is latent. Remarkably, despite its
simplicity, widespread use, and rich history \cite{mclachlan2007algorithm,gupta2011theory},
no theoretical guarantees on its performance for finite number of
iterations and samples were established until recently. The first
such explicit guarantees were obtained in \cite{balakrishnan2017statistical},
which stated general bounds on the statistical precision, the convergence
rate, and the ``basin of attraction'' (the distance of the initial
estimate from the ground truth sufficient to obtain a statistically
accurate solution). These bounds apply to any latent variables model,
yet require verifying several conditions for each concrete model.
As a canonical example, these conditions were explicitly verified
by \cite{balakrishnan2017statistical} for the symmetric two-component
Gaussian mixture (2-GM). The resulting guarantees are not sharp, however,
both in the strong conditions required for their validity, as well
as their distance from the accuracy guarantees of optimal algorithms.
Consequently, a dedicated analysis of EM for 2-GM was conducted by
various authors \cite{klusowski2016statistical,Wu2016OnTC,xu2016global,pmlr-v65-daskalakis17b,wu2019EM,dwivedi2018singularity,dwivedi2018theoretical,dwivedi2019challenges}.
The performance of EM for 2-GM with \emph{balanced} components, i.e.,
when both weights equal $1/2$, was by and large recently settled
in \cite{wu2019EM}.

In this paper, we proceed in the direction of \cite{wu2019EM}, and
sharply analyze a slight variation of the balanced 2-GM model --
namely, the \emph{unbalanced} symmetric 2-GM model. Some of the key
arguments in \cite{wu2019EM} strongly depend on the symmetry properties
of the EM iteration, which are a direct result of the symmetry in
the model. It seems challenging to adapt these arguments to the unbalanced
model, where symmetry breaks down due to the unequal weights. Our
analysis therefore uses \emph{indirect }arguments, which are based
on \emph{comparisons} between the EM iterations of unbalanced models
for different weights. In particular, we compare the iterations for
unbalanced models with the iterations for the \emph{balanced} model,
since the latter is already known to globally converge \cite{wu2019EM}.
For the population iteration, we prove that increasing the larger
of the two weights, that is, enhancing the model imbalance, makes
the corresponding EM iteration converge faster. By contrast, this
increase also increases our empirical error bound, i.e., the bound
on the difference between the empirical iteration and the population
iteration. As we prove, however, this does not result in deterioration
of the statistical accuracy of the estimate because this increased
error is compensated for by the improved convergence of the population
iteration. Hence, the overall statistical accuracy actually improves
when the model is more unbalanced.

\subsection{EM for two-component Gaussian mixture \label{subsec:EM for Gaussian mixture}}

The symmetric two-component Gaussian mixture (2-GM) model in $d\geq1$
dimensions is given by

\begin{equation}
P_{\theta,\rho}=\tfrac{1+\rho}{2}\cdot N(\theta,I_{d})+\tfrac{1-\rho}{2}\cdot N(-\theta,I_{d})\,.\label{eq: Gaussian mixture model}
\end{equation}
The goal is to estimate the parameter $\theta_{*}\in\mathbb{R}^{d}$
from $n$ samples $(X_{1},\ldots,X_{n})\stackrel{\tiny\mathrm{i.i.d.}}{\sim}P_{\theta_{*},\rho_{*}}$
under the $\ell_{2}$ loss function $\ell(\theta,\theta_{*})=\|\theta-\theta_{*}\|$
when $\rho_{*}\neq0$ (unbalanced model), or under $\ell_{0}(\theta,\theta_{*})=\min(\|\theta-\theta_{*}\|,\|\theta+\theta_{*}\|)$
when $\rho_{*}=0$ (balanced model). The dimension $d$ is allowed
to be large, and both $d$ and $\rho_{*}$ may scale with the number
of samples $n$. Based on the $n$ samples and the value of $\rho_{*}$,
the EM algorithm defines a mapping $f_{n}(\theta)$ which is iteratively
applied to produce a sequence of estimates $\theta_{t}=f_{n}(\theta_{t-1})$
for all $t\geq1$, given an initial guess $\theta_{0}$. This mapping
$f_{n}$ is described in detail later in the introduction. We will
refer to $f_{n}(\theta)$ as the \emph{empirical iteration}, and to
the idealized operator $f(\theta)$ obtained by replacing empirical
averages with expected values as the \emph{population iteration}.

\paragraph*{Balanced GM.}

The general results of \cite{balakrishnan2017statistical} specialized
to the balanced 2-GM (\ref{eq: Gaussian mixture model}) ($\rho_{*}=0$)
require that the separation between the means is lower bounded as
$\|\theta_{*}\|=\Omega(1)$, and that the initial estimate $\theta_{0}$
is at most $\|\theta_{*}\|/4$ in $\ell_{2}$ distance from $\theta_{*}$.
When these two conditions hold, \cite{balakrishnan2017statistical}
states that EM converges to a neighborhood of $\theta_{*}$ of radius
$O(\sqrt{d/n})$ (i.e., parametric error rate), after no more than
$O(1/\|\theta_{*}\|^{2})$ iterations. The qualifying conditions above
are problematic for several reasons: (1) Without knowing $\theta^{*}$
one has no way of knowing when the separation condition holds; (2)
EM can be slow and inaccurate when there is no separation between
the components \cite{redner1984mixture} and in this case no guarantees
are provided by \cite{balakrishnan2017statistical}; (3) One of the
main challenges in utilizing EM is the choice of initial guess. A
common method is attempting multiple random guesses \cite{karlis2003choosing}
followed by a choice of the optimal converged solution. For a high-dimensional
parameter, the guarantee of \cite{balakrishnan2017statistical} on
the volume of the basin of attraction that ensures good convergence
is negligible compared to the volume of the feasible set of parameters,
and hence randomly initializing is not proved to succeed.

These drawbacks have lead to various attempts to sharpen the above
results \cite{klusowski2016statistical,Wu2016OnTC,xu2016global,pmlr-v65-daskalakis17b,wu2019EM},
which will be discussed in more detail in Section \ref{subsec:Related-work}.
For the population iteration, the papers \cite{xu2016global,pmlr-v65-daskalakis17b,wu2019EM}
proved \emph{global} convergence to $\pm\theta_{*}$ at a geometric
rate, unless the initial guess $\theta_{0}$ is orthogonal to $\theta_{*}$
(in which case EM converges to the saddle point $\theta=0$). For
the empirical iteration, sharp high-probability guarantees were obtained
in \cite{wu2019EM} as follows: In the worst case, without any separation
condition, the EM algorithm applied to (\ref{eq: Gaussian mixture model})
achieves an error rate of $\tilde{O}((d/n)^{1/4})$ in at most $O(\sqrt{n})$
iterations. If, however, a separation of $\|\theta_{*}\|=\Omega((\frac{\log^{3}n\cdot d}{n})^{1/4})$
holds, then an error rate of $O(\frac{1}{\|\theta_{*}\|}\sqrt{\frac{\log^{3}n\cdot d}{n}})$
is achieved by EM after no more than $O(\frac{\log n}{\|\theta_{*}\|^{2}})$
iterations, and in addition, the EM iteration converges to the MLE.
Evidently, for $\|\theta_{*}\|=\Omega(1)$, this implies a parametric
error rate in the number of samples, and geometric rate in the number
of iterations. Hence, the EM algorithm\emph{ adapts} to the actual
separation between the two means (as captured by $\|\theta_{*}\|$),
to achieve error rate of $\tilde{O}(\min\{\frac{1}{\|\theta_{*}\|}\sqrt{d/n},(d/n)^{1/4}\})$.
Moreover, no other estimation technique can perform significantly
better since, up to logarithmic factors, this error rate matches the
\emph{local} minimax rate \cite[Appendix B]{wu2019EM}. Remarkably,
it was also shown in \cite{wu2019EM} that these guarantees are achieved
by a random initialization of the EM algorithm, in which $\theta_{0}$
is an isotropic random $d$-dimensional vector scaled to have appropriately
low norm.

\paragraph{Unbalanced GM and preview of results.}

In this work, we study the model (\ref{eq: Gaussian mixture model})
for $\rho_{*}\in(0,1)$. The value of $\rho_{*}$ may be fixed, or,
more interestingly, $\rho_{*}\equiv\rho_{*,n}\to0$ as $n\to\infty$
at some arbitrary rate. Note that the samples from the model (\ref{eq: Gaussian mixture model})
are equal in distribution to 
\begin{equation}
X=S\theta+Z\,,\label{eq: Guassian mixture model - random variables}
\end{equation}
where $S\in\{\pm1\}$ is such that $\P[S=1]=(1+\rho_{*})/2$ and $Z\sim N(0,I_{d})$,
with $S$ and $Z$ independent. Intuitively, moving $\rho_{*}$ away
from $0$ reduces uncertainty in the signs $\{S_{i}\}$, and one might
expect that this would lead to better error rates for estimating $\theta_{*}$.
Note that the problem is indeed trivial for the extreme case $\rho_{*}=1$
in which case (\ref{eq: Gaussian mixture model}) coincides with the
Gaussian location model. More generally, it seems helpful that for
$\rho_{*}\neq0$ the expectation $\E[X]=\rho_{*}\theta_{*}$ is a
vector in the direction of $\theta_{*}$.

While estimation seems easier for $\rho_{*}\neq0$, in this case the
model (\ref{eq: Gaussian mixture model}) is no longer balanced, and
this makes a direct analysis of the EM iteration difficult. Nonetheless,
we prove a global convergence property for the population iteration,
which shows that any initial guess $\theta_{0}$ with $\langle\theta_{0},\theta_{*}\rangle\geq0$
converges to $\theta_{*}$ (including the trivial initialization $\theta_{0}=0$).
We also show that the EM iteration might have a spurious (stable)
fixed point $\theta_{-}\neq-\theta_{*}$ which satisfies $\langle\theta_{-},\theta_{*}\rangle<0$
(whose existence depends on the value of $(\rho_{*},\theta_{*})$).
This phenomenon does not occur in the balanced case.

For the empirical iteration, we first note that a method-of-moments
estimator $\frac{1}{\rho_{*}}\E_{n}[X]\dfn\frac{1}{\rho_{*}}\sum_{i=1}^{n}X_{i}$
achieves an error rate of $O(\frac{1}{\rho_{*}}\sqrt{d/n})$. In addition,
an estimator can always ignore the reduced uncertainty in the signs,
formally, by multiplying each sample with a random sign $R_{i}\in\{\pm1\}$
such that $\P[R_{i}=1]=1/2$ for each $i\in[n]$. This reduces the
$\rho_{*}\neq0$ case to the $\rho_{*}=0$ case, and then an error
rate of $\tilde{O}(\min\{\frac{1}{\|\theta_{*}\|}\sqrt{d/n},(d/n)^{1/4}\})$
can be achieved, using the balanced EM iteration.\footnote{In the latter case, this error is actually only w.r.t. the sign-ambiguous
loss function $\ell_{0}$ (see Proposition \ref{prop: empirical d>1 known delta small rho}).} The main result of this paper is analysis of the \emph{unbalanced}
EM iteration for the estimation of $\theta_{*}$, which shows that
the EM iteration \emph{adaptively} achieves the minimum of both error
rates, i.e., $\tilde{O}(\min\{\frac{1}{\rho_{*}}\sqrt{d/n},\frac{1}{\|\theta_{*}\|}\sqrt{d/n},(d/n)^{1/4}\})$.
As for the balanced case, this error rate obtained by the EM algorithm
coincides with the local minimax rate for any $\rho_{*}$, up to logarithmic
terms.

\subsection{Main result\label{subsec:Main-result}}

It will be convenient throughout to use the weight parameter $\delta\dfn(1-\rho)/2$
interchangeably with $\rho$ according to convenience.\footnote{The notation used in this section is standard. See Section \ref{subsec:Notation-conventions}
for notational conventions.} We denote the corresponding \emph{inverse-temperature parameter }by
\begin{equation}
\beta_{\rho}\dfn\frac{1}{2}\log\frac{1+\rho}{1-\rho}=\tanh^{-1}(\rho)\label{eq: beta definition as a function of rho}
\end{equation}
and let $\rho_{\beta}$ denote the inverse relation. With a slight
abuse of notation from (\ref{eq: beta definition as a function of rho}),
we also denote $\beta_{\delta}\dfn\frac{1}{2}\log\frac{1-\delta}{\delta}$
(and sometimes just $\beta$). Let $\theta_{*}\in\mathbb{R}^{d}$
and $\rho_{*}\in[0,1]$ (or $\delta_{*}\in[0,1/2]$) denote the ground
truth of the model (\ref{eq: Gaussian mixture model}). Given $n$
i.i.d. samples $\underline{X}=(X_{1},\ldots,X_{n})\stackrel{\tiny\mathrm{i.i.d.}}{\sim}P_{\theta_{*},\rho_{*}}$,
the goal is to estimate the parameter $\theta_{*}$ under the $\ell_{2}$
loss function, up to the identifiability of the model. For $\rho_{*}>0$
this amounts to the standard loss function $\ell(\theta,\theta_{*})=\|\theta-\theta_{*}\|$
and when $\rho_{*}=0$ then the loss function is $\ell_{0}(\theta,\theta_{*})=\min\left\{ \|\theta-\theta_{*}\|,\|\theta+\theta_{*}\|\right\} $.

\paragraph{Assumptions.}

Our results will depend on the following \emph{global assumptions}: 
\begin{enumerate}
\item \uline{Norm assumption:} There exists $\gl[C]_{\theta}>0$ such
that $\|\theta_{*}\|\leq\gl[C]_{\theta}$. 
\item \uline{Unbalancedness assumption:} There exists $\gl[C]_{\rho}\in(0,1)$
such that $|\rho_{*}|\leq\gl[C]_{\rho}$. 
\end{enumerate}
Because $\rho\mapsto\frac{1}{2}\log\frac{1+2a}{1-2a}$ is convex and
increasing in $[0,1/2)$, an immediate consequence of the unbalancedness
assumption is that $|\beta_{\rho_{*}}|\leq\gl[C]_{\beta}$ holds for
$\gl[C]_{\beta}\dfn\frac{1}{2}\log\frac{1+\gl[C]_{\rho}}{1-\gl[C]_{\rho}}$,
and that there exist $(\gl[\underline{C}]_{\beta},\gl[\overline{C}]_{\beta})$
such that $\gl[\underline{C}]_{\beta}\rho_{*}\leq|\beta_{\rho_{*}}|\leq\gl[\overline{C}]_{\beta}\rho_{*}.$
These assumptions are based on the fact that the interesting regime
is in which $\|\theta_{*}\|$ and $\rho_{*}$ are close to zero.

\paragraph{EM iteration.}

While we focus on estimating $\theta_{*}$ for a given $\rho_{*}$,
we will also consider the opposite case of estimating $\rho_{*}$,
and briefly discuss the joint estimation problem. Thus, we will next
consider the more general \emph{joint} iteration. The evolution of
the iterates $\{(\theta_{t},\rho_{t})\}_{t=1}^{\infty}$ of the EM
algorithm can be brought to a simple closed form we describe next.
To start, the density function of observed samples $X$ from (\ref{eq: Gaussian mixture model})
is given by 
\begin{align}
p_{\theta,\rho}(x) & =\left(\frac{1+\rho}{2}\right)\varphi(x-\theta)+\left(\frac{1-\rho}{2}\right)\varphi(x+\theta)\nonumber \\
 & =e^{-\|x\|^{2}/2}\cdot\varphi(x)\cdot\left[\left(\frac{1+\rho}{2}\right)e^{-\langle\theta,X\rangle}+\left(\frac{1-\rho}{2}\right)e^{\langle\theta,X\rangle}\right]\nonumber \\
 & =e^{-\|x\|^{2}/2}\cdot\varphi(x)\cdot\cosh\left(\langle\theta,X\rangle+\beta_{\rho}\right)\,,\label{eq: density of observed samples}
\end{align}
where $\varphi(x)\dfn\frac{1}{\sqrt{2\pi}}e^{-\|x\|^{2}/2}$ is the
standard normal density in $\mathbb{R}^{d}$. Similarly, the full
observation, which also includes the latent sign $s$ (\ref{eq: Guassian mixture model - random variables})
is given by a standard Gaussian density 
\[
p_{\theta,\rho}(s,x)=\left(\frac{1+s\rho}{2}\right)\varphi(x-s\theta)\,.
\]
Assume that $\underline{X}=\underline{x}$ is given and the EM algorithm
has ran up to its $t$th iteration, and so $(\theta_{t},\rho_{t})$
is given. The next iteration of the EM algorithm is the pair $(\theta_{t+1},\rho_{t+1})$
which maximizes the following $Q$-function: 
\[
Q(\theta,\rho\mid\theta_{t},\rho_{t})\dfn\sum_{\underline{s}\in\{\pm1\}^{n}}p_{\theta_{t},\rho_{t}}(\underline{s}\mid\underline{X})\log p_{\theta,\rho}(\underline{s},\underline{X})\,.
\]
Using the i.i.d. property of $\underline{X}$, and the expression
(\ref{eq: density of observed samples}) for the density, this is
equivalent to 
\begin{align*}
(\theta_{t+1},\rho_{t+1}) & \in\argmin_{\rho}\sum_{i=1}^{n}\E_{\theta_{t},\rho_{t}}\left[\log\left(\frac{1+S_{i}\rho}{2}\right)\mid X_{i}=x_{i}\right]\\
 & +\argmin_{\theta}\left\{ n\|\theta\|^{2}-\left\langle \theta,\sum_{i=1}^{n}x_{i}\E_{\theta_{t},\rho_{t}}\left[S_{i}\mid X_{i}=x_{i}\right]\right\rangle \right\} \,,
\end{align*}
where $S_{i}\in\{\pm1\}$ for $i\in[n]$ with $\P[S_{i}=1]=(1+\rho_{t})/2$,
and are i.i.d.. Hence, given $(\theta_{t},\rho_{t})$, the optimization
over $(\theta,\rho)$ is decoupled, and its solution is given by the
pair 
\[
\theta_{t+1}=\frac{1}{n}\sum_{i=1}^{n}x_{i}\cdot\E_{\theta_{t},\rho_{t}}\left[S_{i}\mid X_{i}=x_{i}\right]\,,\qquad\rho_{t+1}=\frac{1}{n}\sum_{i=1}^{n}\E_{\theta_{t},\rho_{t}}\left[S_{i}\mid X_{i}=x_{i}\right]\,,
\]
where 
\begin{equation}
\E_{\theta_{t},\rho_{t}}\left[S\mid X=x\right]=\frac{(1+\rho)\cdot e^{\langle\theta,x\rangle}-(1-\rho)\cdot e^{-\langle\theta,x\rangle}}{(1+\rho)\cdot e^{\langle\theta,x\rangle}+(1-\rho)\cdot e^{-\langle\theta,x\rangle}}=\tanh(\langle\theta,x\rangle+\beta_{\rho})\,.\label{eq: expected value of the sign given sample}
\end{equation}
Hence the EM iteration $\{\theta_{t,}\rho_{t}\}_{t=1}^{\infty}$ of
the symmetric 2-GM model evolves according to 
\begin{align}
\theta_{t+1} & =f_{n}(\theta_{t},\rho_{t}\mid\theta_{*},\rho_{*})\label{eq: EM iteration mean}\\
\rho_{t+1} & =h_{n}(\rho_{t},\theta_{t}\mid\theta_{*},\rho_{*})\,,\label{eq: EM iteration weight}
\end{align}
where the \emph{sample mean EM iteration} is 
\begin{equation}
f_{n}(\theta,\rho\mid\theta_{*},\rho_{*})=\E_{n}\left[X\cdot\frac{(1+\rho)\cdot e^{\langle\theta,X\rangle}-(1-\rho)\cdot e^{-\langle\theta,X\rangle}}{(1+\rho)\cdot e^{\langle\theta,X\rangle}+(1-\rho)\cdot e^{-\langle\theta,X\rangle}}\right]=\E_{n}\left[X\cdot\tanh\left(\langle\theta,X\rangle+\beta_{\rho}\right)\right]\,,\label{eq: sample mean iteration}
\end{equation}
and the \emph{sample weight EM iteration} is 
\begin{equation}
h_{n}(\rho,\theta\mid\theta_{*},\rho_{*})=\E_{n}\left[\frac{(1+\rho)e^{\langle\theta,X\rangle}-(1-\rho)e^{-\langle\theta,X\rangle}}{(1+\rho)e^{\langle\theta,X\rangle}+(1-\rho)e^{-\langle\theta,X\rangle}}\right]=\E_{n}\left[\tanh\left(\langle\theta,X\rangle+\beta_{\rho}\right)\right]\,.\label{eq: sample prior iteration}
\end{equation}
In the limit of $n\to\infty$, the iterations (\ref{eq: sample mean iteration})
and (\ref{eq: sample prior iteration}) tend, respectively, to the
\emph{population mean and population weight EM iterations} 
\[
f(\theta,\rho\mid\theta_{*},\rho_{*})=\E\left[X\cdot\tanh\left(\langle\theta,X\rangle+\beta_{\rho}\right)\right],\quad X\sim P_{\theta_{*},\rho_{*}}
\]
and 
\[
h(\rho,\theta\mid\theta_{*},\rho_{*})=\E\left[\tanh\left(\langle\theta,X\rangle+\beta_{\rho}\right)\right],\quad X\sim P_{\theta_{*},\rho_{*}}\,.
\]
We will usually omit $(\theta_{*},\rho_{*})$ from the notation for
the iteration, except when it is required to avoid confusion.

\paragraph{Statement of Results.}

The balanced case $\rho_{*}=0$ was analyzed in \cite{wu2019EM}: 
\begin{thm}[Theorems 1 and 2 in \cite{wu2019EM}]
\label{thm:balanced case}Assume that $\|\theta_{*}\|\leq\gl[C]_{\theta}$
and that $n\gtrsim d\log^{3}d$, and consider the balanced EM iteration
$\theta_{t+1}=f_{n}(\theta_{t},0\mid\theta_{*},0)$. There exists
$C_{0}>0$ such that if $\hat{u}$ is drawn uniformly from the unit
sphere $\mathbb{S}^{d-1}$, and the iteration is initialized with
$\theta_{0}=C_{0}\left(\frac{d\log n}{n}\right)^{1/4}\cdot\hat{u}$
then with probability $1-o_{n}(1)$ 
\begin{equation}
\ell_{0}(\theta_{*},\theta_{t})\lesssim\left(\frac{d\log^{3}n}{n}\right)^{1/4}\label{eq: balanced low signal guarantee}
\end{equation}
holds for all $t\gtrsim\sqrt{n}$. Furthermore, if $\|\theta_{*}\|\gtrsim\left(\frac{d\log^{3}n}{n}\right)^{1/4}$then
with probability $1-o_{n}(1)$ 
\[
\ell_{0}(\theta_{*},\theta_{t})\lesssim\frac{1}{\|\theta_{*}\|}\sqrt{\frac{d\log n}{n}}
\]
holds for all $t\gtrsim\frac{\log n}{\|\theta_{*}\|^{2}}$. The constants
involved in the asymptotic inequalities depend only on $\gl[C]_{\theta}$. 
\end{thm}

Our main result complements Theorem \ref{thm:balanced case} in the
unbalanced case, $\rho_{*}\neq0$: 
\begin{thm}[Simplified version of Theorem \ref{thm: empirical d>1 known delta large rho}]
\label{thm:Main result}Assume that $\|\theta_{*}\|\leq\gl[C]_{\theta}$
and that $|\rho_{*}|\leq\gl[C]_{\rho}$, as well as $n\gtrsim d\log n$.

If \textbf{$\rho_{*}\gtrsim\left(\frac{d\log n}{n}\right)^{1/4}$}
then the unbalanced EM iteration $\theta_{t+1}=f_{n}(\theta_{t},\rho_{*}\mid\theta_{*},\rho_{*})$
initialized with either $\theta_{0}=0$ or $\theta_{0}=\frac{1}{\rho_{*}}\E_{n}(X)$
satisfies that with probability $1-o_{n}(1)$ 
\[
\ell(\theta_{*},\theta_{t})\lesssim\frac{1}{\max\left\{ \rho_{*},\|\theta_{*}\|\right\} }\sqrt{\frac{d\log n}{n}}
\]
hold for all $t\geq\gl[T]$, where upper bounds on $\gl[T]$ are specified
in Table \ref{tab: Iterations until convergence}. The constants involved
in the asymptotic inequalities depend only on $(\gl[C]_{\theta},\gl[C]_{\rho})$.

If \textbf{$\rho_{*}\lesssim\left(\frac{d\log n}{n}\right)^{1/4}$}
the balanced EM iteration as in Theorem \ref{thm:balanced case} guarantees
(\ref{eq: balanced low signal guarantee}). If, in addition 
\begin{equation}
\|\theta_{*}\|\gtrsim\left(\frac{d\log n}{n}\right)^{1/4}\gtrsim\rho_{*}\gtrsim\frac{1}{\|\theta_{*}\|}\sqrt{\frac{d\log n}{n}}\label{eq: condition for sign amibguity}
\end{equation}
holds, then by setting $s_{t}=\sgn\langle\theta_{t},\E_{n}[X]\rangle$
it holds that 
\[
\ell(\theta_{*},s_{t}\cdot\theta_{t})\lesssim\frac{1}{\|\theta_{*}\|}\sqrt{\frac{d\log n}{n}}
\]
for all $t\gtrsim\frac{\log n}{\|\theta_{*}\|^{2}}$. 
\begin{table}
\begin{centering}
\begin{tabular}{|c|c|c|c|}
\hline 
 & $\|\theta_{*}\|\lesssim\frac{1}{\rho_{*}}\sqrt{\frac{d\log n}{n}}$  & $\frac{1}{\rho_{*}}\sqrt{\frac{d\log n}{n}}\lesssim\|\theta_{*}\|\lesssim\rho_{*}$  & $\rho_{*}\lesssim\|\theta_{*}\|$\tabularnewline
\hline 
\hline 
$\theta_{0}=0$  & $\gl[T]\lesssim1$  & $\gl[T]\lesssim\frac{1}{\rho_{*}^{2}}$  & $\gl[T]\lesssim\frac{1}{\rho_{*}\|\theta_{*}\|}$\tabularnewline
\hline 
$\theta_{0}=\frac{1}{\rho_{*}}\E_{n}(X)$  & $\gl[T]\lesssim1$  & $\gl[T]\lesssim1$  & $\gl[T]\lesssim\frac{1}{\|\theta_{*}\|^{2}}$\tabularnewline
\hline 
\end{tabular}
\par\end{centering}
\centering{}~\caption{$\gl[T]$:\label{tab: Iterations until convergence} Number of iterations
until convergence of unbalanced EM algorithm.\label{tab: Convergence times}}
\end{table}
\end{thm}

\paragraph{Interpretation of results.}

Note that in comparison to the balanced case $\rho_{*}=0$, the case
$\rho_{*}>0$ simplifies the analysis of the EM iteration in the sense
that the algorithm may be initialized at $\theta_{0}=0$ or at $\theta_{0}=\frac{1}{\rho_{*}}\E_{n}[X]$,
and no random initialization is required -- the expected value $\E[X]$
is proportional to $\theta_{*}$ and steers the iteration in the right
direction.

The convergence times specified in Table \ref{tab: Convergence times}
in case $\rho_{*}\gtrsim\left(\frac{d\log n}{n}\right)^{1/4}$ can
be interpreted as follows. While the EM iteration is $d$-dimensional,
it can be decomposed into movements in the signal direction (the direction
of $\theta_{*}$), and in its orthogonal direction \cite{pmlr-v65-daskalakis17b,wu2019EM}.
The factor dominating the number of iterations until convergence is
the time it takes the projected one-dimensional EM iteration in the
direction of $\theta_{*}$ to converge: 
\begin{itemize}
\item When $\|\theta_{*}\|\lesssim\frac{1}{\rho_{*}}\sqrt{\frac{d\log n}{n}}$
the signal is very low, and the EM estimate remains around $\theta_{*}$
for all iterations (for both types of initialization). 
\item When $\frac{1}{\rho_{*}}\sqrt{\frac{d\log n}{n}}\lesssim\|\theta_{*}\|\lesssim\rho_{*}$,
an error rate of $O(\frac{1}{\rho_{*}}\sqrt{\frac{d\log n}{n}})$
is achieved by $\theta_{0}=\frac{1}{\rho_{*}}\E_{n}(X)$ starting
from the first iteration (and the EM iterations remain at this area
of low statistical error). When $\theta_{0}=0$ the one-dimensional
EM iteration in direction of $\theta_{*}$ is contracting with slope
bounded by $1-c\rho_{*}^{2}$ for some $c>0$ and the convergence
time is $O(1/\rho_{*}^{2})$. 
\item When $\rho_{*}\lesssim\|\theta_{*}\|$, an error rate of $O(\frac{1}{\|\theta_{*}\|}\sqrt{\frac{d\log n}{n}})$
is achieved. For $\theta_{0}=\frac{1}{\rho_{*}}\E_{n}(X)$, it is
shown that the empirical iteration converges faster than the corresponding
balanced iteration starting from the first iteration. For $\theta_{0}=0$
the same effect occurs, but after an initial phase of additive increase
in $\theta_{t}$, and this early phase dominates the convergence time. 
\end{itemize}
Evidently, the worst convergence time of the balanced iteration is
also similar to the worst case convergence time of the unbalanced
iteration and given by $\tilde{O}(\sqrt{n})$, which is achieved when
$\rho_{*}\asymp\left(\frac{d\log n}{n}\right)^{1/4}$. We also remark
that as shown in \cite{pmlr-v65-daskalakis17b,wu2019EM}, the analysis
of the EM iteration in high dimension is possible when it is initialized
with a low norm, but not zero. For the unbalanced model, initializing
at $\theta_{0}=0$ is possible, and represents the longest convergence
time. It should be noted that the bounds on the convergence times
for $\theta_{0}=\frac{1}{\rho_{*}}\E_{n}(X)$ exhibit a discontinuity
at $\|\theta_{*}\|=\rho_{*}$. This is because $\gl[T]$ does not
capture the time required for convergence to a fixed point but rather
to a neighborhood around $\theta_{*}$ within the statistical error
rate.\footnote{For illustration, consider one-dimensional convergence, let the required
statistical accuracy be $\omega$, and suppose that $\theta_{0}=0$.
If $\theta_{*}\leq\omega$ then statistical accuracy is achieved already
in the first iteration, and then it is only need to be proved (and
also possible, as we shall show throughout) that the iteration remains
at this accuracy for all subsequent iterations. If, however, the order
of $\theta_{*}$ is increased, say $\theta_{*}=2\omega$, then the
iteration should increase, say, from $\theta_{0}=0$ to $\theta_{t}\geq\omega$
to achieve statistical accuracy, and the required number of iteration
for this increase depends on $\omega$. }\textbf{ }

The information-theoretic lower bounds obtained in \cite{wu2019EM}
for $\rho_{*}=0$ are generalized in Theorem \ref{thm: mean estimation minimax}
(Appendix \ref{sec:Minimax-rates}) and show that the error rate achieved
by EM in Theorem \ref{thm:Main result} equals the minimax error rates
(up to logarithmic factors) whenever $\rho_{*}\gtrsim\left(\frac{d\log n}{n}\right)^{1/4}$.
It switches from the minimax error rate $O(\frac{1}{\rho_{*}}\sqrt{d/n})$
assured for any signal strength to the local minimax error rates for
stronger signals $O(\frac{1}{\|\theta_{*}\|}\sqrt{d/n})$ at $\|\theta_{*}\|\asymp\rho_{*}$.
In the balanced case $\rho_{*}=0$, a similar switch occurs at $\|\theta_{*}\|\asymp\left(d/n\right)^{1/4}$,
improving from error rate of $O((d/n)^{1/4})$ to $O(\frac{1}{\|\theta_{*}\|}\sqrt{d/n})$.
This observation along with expected monotonicity of the error rates
in $\rho_{*}$ elucidates the condition $\rho_{*}\gtrsim\left(\frac{d\log n}{n}\right)^{1/4}$
in Theorem \ref{thm:Main result} (see the rigorous statement in Theorem
\ref{thm: mean estimation minimax}).

We complete the picture by discussing the case $\rho_{*}\lesssim\left(\frac{d\log n}{n}\right)^{1/4}$.
In this case, the minimax error rate analysis (Theorem \ref{thm: mean estimation minimax})
suggests that the error rates cannot be improved due to the unbalancedness
of the samples. However, the error rate of the balanced case can be
achieved for the $\ell_{0}$ loss function (which allows for sign
ambiguity), and when condition (\ref{eq: condition for sign amibguity})
holds, it can be achieved without sign ambiguity. The idea is simply
to use the balanced iteration which is insensitive to the actual signs
generating the samples $\underline{X}$, and upon convergence, evaluate
the angle between $\theta_{t}$ and $\E_{n}[X]$. With high probability,
this detects the correct sign required to estimate $\theta_{*}$ when
$\rho_{*}\gtrsim\frac{1}{\|\theta_{*}\|}\sqrt{\frac{d\log n}{n}}$.
If this condition fails then no correct decoding of the sign is possible,
as the signal is too low compared to the unbalancedness of the iteration
(cf. the minimax error rates of estimating $\rho$ when $\theta_{*}$
is known and $d$ is fixed of Theorem \ref{thm: weight estimation minimax}
in Appendix \ref{sec:Minimax-rates}).

We note in passing that we also analyze an EM iteration for estimating
$\rho_{*}$ given any fixed value of $\theta$ (perhaps mismatched
to $\theta_{*}$). As we will discuss in Section \ref{subsec:Joint-mean-and},
this shows that the given EM algorithm can be used for joint estimation
of $(\theta_{*},\rho_{*})$ if sufficient separation holds. Characterizing
the minimal separation required for \emph{joint} estimation remains
an open problem.

\paragraph{Significance of the unbalanced model. }
\begin{enumerate}
\item The likelihood-based EM has method-of-moments alternatives \cite{anandkumar2014tensor,heinrich2015optimal,wu2018optimal}
which may achieve the same error rates as the EM algorithm, perhaps
at a higher computational cost. Specifically, for the balanced 2-GM
model, the optimal error rate\footnote{Which, in fact, unlike EM, do not have ``spurious'' logarithmic
terms.} is achieved by a \emph{spectral} algorithm \cite{wu2019EM}. Such
an algorithm estimates $\theta_{*}$ by $\theta_{\text{SP}}=\sqrt{\max\{\lambda_{\max}-1,0\}}\cdot\hat{\theta}_{\text{sp}}$
where $\lambda_{\max}$ and $\hat{\theta}_{\text{sp}}$ are, respectively,
the maximal eigenvalue and the corresponding normalized maximal eigenvector
$\hat{\theta}_{\text{sp}}$, of the empirical covariance matrix $\E_{n}[XX^{T}]$.
The spectral algorithm can be interpreted as eliminating the sign
ambiguity by ``squaring'' the samples, since the covariance matrix
\[
\E[XX^{T}]=\theta_{*}\theta_{*}^{T}+I_{d}\,,
\]
does not depend on the unknown sign $S$ (cf. the model (\ref{eq: Gaussian mixture model})).
Hence, while EM attempts to \emph{learn} the latent signs, spectral
algorithms attempt to \emph{eliminate} them. Despite this conceptual
difference, it was observed in \cite{pmlr-v65-daskalakis17b} that
whenever $\|\theta_{t}\|$ has sufficiently low norm, the EM iteration
behaves as a power iteration on the empirical covariance matrix, and
in this regime the operation of EM is not fundamentally different
from a spectral algorithm. Nonetheless, sign elimination can only
be optimal for sufficiently small values of $\rho_{*}$, since the
distribution of the statistic $\E_{n}[XX^{T}]$ is insensitive to
the value of $\rho_{*}$, so it cannot lower its error in case $\rho_{*}>0$.
Our results thus demonstrate that EM is nearly optimal in a regime
in which the estimator must learn the latent signs. 
\item The worst case error over $\|\theta\|^{*}$ is given by $\max\{\left(d/n\right)^{1/4},\frac{1}{\rho_{*}}\sqrt{d/n}\}$
and improves as $\rho_{*}$ is increased. In practice, $\rho_{*}$
may be increased, e.g., by collecting additional information on the
latent signs generating $\lceil\tfrac{1}{2}\rho_{*}n\rceil$ of the
samples, and then align the signs of those samples by proper multiplication
by $\{\pm1\}$. As another example, consider a communication system
in which $(S_{1},\ldots,S_{n})\in\{\pm1\}^{n}$ are the input bits
to a noisy channel whose output at time $i$ is given by $X_{i}=\theta_{*}S_{i}+Z$,
as in (\ref{eq: Guassian mixture model - random variables}). In order
to decode the bits, a typical decoder will estimate $\theta_{*}$
as a preliminary step, and assume that the samples are i.i.d. .\footnote{Typically, the data bits are encoded using an error correcting code
before being sent over the channel, and so the bits $\{S_{i}\}$ are
not i.i.d.. Nonetheless, the receiver may ignore these dependencies
for the purpose of estimation. } The input distribution $\P[S_{i}=1]=(1-\rho_{*})/2$ then trades-off
between estimation and data rate, with best estimation and zero data
rate for $\rho_{*}=1$ v.s. maximal data rate and worst estimation
for $\rho_{*}=0$. 
\item The proofs of global convergence for the balanced 2-GM model ($\rho_{*}=0$)
\cite{xu2016global,pmlr-v65-daskalakis17b,wu2019EM} rely heavily
on global symmetry properties of the population iteration (see next,
Section \ref{subsec: proof methods}). This lack of symmetry is challenging
for proving global convergence. For example, we show that a stable
spurious fixed point is possible at some $\theta\in(-\theta_{*},0)$
. Nonetheless, we show that (essentially) global convergence to $\theta_{*}$
is not restricted to $\rho_{*}=0$. 
\end{enumerate}

\subsection{Discussion of proof ideas \label{subsec: proof methods}}

In order to give context for the proof ideas, we first consider the
balanced case $\rho_{*}=0$ and describe the ideas behind the results
of \cite{xu2016global,pmlr-v65-daskalakis17b,wu2019EM} and how they
compare with the general analysis of \cite{balakrishnan2017statistical}.
There are two main ideas -- one pertains to the population iteration
and the other to the empirical error.

For the population iteration, the convergence radius guaranteed in
\cite{balakrishnan2017statistical} is proved using a standard fixed-point
theorem which requires \emph{contractivity} of the iterative iteration.
The guarantee on the size of the basin of attraction is obtained from
a guarantee on the contractivity of $f(\theta)$ in this region. However,
global convergence cannot be established by such an argument since
the EM iteration for (\ref{eq: Gaussian mixture model}) with $\rho_{*}=0$
is in fact \emph{not} globally contractive. Nonetheless, contractivity
is only a sufficient, but not necessary condition for convergence,
and other global properties of the iteration may be used. For example,
in the one-dimensional case $d=1$, the balanced EM iteration has
two stable fixed points $\theta=\pm\theta_{*}$, due to the well known
consistency property of EM (both which are acceptable solutions with
$\ell_{0}(\theta,\theta_{*})=0$), and a single unstable fixed point
$\theta=0$. The fact that any other fixed point is impossible follows
from the observation that $f(\theta)$ is an odd function, which is
concave for $\theta\in\mathbb{R}_{+}$ \cite{wu2019EM}. By contrast,
in the unbalanced case ($\rho_{*}>0$), neither concavity (say, for
all $\theta\in\mathbb{R}_{+}$) nor global contractivity hold for
unbalanced iterations. It is also seems to be difficult to analytically
characterize the required distance of $\theta$ from $\theta_{*}$
for these properties to hold.

For the empirical iteration, the error guarantee of \cite{balakrishnan2017statistical}
is obtained from the following high probability uniform error bound
on the empirical error 
\begin{equation}
\sup_{\theta\colon\|\theta_{0}-\theta_{*}\|\leq\frac{1}{4}\|\theta_{*}\|}\|f_{n}(\theta)-f(\theta)\|=\tilde{O}\left(\sqrt{\frac{d}{n}}\right)\,.\label{eq: balanced weak error bound}
\end{equation}
However, it was observed in \cite{dwivedi2018singularity} and \cite{wu2019EM}
that a stronger bound on the error can be obtained which allows arbitrarily
small $\|\theta_{*}\|$ and $\|\theta\|$ by ``localizing'' the
error as follows: 
\begin{equation}
\sup_{\theta\colon\|\theta\|\leq\gl[C]_{\theta}}\|f_{n}(\theta)-f(\theta)\|=\|\theta\|\cdot\tilde{O}\left(\sqrt{\frac{d}{n}}\right)\,.\label{eq: balanced localized error bound}
\end{equation}
So, while the empirical iteration analyzed using (\ref{eq: balanced weak error bound})
requires strong separation $\|\theta_{*}\|=\Omega(1)$, no such condition
is required when the bound (\ref{eq: balanced localized error bound})
is utilized, leading to the sharp results of \cite{wu2019EM}.

The analysis of the unbalanced case $\rho_{*}\neq0$ in this paper
is based on the following intuitive idea of $\rho$\emph{-ordering
of iterations}, which allows a comparison with the $\rho_{*}=0$ case.
If $\rho_{*}=1$, the model (\ref{eq: Gaussian mixture model}) is
the Gaussian location model, for which it can be easily verified (see
Section \ref{subsec:Main-result}) that the EM iteration converges
in a single iteration to the sample mean (which is also the MLE).
Extrapolating from this extreme case, we might expect that if $\rho_{1}>\rho_{0}$
then the iteration for $\rho_{1}$ will converge faster since the
model more closely resembles the Gaussian location model. We state
global comparison results (Theorem \ref{thm: population d=00003D00003D1 known delta}
for $d=1$ and Proposition \ref{prop: Comparison theorem d>1} for
$d>1$) establishing this property for any arbitrary pair $\rho_{0},\rho_{1}\in[0,1]$.
Combining this property with the known global convergence rate of
the balanced case $\rho_{*}=0$ yields the global convergence proof
of the population iteration for unbalanced $\rho_{*}\neq0$.

For the empirical iteration, it turns out that increasing $\rho$
has an \emph{opposite} effect. We generalize the localized error bound
of \cite{wu2019EM} in (\ref{eq: balanced localized error bound})
from $\rho_{*}=0$ to a general $\rho_{*}\in[0,1]$ and obtain that
\begin{equation}
\sup_{\theta\colon\|\theta\|\leq C}\|f_{n}(\theta)-f(\theta)\|=\max\left\{ \|\theta\|,\rho_{*}\right\} \cdot\tilde{O}\left(\sqrt{\frac{d}{n}}\right)\,,\label{eq: unbalanced localized error bound}
\end{equation}
indicating that the empirical error increases with $\rho_{*}$. The
main challenge of the analysis of the empirical iteration is to prove
that the increased empirical error for larger $\rho_{*}$ is compensated
by the improved convergence rate of the population iteration. It should
be noted, however, that the empirical error may break key properties
of the population iteration. For example, for $d=1$, the convergence
of the population iteration for $\theta_{0}=0$ towards $\theta_{*}$
is based on the fact that $f(0)>0$ (assuming w.l.o.g. that $\theta_{*}>0$).
Clearly, the empirical error (\ref{eq: unbalanced localized error bound})
might result in $f_{n}(0)<0$ which would steer the iteration towards
a spurious fixed point in $\mathbb{R}_{-}$. Our analysis shows that
with high probability this occurs only if $\|\theta_{*}\|$ is low,
so that this bad convergence does not dominate the error rate.

\subsection{General background on the EM algorithm}

In this section, we briefly outline relevant background on the EM
algorithm. It is well known that it is typically computationally complex
to compute the MLE 
\[
\theta_{\text{MLE}}=\argmax_{\theta}\E_{n}\left[\log P_{\theta}(X)\right]
\]
in parametric models $(X,S)\sim P_{\theta}(x,s)$ for which only $X$
is observed but $S$ is latent. For one thing, exact marginalization
over the latent variables $S$ to obtain the likelihood $P_{\theta}(x)$
(or its gradient) is computationally heavy due the need to sum over
all possible configurations of the latent variable. Moreover, in most
interesting cases, the likelihood $P_{\theta}(x)$ is not a concave
function of $\theta$, and so standard optimization techniques do
not have strong guarantees. Various authors \cite{baum1970maximization,beale1975missing,hartley1958maximum,healy1956missing,sundberg1974maximum,woodbury1970missing,hasselblad1966estimation,hasselblad1969estimation}
have independently proposed several heuristics akin to the EM algorithm
for this problem, and the EM algorithm was later on formulated in
its well known form in the seminal paper \cite{dempster1977maximum},
which also proposed a wide range of statistical applications.

The EM is an iterative procedure, which determines an empirical operator
$f_{n}$ based on $n$ samples from the data $X\sim P_{\theta}$.
Given an initial guess $\theta_{0}$, the algorithm produces a sequence
of iterations $\theta_{t}=f_{n}(\theta_{t-1})$ for all $t\geq1$.
Owing to its name, the empirical operator is determined by solving
two steps. The first step computes a posterior probability $P_{\theta_{t}}(S\mid X)$
on the latent variable $S$ based on the current estimate $\theta_{t}$,
and then averages the log-likelihood with this posterior (``expectation'')
to obtain the $Q$-function 
\[
Q(\theta\mid\theta_{t})=\int p_{\theta_{t}}(s\mid X)\cdot\log p_{\theta}(X,s)\cdot\d s\,.
\]
The second step then sets $\theta_{t+1}=f_{n}(\theta_{t})\dfn\argmax_{\theta}Q(\theta\mid\theta_{t})$
(``maximization''). In many practical cases, the last maximization
step can be solved analytically and an explicit expression of the
operator $f_{n}(\cdot)$ is available. A different interpretation
of EM as a minorization-maximization algorithm is obtained from the
fact that the bound 
\[
\log P_{\theta}(x)-\log P_{\theta_{t}}(x)\geq Q(\theta\mid\theta_{t})-Q(\theta_{t}\mid\theta_{t})
\]
holds for any $\theta$, which immediately implies a strong general
property: The EM algorithm produces increasing likelihoods $P_{\theta_{t}}(x)$
as $t$ increases. This elegant property, along with its typically
low computational complexity has contributed to its widespread application
in numerous applications~\cite{gupta2011theory}.

Despite the above appealing properties, not long after its formulation
in \cite{dempster1977maximum}, it was recognized that the EM algorithm
may actually fail to compute the MLE. In \cite{wu1983convergence},
it was clarified that in the general case, the EM algorithm may converge
to \emph{local} maxima of the likelihood, or even get trapped in a
saddle point. Clearly, such local maxima may be far from the required
MLE, and in high dimension their number could be exponentially large.
Consequently, except in favorable cases in which the likelihood is
unimodal, the convergence of the EM algorithm heavily depends on the
initial guess. In practice, this necessitates complicated initialization
algorithms such as multiple restarts with random initial estimates\cite{karlis2003choosing},
or using a pilot estimator to obtain an initial guess. Both options
are typically costly. In the more restricted case of mixtures of exponential
families, \cite{redner1984mixture} showed that EM converges at a
geometric rate to the MLE, under positivity conditions of the Fisher
information matrix and the mixing weights, and more importantly, assuming
local initialization. However, the dependence of the guarantees on
the convergence radius and rate are only qualitative and do not specify
their dependence on the parameters of the model. Furthermore, it was
empirically observed in \cite{redner1984mixture} that the EM iterations
can become painfully slow to converge whenever the separation between
the components is low.

Later works, e.g., \cite{hero1995convergence,meng1994global,chretien2008algorithms}
displayed similar guarantees, albeit to a local maxima of the likelihood,
which, naturally, might be far from the true likelihood. The paper
\cite{xu1996convergence} has cast the EM algorithm for Gaussian mixtures
as a gradient ascent algorithm, where in each step the gradient is
pre-multiplied by a positive-definite matrix, and exemplified slow
convergence akin to first-order optimization methods. These drawbacks
of EM were then addressed by a multitude of ad hoc methods and variants,
comprehensively summarized in \cite{mclachlan2007algorithm}. The
bottom line however, that even if the MLE is known to have good statistical
properties, it is not clear weather they can be computationally achieved
by the EM algorithm.

The apparent discrepancy between the wide practicality of the EM algorithm
versus its relatively weak theoretical guarantees mentioned above,
along with the growth in size and dimension of modern data sets, resulted
in two paradigm shifts in the anticipated goals expected from its
analysis. The first one, most notably emerging in \cite{balakrishnan2017statistical},
is the explicit characterization of the statistical precision, convergence
rate, and the distance of the initialization from the ground truth
required to obtain that statistically accurate solution (``basin
of attraction''). The characterization in \cite{balakrishnan2017statistical}
is based on general smoothness and stability properties of the auxiliary
function $Q(\theta\mid\theta')$, which need to be verified independently
for any given problem. As concrete examples, these conditions were
applied in \cite{balakrishnan2017statistical} to canonical models
such as the balanced 2-GM, symmetric mixture of two regressions, and
linear regression with missing covariates. Nonetheless, as discussed
in Section \ref{subsec:EM for Gaussian mixture}, this approach, even
when combined with further refinements \cite{klusowski2016statistical,Wu2016OnTC},
did not lead to sharp results for the basic balanced 2-GM model. As
discussed in Section \ref{subsec:Main-result}, the local convergence
result of the 2-GM model was then improved to global convergence guarantees
by various authors. For the idealized population version, it was shown
in \cite{xu2016global,pmlr-v65-daskalakis17b} that EM converges at
a geometric rate to $\pm\theta_{*}$, unless the initial guess $\theta_{0}$
is orthogonal to $\theta_{*}$. A finite sample analysis was made
in \cite{pmlr-v65-daskalakis17b}, but was based on sample-splitting
-- EM was assumed to run on a fresh batch of samples at each iteration.
Optimality of EM in terms of statistical error and convergence time
was ultimately established in~\cite{wu2019EM}.

\subsection{Other known results\label{subsec:Related-work}}

The unbalanced 2-GM model studied in this paper was mostly explored
in relation to misspecification or overspecification, i.e., cases
in which the true model does not belong to the set of fitted models,
or belongs to a simpler set of models. An extreme case of 2-GM mixture
model was considered in \cite{dwivedi2018singularity}, in which the
components are not separated at all, thus reduced to a zero-mean Gaussian
$\theta_{*}=0$. The EM algorithm was designed to operate on the unbalanced
model (\ref{eq: Gaussian mixture model}) with $\rho_{*}\neq0$ that
over-fits the true model. For this case, it was shown that the population
iteration is globally contracting at a rate $\|\theta_{t+1}\|\asymp\|\theta_{t}\|(1-\frac{\rho_{*}^{2}}{2})$
and thus globally converging at a geometric rate, and has a statistical
error of $O(\frac{1}{\rho_{*}^{2}}\sqrt{\frac{d}{n}})$, which is
parametric for fixed $\rho$, but in general, worse than the minimax
rate $(\frac{1}{\rho_{*}}\sqrt{\frac{d}{n}})$ ), and from our Theorem
\ref{thm:Main result}. This behavior was contrasted with the same
setting, except for which $\rho_{*}=0$, where it was shown that convergence
of the population EM is much slower, and behaves as $\|\theta_{t+1}\|\asymp\|\theta_{t}\|\left(1-\|\theta_{t}\|^{2}\right)$,
and the error rate for the sample-based EM is $O((d/n)^{1/4})$. This
error rate was achieved by partitioning the EM iterations to multi-epochs,
where in the $l$th epoch, $\|\theta_{t}\|\in[\left(\frac{d}{n}\right)^{\alpha_{l+1}},\left(\frac{d}{n}\right)^{\alpha_{l}}]$
for judiciously chosen powers $\alpha_{l}$. With this approach, the
guarantees on the empirical error in the iterations of the $l$th
epoch improve as $l\uparrow\infty$, which allows the ``localization''
of the empirical error discussed in Section \ref{subsec: proof methods}.
In \cite{dwivedi2019challenges}, EM for 2-GM mixture model with $\theta_{*}=0$
was considered again, but in which the algorithm is also allowed to
fit the variance of the samples. The obtained behavior is distinctively
different in one and multiple dimensions. For $d\geq2$, the number
of required iterations is $O(\sqrt{d/n})$, and the error rate for
estimating the mean is $O((d/n)^{1/4})$, whereas for $d=1$ the number
of required iterations is even larger $O(n^{3/4})$, and so is the
error rate $O((1/n)^{1/8})$. Other misspecified models were considered
in \cite{dwivedi2018theoretical}, and one of them is an unbalanced
2-GM one-dimensional mixture to a balanced 2-GM one-dimensional mixture,
albeit with a smaller, unknown variance. The paper bounded the distance
between the true parameter and the parameter corresponding to the
KL projection of the true model onto the set of allowed models. Based
on this bound, the population EM operator was shown be contractive
w.r.t. the projected parameter, and geometric convergence with statistical
error rate $O(1/\sqrt{n})$ of the samples-based EM iteration was
established.

Following the general analysis of \cite{balakrishnan2017statistical},
various latent models were explored. A high-dimensional setting with
$d\geq n$ and sparsity assumptions was studied in \cite{wang2014high,yi2015regularized},
which proposed truncation and regularization approaches for modifying
EM to that setting, and provided results comparable to \cite{balakrishnan2017statistical}.
The problem of estimating mixtures of linear regressions was considered
in \cite{klusowski2019estimating} which enlarged the contraction
region assured in \cite{balakrishnan2017statistical} for this case,
and showed that any initial guess with sufficiently large angle with
the target parameter vector (rather than a small distance in \cite{balakrishnan2017statistical})
will converge to $\theta_{*}$. It also showed that a sample-splitting
version of the EM algorithm converges with high probability. Global
convergence of the sample EM iteration was later established in \cite{kwon2018global}
by controlling both the empirical error and empirical angle between
the population and empirical iterations. Results of this nature were
then generalized to the $k$ mixture of linear regression in \cite{klusowski2019estimating}.
The $k$-GM for a general $k\geq2$ was studied in \cite{yan2017convergence,zhao2018statistical},
which provided results comparable to \cite{balakrishnan2017statistical}
for gradient EM under minimal separation condition between the means,
and closeness of the initial guess to the true means. Beyond the i.i.d.
setting, estimation problems in hidden Markov models using EM were
studied in \cite{yang2015statistical,aiylam2018parameter}.

\subsection{Notational conventions \label{subsec:Notation-conventions}}

Constant values which are used to state results or used in more than
a single place in the paper are denoted by sans-serif letters and
are summarized in Table \ref{tab:Summary-of-global}. Constants which
are used only locally are denoted by $c,C,c_{0},\ldots$. Those constants
are either universal or depend only on the parameters of the global
assumptions $\gl[C]_{\theta}$ and $\gl[C]_{\rho}$. Asymptotic relations
such as $\lesssim,\asymp$ are within these constant factors. The
expectation of a random variable $U$ is denoted by $\E[U]$, and
the empirical mean of $n$ i.i.d. samples $(U_{1},\ldots,U_{n})$
of $U$ is denoted by $\E_{n}[(U)]\dfn\frac{1}{n}\sum_{i=1}^{n}U_{i}$.
The distribution (law) of a random variable $U$ will be denoted by
${\cal L}(U)$. The $1$-Wasserstein distance between probability
measures $\mu$ and $\nu$ is given by \cite{villani2003topics} $W_{1}(V,U)=\inf\E|V-U|$
where the infimum is over all couplings of $\mu$ and $\nu$, i.e.,
a pair of random variables $(V,U)$ such that ${\cal L}(V)=\mu$ and
${\cal L}(U)=\nu$. The Euclidean norm is denoted by $\|\cdot\|$,
and the Euclidean ball of radius $r$ in dimension $d$ is denoted
by $\mathbb{B}^{d}(r)$, where for $d=1$ we omit the superscript.
A unit vector in the direction of a vector $\theta$ is denoted by
$\hat{\theta}$. For a given Orlicz function $\psi$, the Orlicz norm
of a random variable $U$ is denoted by $\|U\|_{\psi}=\inf_{t>0}\left\{ \E[\psi(|U|/t)]\leq1\right\} $,
where $U$ is called $\sigma^{2}$-sub-gaussian (resp. $\sigma$-sub-exponential)
if $\|U\|_{\psi_{2}}\leq\sigma$ where $\psi_{2}(t)=\exp(t^{2})-1$=
(resp. $\|U\|_{\psi_{1}}\leq\sigma$ where $\psi_{1}(t)=\exp(t)-1$)
. The set $\{1,\ldots,n\}$ is denoted by $[n]$, equivalence (usually
local simplification of notation) is denoted by $\equiv$, and equality
in distribution by $\eqd$. 
\begin{table}
\centering{}{\footnotesize{}}%
\begin{tabular}{|c|c|}
\hline 
{\footnotesize{}Constant } & {\footnotesize{}Description}\tabularnewline
\hline 
\hline 
{\footnotesize{}$\gl[C]_{\theta}$ } & {\footnotesize{}Global assumption (Section \ref{subsec:Main-result}):
Maximal norm of $\theta_{*}$}\tabularnewline
\hline 
{\footnotesize{}$\gl[C]_{\rho}$ } & {\footnotesize{}Global assumption (Section \ref{subsec:Main-result}):
Maximal absolute value of $\rho_{*}$}\tabularnewline
\hline 
{\footnotesize{}$\gl[C]_{\beta}$ } & {\footnotesize{}Global assumption (Section \ref{subsec:Main-result}):
Maximal absolute value of true inverse temperature parameter $\beta_{\rho_{*}}$}\tabularnewline
\hline 
{\footnotesize{}$(\gl[\underline{C}]_{\beta},\gl[\overline{C}]_{\beta})$ } & {\footnotesize{}Global assumption (Section \ref{subsec:Main-result}):
$\gl[\underline{C}]_{\beta}\rho\leq|\beta_{\rho}|\leq\gl[\overline{C}]_{\beta}\rho.$}\tabularnewline
\hline 
{\footnotesize{}$\gl[C]_{\omega}$ } & {\footnotesize{}Concentration (Section \ref{subsec:Concentration-of-the}):
Constant for empirical iteration error (w.h.p.)}\tabularnewline
\hline 
{\footnotesize{}$\{\gl[C]_{i}^{(1)}\}$ } & {\footnotesize{}Result for $d=1$ mean iteration (Section \ref{subsec:Concentration-of-the}):
Constants in Theorem \ref{thm: empirical d=00003D00003D1 known delta}}\tabularnewline
\hline 
{\footnotesize{}$\gl[T]^{(1)}$ } & {\footnotesize{}Result for $d=1$ mean iteration (Section \ref{subsec:Concentration-of-the}):
Convergence time in Theorem \ref{thm: empirical d=00003D00003D1 known delta}}\tabularnewline
\hline 
{\footnotesize{}$\gl[C]''$ } & {\footnotesize{}Proof of $d=1$ mean iteration (Lemma \ref{lem: d=00003D00003D1 simple properties}):
A constant for a bound on the second derivative of $f(\theta)$}\tabularnewline
\hline 
{\footnotesize{}$\{\gl[C]_{i}^{(d)}\}$ } & {\footnotesize{}Result for $d>1$ mean iteration (Section \ref{subsec:The-mean-iteration d>1}):
Constants in Theorem \ref{thm: empirical d>1 known delta large rho}}\tabularnewline
\hline 
{\footnotesize{}$\gl[T]_{\theta_{0}}^{(d)}$, $\gl[T]_{G}^{(d)}$ } & {\footnotesize{}Result for $d>1$ mean iteration (Section \ref{subsec:The-mean-iteration d>1}):
Convergence times in Theorem \ref{thm: empirical d>1 known delta large rho}}\tabularnewline
\hline 
{\footnotesize{}$\gl[C]_{F,0},\gl[C]''_{F},\gl[C]'''_{F}$ } & {\footnotesize{}Proof of $d>1$ mean iteration (Lemma \ref{lem: properties of F and G}):
Constants for bounds on $F(a,b)$ and its derivatives }\tabularnewline
\hline 
{\footnotesize{}$\gl[C]_{G,\rho}^{(d)},\gl[C]_{G,\eta}^{(d)}$ } & {\footnotesize{}Proof of $d>1$ mean iteration (Proposition \ref{prop: Comparison theorem d>1}):
Constants for bounds on $G(a,b)$ and its derivatives }\tabularnewline
\hline 
{\footnotesize{}$\gl[C]_{\eta}^{(d)}$ } & {\footnotesize{}Proof of $d>1$ mean iteration (Proposition \ref{prop: empirical d>1 known delta small rho}):
A constant for a condition on $\eta=\|\theta_{*}\|$}\tabularnewline
\hline 
{\footnotesize{}$\{\gl[C]_{i}^{(\rho)}\}$ } & {\footnotesize{}Result for weight iteration (Section \ref{subsec:The-weight-iteration}):
Constants in Theorem \ref{thm: population weight fixed mean}}\tabularnewline
\hline 
{\footnotesize{}$\gl[T]^{(\rho)}$ } & {\footnotesize{}Result for weight iteration (Section \ref{subsec:The-weight-iteration}):
Convergence time in Theorem \ref{thm: population weight fixed mean}}\tabularnewline
\hline 
{\footnotesize{}$\gl[C]''_{h}$ } & {\footnotesize{}Proof for weight iteration (Lemma \ref{lem: weight simple properties}):
A constant for a bound on the second derivative of $h(\rho)$}\tabularnewline
\hline 
\end{tabular}{\footnotesize{}~\caption{Summary of global constants\label{tab:Summary-of-global}}
}
\end{table}

\subsection{Organization}

Section \ref{sec:Detailed-Results} contains detailed statements of
the results, along with discussions, and proof outlines. Proofs appear
in later sections according to order. Specifically: 
\begin{itemize}
\item In Section \ref{subsec:Concentration-of-the} we generalize the uniform
error concentration bounds of \cite{wu2019EM} to the unbalanced case,
and also states such a bound for the weight iteration. The proof is
not fundamentally different from \cite{wu2019EM} and is provided
in Appendix \ref{sec: concentration proof} for completeness. 
\item In Section \ref{subsec:The-mean-iteration d=00003D00003D1} we analyze
the mean population and empirical EM iterations assuming the true
weight $\rho_{*}$ is known for $d=1$. 
\item In Section \ref{subsec:The-mean-iteration d=00003D00003D1} we extend
the analysis of the previous section to $d>1$, and prove the main
result of the paper. 
\item In Section \ref{subsec:The-mean-iteration d=00003D00003D1} we analyze
the population and empirical EM iterations for the weight assuming
a fixed mean $\theta$ (possibly mismatched to $\theta_{*}$). 
\item In Section \ref{subsec:Joint-mean-and} we briefly discuss the problem
in which both $\rho_{*}$ and $\theta_{*}$ are unknown, and the estimator
is required to jointly estimate both. 
\end{itemize}
In Appendix \ref{sec:Minimax-rates} we analyze minimax rates, and
in Appendix \ref{sec:Miscellaneous} we provide miscellaneous results
used in the paper.

\section{Detailed results\label{sec:Detailed-Results}}

\subsection{Concentration of the empirical EM iteration \label{subsec:Concentration-of-the}}

We first establish the concentration properties of the empirical iterations
to their population versions. The following Theorem is a generalization
of \cite[Theorem 4]{wu2019EM} from the $\rho=0$ case to $\rho\neq0$
case, and is proved in Appendix \ref{sec: concentration proof}. 
\begin{thm}
\label{thm: Error concentration}Assume that that $\|\theta_{*}\|\leq\gl[C]_{\theta}$
and that $|\rho_{*}|\leq\gl[C]_{\rho}$, and consider the event 
\begin{equation}
{\cal E}\dfn\left\{ \|f_{n}(\theta,\rho)-f(\theta,\rho)\|\leq\max\{\|\theta\|,\rho\}\cdot\omega_{d}\right\} \cap\left\{ \left|h_{n}(\rho,\theta)-h(\rho,\theta)\right|\leq\|\theta\|\cdot\omega_{1}\right\} \label{eq: high probability event}
\end{equation}
where 
\[
\omega_{d}\dfn\sqrt{\gl[C]_{\omega}\frac{d\log n}{n}}\,.
\]
Then, there exist a constant $\gl[C]_{\omega}$ which depends on $(\gl[C]_{\theta},\gl[C]_{\rho})$
such that $\P[{\cal E}]\geq1-\frac{1}{n^{cd}}$ for all $n\geq Cd\log n$. 
\end{thm}

We assume in the rest of the paper that the high probability event
(\ref{eq: high probability event}) holds, and often denote $\omega_{d}$
by $\omega$ for brevity. Note that in general, the error bound depends
on the iteration values $(\theta,\rho)$ and is uniform in the ground
truth parameters $(\theta_{*},\rho_{*})$ (as long as they satisfy
the global assumptions). For the mean iteration, it is interesting
to contrast the balanced iteration of $\rho=0$ with $\rho\neq0$.
For the balanced iteration, $f_{n}(\theta,0)=\E_{n}[X\cdot\tanh\langle\theta,X\rangle]$
and so $f_{n}(0,0)=f(0,0)=0\in\mathbb{R}^{d}$ with probability $1$.
Hence, a valid upper bound on the empirical error may tend to zero
as $\|\theta\|\to0$, and, indeed, \cite[Theorem 4]{wu2019EM} has
obtained an empirical error bound of order $O_{P}(\|\theta\|\omega)$.
Similar intuition was used in \cite{dwivedi2018singularity,dwivedi2019challenges}
to ``localize'' the error around $\|\theta\|\approx0$, although
in a more granular way. When the iteration is unbalanced, i.e., $\rho\neq0$,
the iteration $f_{n}(0,\rho)=\rho\cdot\E_{n}[X]$ is a non-degenerate
random variable, whose population version is $f(0,\rho)=\rho\cdot\E[X]=\rho^{2}\cdot\|\theta_{*}\|$.
In fact, in one-dimension, $f_{n}(0,\rho)$ might even be negative
(i.e., have opposite sign to its population version). Hence, one cannot
expect the empirical error to behave as in the balanced case, and
an additional term is required, which depends on $\rho$, as given
in (\ref{eq: high probability event}). Evidently, as $\rho$ increases,
so is the error bound. Intuition again may arise from the extreme
case, this time when $\rho=1$. In this case the iteration is simply
$\lim_{\rho\to1}f_{n}(\theta,\rho)=\E_{n}[X]$, i.e., the iteration
provides the empirical mean at a single step, which clearly must have
an empirical error of $O_{P}(\sqrt{\frac{d}{n}})$, even with $\|\theta\|$
being arbitrarily small. Figuratively speaking, the more the iteration
is ``aggressive'' in assuming prior knowledge regarding the signs
$\{S_{i}\}$ generating the samples, the larger is the empirical error
in the iteration. On the other hand, as we shall see, such (correct)
prior knowledge improves the convergence properties of the population
iteration. So, the convergence properties of the empirical iteration
for $\rho>0$ are obtained from a balance between improved population
convergence compared to $\rho=0$ which compensate for the larger
empirical error. The error in the weight iteration is proportional
to $O_{P}(\|\theta\|\cdot\omega_{1})$, which agrees with the observations
that $h_{n}(\theta,\rho)\to\rho$ as $\|\theta\|\to0$, and $\rho$
is unidentifiable, along with the observation that the weight iteration
is effectively one-dimensional, and so the error is proportional to
$\omega_{1}$ rather than to $\omega_{d}$.

\subsection{The mean iteration for known weight at $d=1$\label{subsec:The-mean-iteration d=00003D00003D1}}

In this section we consider the mean iteration in one dimension. While
the model for $d=1$ is simple, its analysis already captures some
of the complication of the analysis, and also serves as a building
block for the analysis of the $d>1$ case. We assume that $\delta_{*}=(1-\rho_{*})/2\equiv\delta$
is known and fixed, and this true parameter is used in the EM iteration.
Hence the model can be written as: 
\[
P_{\eta,\delta}=(1-\delta)\cdot N(\eta,1)+\delta\cdot N(-\eta,1)\,,
\]
where $\eta\dfn\|\theta_{*}\|>0$ is assumed w.l.o.g.. The population
version of the EM iteration for this case can be written as 
\[
f(\theta\mid\eta,\delta)\dfn\E\left[X\cdot\frac{(1-\delta)\cdot e^{X\theta}-\delta\cdot e^{-X\theta}}{(1-\delta)\cdot e^{X\theta}+\delta\cdot e^{-X\theta}}\right]=\E\left[X\cdot\tanh(X\theta+\beta)\right]\,,
\]
with $X\sim(1-\delta)N(\eta,1)+\delta N(-\eta,1)$ and where we abbreviate
to $f(\theta\mid\eta)$ or $f(\theta)$ when possible. Similarly,
the empirical iteration will be denoted by $f_{n}(\theta\mid\eta,\delta)$,
and abbreviated to $f_{n}(\theta)$. Fig. \ref{fig: unbalanced iteration}
illustrates several EM iterations (based on single runs of $n=10^{4}$
samples). 
\begin{figure}
\begin{centering}
\includegraphics[scale=0.8]{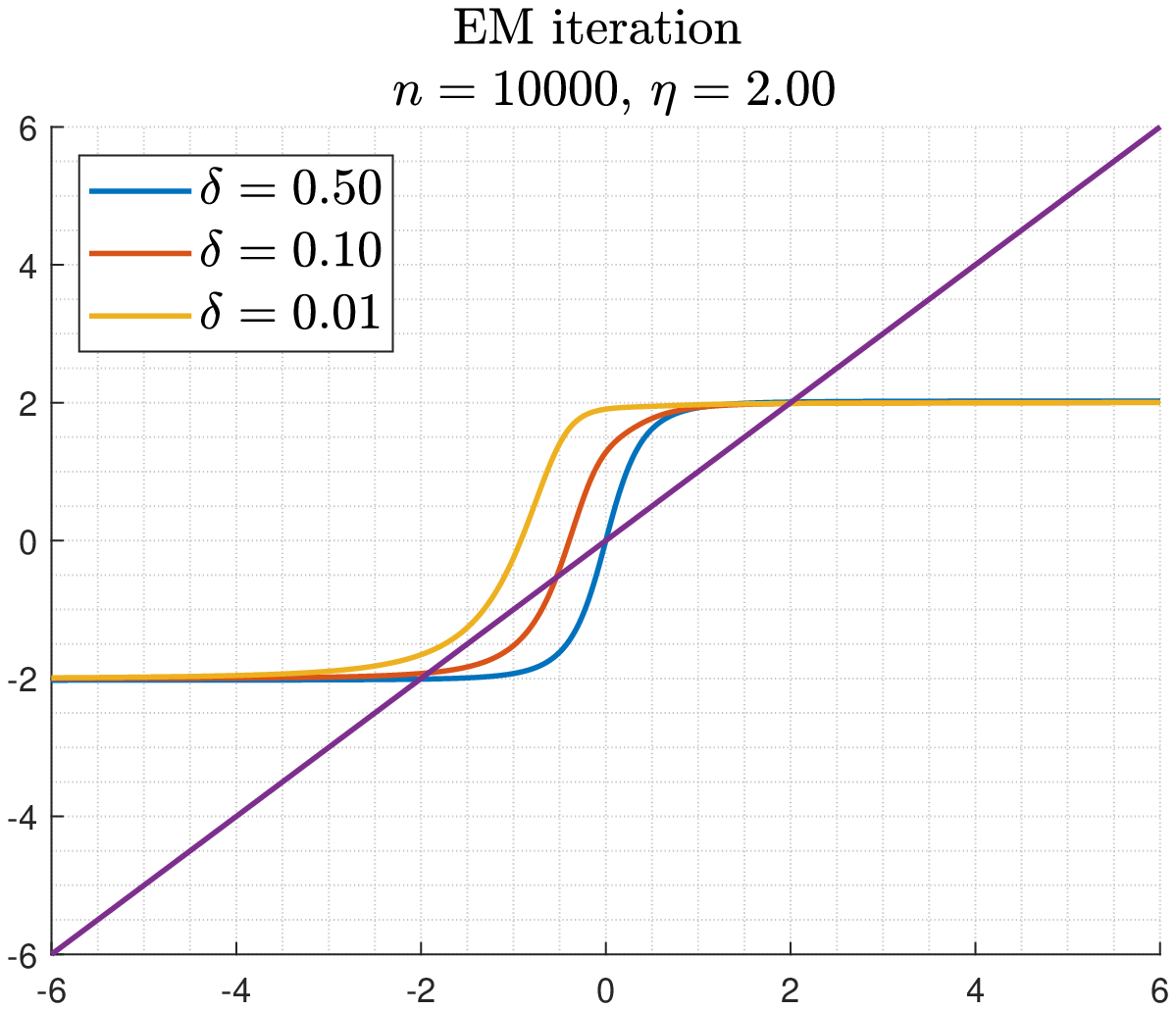} 
\par\end{centering}
\begin{centering}
\includegraphics[scale=0.8]{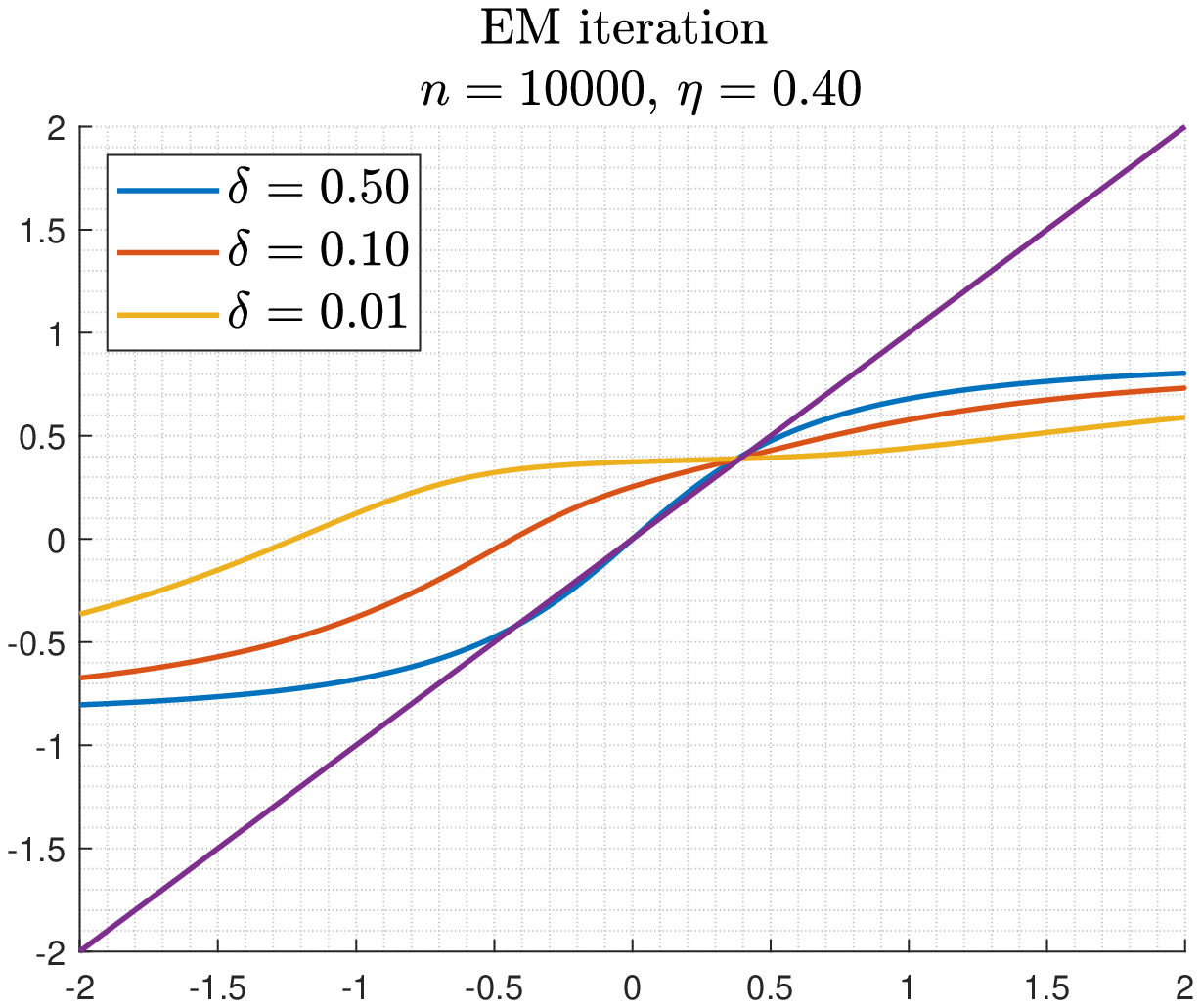} 
\par\end{centering}
\caption{Illustration of $f(\theta,\delta\mid\eta,\delta)$ for $\eta=2$ and
$\eta=0.4$. \label{fig: unbalanced iteration}}
\end{figure}

We begin with the population iteration. 
\begin{thm}[Population mean iteration, known weight, $d=1$]
\label{thm: population d=00003D00003D1 known delta} The following
holds: 
\begin{enumerate}
\item The unique fixed point of $\theta\mapsto f(\theta\mid\eta,\delta)$
in $\mathbb{R}_{+}$ is $\theta=\eta$, and its fixed points in $\mathbb{R}_{-}$
are confined to the interval $(-\eta,0)$. 
\item If $\theta_{0}\geq0$ then the iteration $\theta_{t+1}=f(\theta_{t}\mid\eta,\delta)$
converges to $\eta$. 
\item Let $\delta\leq\tilde{\delta}<1/2$ and $\theta_{0}\geq0$. Consider
the iteration $\tilde{\theta}_{t}=f(\tilde{\theta}_{t-1}\mid\eta,\tilde{\delta})$
such that $\theta_{0}=\tilde{\theta}_{0}\geq0$. Then $|\eta-\tilde{\theta}_{t}|>|\eta-\theta_{t}|$
for all $t\geq1$, i.e., the convergence is faster as $\delta$ is
lower. The same holds for $\tilde{\delta}=1/2$ if $\theta_{0}>0$. 
\end{enumerate}
\end{thm}

Though not crucial for the analysis or later derivations, we conjecture
from exhaustive numerical evidence that there are only two spurious
fixed points of $f(\theta)$ in $\mathbb{R}_{-}$, and furthermore,
there exists $\delta_{\text{cr}}\in(0,1/2)$ such that the number
of fixed points of $f(\theta)$ in $\mathbb{R}_{-}$ is 
\[
\begin{cases}
2, & \delta\in(\delta_{\text{cr}},\frac{1}{2})\\
1, & \delta=\delta_{\text{cr}}\\
0, & \delta\in[0,\delta_{\text{cr}})
\end{cases}\,.
\]
Theorem \ref{thm: population d=00003D00003D1 known delta} establishes
global convergence properties for the unbalanced one-dimensional EM
population iteration, when initializing either with the correct sign
of the larger weight component, or with a neutral sign ($\theta=0$).
Nonetheless, the illustration in Fig. \ref{fig: unbalanced iteration}
also demonstrates that initializing with the strictly wrong sign may
lead to convergence to a spurious stable fixed point which is at a
non-zero distance from $-\eta$, even when $n\to\infty$. Thus, a
correct initialization of this iteration is both simple and crucial.

To describe the proof idea of Theorem \ref{thm: population d=00003D00003D1 known delta},
we briefly recall the properties of the balanced iteration $f(\theta\mid\eta,\frac{1}{2})$
and why this iteration globally converges. As evident from Fig. \ref{fig: unbalanced iteration}
(and proved in \cite[Section 2]{wu2019EM}), $\theta\mapsto f(\theta\mid\eta,\frac{1}{2})$
is an odd increasing function which is concave on $\mathbb{R}_{+}$.
It is not difficult to show (see also Proposition \ref{prop: Properties of one dimensional iterations}
in Appendix \ref{sec:Miscellaneous}) that in this case that the iteration
must converge to one of the three unique fixed points $\theta=0,\pm\tilde{\theta}$.
The general consistency property of EM\footnote{which can also be proved directly, see proof of Lemma \ref{lem: d=00003D00003D1 simple properties}
below.} then implies that $\tilde{\theta}=\eta$. Thus, for the balanced
iteration, the concavity property provides a global property on the
iteration which is used to establish global convergence.

The reduced uncertainty in the $\delta<1/2$ case hints that the iteration
should converge faster and better than for the $\delta=1/2$ case.
However, at the same time, the change from $\delta=1/2$ to $\delta<1/2$
breaks the symmetry in the iteration, and hence the concavity property
of the iteration in $\mathbb{R}_{+}$. For example, the yellow iteration
in Fig. \ref{fig: unbalanced iteration} corresponding to $\delta=0.01$
shows that the iteration is non-concave around $\theta\approx0.6$.
Hence a different global property is required. The consistency property
assures at $\theta=\eta$ the iteration is insensitive to $\delta$,
and specifically that $f(\eta\mid\eta,\delta)=\eta$ for \emph{all}
$\delta$. As evident from Fig. \ref{fig: unbalanced iteration},
and as is true in general, for $\theta<\eta$ and $\theta>\eta$ a
change in $\delta$ bares an opposite, yet consistent effect. The
next proposition summarizes this property, along with another global
property related to ``oddness dominance'' that will be used in the
analysis of the empirical iteration. 
\begin{prop}
\label{prop: Comparison theorem d=00003D00003D1}Assume that $\eta>0$
and $\delta\in[0,\frac{1}{2}]$. 
\begin{enumerate}
\item \label{enu: one-dimensioanl iteration oddness dominance}$f(\theta\mid\eta,\delta)\geq-f(-\theta\mid\eta,\delta)$
for all $\theta>0$. 
\item \label{enu: one-dimensioanl iteration delta derivative dominance}$\delta\mapsto f(\theta\mid\eta,\delta)$
is non-increasing (resp. constant, resp. non-decreasing) for $\theta<\eta$
(resp. for $\eta=\theta$, resp. $\theta>\eta$). 
\end{enumerate}
\end{prop}

The second property can be described as $\theta=\eta$ being a ``pivot-point''
for the iteration as $\delta$ is varied (see Fig. \ref{fig: unbalanced iteration}).
Given the second property of Proposition \ref{prop: Comparison theorem d=00003D00003D1},
along with the known convergence for the case $\delta=1/2$, global
convergence can be easily established for $\theta_{0}\geq0$.

In order to prove property \ref{enu: one-dimensioanl iteration delta derivative dominance}
in Proposition \ref{prop: Comparison theorem d=00003D00003D1}, one
may assume an arbitrary and fixed true parameter $\eta>0$, and attempt
to establish this property for all $\theta\in\mathbb{R}$. It turns
out that if $\theta<0$, this property can be proved by a direct reasoning
on $\frac{\dee}{\dee\delta}f(\theta\mid\eta,\delta)$. However, similar
strategy seems daunting for $\theta>0$, and our proof is built on
an indirect argument. More explicitly, for $\theta>0$ it is required
to be proved that: 
\begin{equation}
\frac{\dee f(\theta\mid\eta,\delta)}{\dee\delta}\;\;\begin{cases}
>0, & 0<\eta<\theta\\
=0, & \eta=\theta\\
<0, & \eta>\theta
\end{cases}\,.\label{eq: derivative of the iteration wrt delta for positive theta}
\end{equation}
The proof of (\ref{eq: derivative of the iteration wrt delta for positive theta})
is based on the following ideas: 
\begin{itemize}
\item Analyzing $\frac{\dee}{\dee\delta}f(\theta\mid\eta,\delta)$ as a
function of $\eta$, rather than as a function of $\theta$ directly.
For this to be useful, we will need to analyze $\eta\in\mathbb{R}$
and not just $\eta\in\mathbb{R}_{+}$. 
\item Expressing $\frac{\dee}{\dee\delta}f(\theta\mid\eta,\delta)$ as a
convolution of some function with a Gaussian kernel. We then exploit
a \emph{variation diminishing property} of Gaussian kernels which
implies that if a function $h(\theta)$ has $k$ zero-crossings in
$\mathbb{R}$, its convolution with a Gaussian kernel may only reduce
the number of zero-crossings. See Proposition \ref{prop: Gaussian variation diminishing}
in Appendix \ref{subsec:Totally-positive-kernels} for a formal statement.
This allows us to prove that $\eta\mapsto\frac{\dee}{\dee\delta}f(\theta\mid\eta,\delta)$
has a single crossing point for some $\theta_{0}$. 
\item Then, utilizing the consistency property, which states that $f(\theta\mid\eta,\delta)=\theta$
for $\theta=\eta$ and \emph{any} $\delta$, establishes that $\theta_{0}=\theta$,
and this results (\ref{eq: derivative of the iteration wrt delta for positive theta}). 
\end{itemize}
The idea of analyzing the iteration w.r.t. the true parameter $\eta$,
rather than w.r.t. $\theta$ was previously used in \cite{pmlr-v65-daskalakis17b},
which utilized it in order to prove one-dimensional global convergence
(as well as convergence rates) for the balanced iteration. While such
an argument is not essential for this case (given the direct analysis
in \cite{wu2019EM}), an idea in that spirit is useful here for proving
a different global property.

We now turn to the empirical iteration. As we show next, the improved
convergence in the unbalanced case compared to the balanced case stated
in Theorem \ref{thm: population d=00003D00003D1 known delta} compensates
for the larger empirical error. 
\begin{thm}[Empirical mean iteration, known weight, $d=1$]
\label{thm: empirical d=00003D00003D1 known delta}Assume that $|\eta|\leq\gl[C]_{\theta}$,
$|\rho_{*}|\leq\gl[C]_{\rho}$ and that the high probability event
(\ref{eq: high probability event}) holds. Consider the empirical
mean EM iteration $\theta_{t}=f_{n}(\theta_{t-1})\equiv f_{n}(\theta_{t-1},\delta\mid\eta,\delta)$.
There exists $n_{0}(\gl[C]_{\theta},\gl[C]_{\rho})$ and constants
$\{\gl[C]_{i}^{(1)}(\gl[C]_{\theta},\gl[C]_{\rho})\}$ such that if
$\rho_{*}>\gl[C]_{1}^{(1)}\sqrt{\omega_{1}}$ and $n\geq n_{0}$ then
\[
\ell(\theta_{t},\eta)\leq\gl[C]_{2}^{(1)}\cdot\min\left\{ \frac{\omega_{1}}{\rho},\frac{\omega_{1}}{\eta}\right\} 
\]
holds for all $t\geq\gl[T]_{\theta_{0}=0}^{(1)}$ where 
\[
\gl[T]_{\theta_{0}=0}^{(1)}\dfn\begin{cases}
1, & \eta\leq\frac{\omega_{1}}{\rho}\\
\gl[C]_{3}^{(1)}\cdot\frac{1}{\rho^{2}}\log\left(\frac{\omega_{1}}{\gl[C]_{5}^{(1)}\rho}\right), & \frac{\omega_{1}}{\rho}\leq\eta\leq\gl[C]_{4}^{(1)}\rho\\
\gl[C]_{3}^{(1)}\cdot\left(\frac{1}{\rho\eta}+\frac{1}{\eta^{2}}\log\left(\frac{1}{\gl[C]_{6}^{(1)}\rho\omega_{1}}\right)\right), & \gl[C]_{4}^{(1)}\rho<\eta\leq\gl[C]_{\theta}
\end{cases}
\]
in case the iteration was initialized by $\theta_{0}=0$, and for
all $t\geq\gl[T]_{\theta_{0}=\frac{1}{\rho}\E_{n}[X]}^{(1)}$ where
\[
\gl[T]_{\theta_{0}=\frac{1}{\rho}\E_{n}[X]}^{(1)}\dfn\begin{cases}
1, & \eta\leq\gl[C]_{4}^{(1)}\rho\\
\gl[C]_{3}^{(1)}\frac{1}{\eta^{2}}\log\left(\frac{1}{\gl[C]_{6}^{(1)}\rho\omega_{1}}\right), & \gl[C]_{4}^{(1)}\rho<\eta\leq\gl[C]_{\theta}
\end{cases}
\]
in case the iteration was initialized by $\theta_{0}=\frac{1}{\rho}\E_{n}[X]$
. 
\end{thm}

The proof uses the sandwiching method developed in \cite{wu2019EM}
that bounds the empirical iteration $f_{-}(\theta)\leq f_{n}(\theta)\leq f_{+}(\theta)$
by the envelopes $f_{\pm}(\theta)=f(\theta)\pm\max\{|\theta|,\rho\}\cdot\omega_{1}$
(which holds with high probability), and the resulting iterations
$\theta_{t}^{\pm}=f_{\pm}(\theta_{t-1})$ obtained when initializing
those iterations and the empirical iteration with the same initial
guess $\theta_{0}^{\pm}=\theta_{0}$. We consider both $\theta_{0}=0$
or $\theta_{0}=\frac{1}{\rho}\E_{n}[X]$. First, it is proved that
all three iterations $\{\theta_{t}\},\{\theta_{t}^{\pm}\}$ converge
to fixed points, $\eta_{n},\eta_{\pm}$ respectively. Then, the analysis
is split into three regimes. For initialization at $\theta_{0}=0$: 
\begin{enumerate}
\item If $\frac{\omega_{1}}{\rho}\lesssim\eta\lesssim\rho$ then the envelopes
converge to fixed points $0<\eta_{-}\leq\eta_{n}\leq\eta_{+}$, and
the envelopes are approximated as $f_{\pm}(\theta)\approx1-c\rho^{2}$
for some $c>0$. Thus, the envelopes are contractions whose convergence
times is on the order of $\tilde{O}(\frac{1}{\rho^{2}})$. 
\item If $\rho\lesssim\eta$ then convergence has two phases. At the first
phase, $\theta_{t}$ is low, and the iteration increases additively
each step from $\theta_{0}=0$ by $\Omega(\rho^{2}\eta)$ at each
iteration. Thus, after $T_{1}=O(\frac{1}{\rho\eta})$ iterations,
$\theta_{t}=\Omega(\rho)$. At this point, the empirical error is
$\omega_{1}\cdot\max\{\theta,\rho\}\lesssim\omega_{1}\theta$, to
wit, the same empirical error as for the balanced iteration. Since
the population iteration converges faster in the unbalanced case than
in the balanced case (Theorem \ref{thm: population d=00003D00003D1 known delta}),
and the empirical error in this case is of the same order, then the
convergence of the unbalanced empirical iteration is only faster than
that of the empirical balanced iteration. The latter was analyzed
in \cite[Theorem 3]{wu2019EM}, and its result is utilized. 
\item If $\eta\lesssim\frac{\omega_{1}}{\rho}$ then analysis similar to
the previous cases shows that the upper envelope $\theta_{t}^{+}$
will increase, and will remain within the required statistical error,
i.e., $0\leq\theta_{t}^{+}\lesssim\frac{\omega_{1}}{\rho}$, for all
$t\geq1$. On the other hand, for the lower envelope, it is not guaranteed
that $f_{-}(0)>0$ , and so it is also not guaranteed that $\theta_{t}^{-}$
will increase and converge to a positive fixed point. If it does converge
to a positive fixed point $\eta_{-}>0$, then $\eta_{-}<\eta_{+}$
must also hold, and, as for the upper envelope, $0\leq\theta_{t}^{-}\lesssim\frac{\omega_{1}}{\rho}$,
for all $t\geq1$. If, however, this is not the case and $f_{-}(0)<0$,
then $\theta_{t}^{-}$ will decrease and converge to a negative fixed
point $\eta_{-}<0$. In that event, the oddness domination property
of the population iteration (Proposition \ref{prop: Comparison theorem d=00003D00003D1},
item \ref{enu: one-dimensioanl iteration oddness dominance}) assures
that $|\eta_{-}|\leq\eta_{+}$, and so the same statistical error
$O(\frac{\omega_{1}}{\rho})$ is again assured. 
\end{enumerate}
For initialization at $\theta_{0}=\frac{1}{\rho}\E_{n}[X]$, the error
in the first iteration is already within the statistical accuracy
in Cases 1 and 3. For Case 2, it can be shown that the first phase
of convergence does not occur, and the convergence is as in the balanced
case.

\subsection{The mean iteration for known weight at $d>1$\label{subsec:The-mean-iteration d>1}}

We now consider the general $d>1$ case. Recall that the $n$ i.i.d.
samples are given by $X\sim(1-\delta)\cdot N(\theta_{*},I_{d})+\delta\cdot N(-\theta_{*},I_{d})$
where $\theta_{*}\in\mathbb{R}^{d}$ and $\delta\in(0,1/2)$ is known.
For the population mean EM iteration we show that: 
\begin{thm}[Population iteration, known weight, $d>1$]
\label{thm: population d>1 known delta} Consider the population
mean EM iteration $\theta_{t}=f(\theta_{t-1},\delta\mid\theta_{*},\delta)$.
If $\langle\theta_{0},\theta_{*}\rangle\geq0$ then $\lim_{t\to\infty}\theta_{t}=\theta_{*}$.
Specifically, this holds for $\theta_{0}=0$. 
\end{thm}

As in the one-dimensional case, the idealized population iteration
does not converge to a spurious fixed point, whenever the iteration
is not initialized in a direction which has an obtuse angle with $\theta_{*}$
(which is the mean of the larger weight component, $1-\delta$). Furthermore,
such initialization can be achieved by setting $\theta_{0}=0$, since
$f(0,\delta\mid\theta_{*},\delta)=(1-2\delta)^{2}\cdot\theta_{*}$
which already points to the desired direction. Evidently, in this
case, $\theta_{t}\propto\hat{\theta}_{*}$ for all $t\geq1$. A slightly
more general property holds for any initialization, and this is the
basic ingredient of the proof which we describe next.

The proof of Theorem \ref{thm: population d>1 known delta} is based
on an observation made in \cite{xu2016global,pmlr-v65-daskalakis17b}
that the population mean EM iteration is ``trapped'' to the two-dimensional
space spanned by the true vector $\theta_{*}$ and the initial guess
$\theta_{0}$. In \cite{wu2019EM} this observation was formulated
in the following way. Let us denote $\eta\dfn\|\theta_{*}\|$ for
brevity, let $V\sim(1-\delta)\cdot N(\eta,1)+\delta\cdot N(-\eta,1)$
and $W\sim N(0,1)$ be such that $V\ind W$, and define: 
\[
F((a,b),\delta\mid\eta,\delta)=\E\left[V\tanh(aV+bW+\beta)\right]\,,
\]
and 
\[
G((a,b),\delta\mid\eta,\delta)=\E\left[W\tanh(aV+bW+\beta)\right]\,,
\]
where we omit the dependence in $(\eta,\delta)$ whenever it is inessential
and simply write $F(a,b\mid\eta,\delta),\;G(a,b\mid\eta,\delta)$,
or even just, $F(a,b),\;G(a,b)$. The following was proved in \cite[Lemma 4]{wu2019EM}
for the balanced case,\footnote{A similar, but not identical claim was previously made in \cite{pmlr-v65-daskalakis17b}.}
but the proof is similar for any $\delta\in[0,1]$ and thus omitted. 
\begin{lem}[{{\cite[Lemma 4]{wu2019EM}}}]
\label{lem:Decomposition to signal and orthogoanl iteration}Consider
the population mean iteration $\theta_{t+1}=f(\theta_{t},\delta\mid\theta_{*},\delta)$,
and define 
\[
\theta_{t}=a_{t}\cdot\hat{\theta}_{*}+b_{t}\cdot\xi_{t}
\]
where $\eta=\|\theta_{*}\|$, $\hat{\theta}_{*}=\theta_{*}/\eta$,
$\xi_{t}\perp\eta$ and $\|\xi_{t}\|=1$ such that $\spa\{\theta_{*},\xi_{t}\}=\spa\{\theta_{*},\theta_{t}\}$
and $b_{t}\geq0$. Then, $\theta_{t}\in\spa\{\theta_{0},\theta_{*}\}$
(i.e., $\xi_{t}=\xi_{0}$) for all $t$, and 
\begin{align*}
a_{t+1} & =F(a_{t},b_{t})\\
b_{t+1} & =G(a_{t},b_{t})\,.
\end{align*}
\end{lem}

We refer to $F(\cdot)$ as the \emph{signal iteration} and to $G(\cdot)$
as the \emph{orthogonal iteration}. Lemma \ref{lem:Decomposition to signal and orthogoanl iteration}
thus implies that to analyze the iteration $\theta_{t+1}=f(\theta_{t},\delta\mid\theta_{*},\delta)$
it suffices to analyze the evolution of $\{a_{t},b_{t}\}$, and that
$\theta_{t}\to\theta_{*}$ is equivalent to $a_{t}\to\eta$ and $b_{t}\to0$.

In the balanced case, the population mean EM iteration was shown to
globally converge to $\pm\theta_{*}$ by showing that the orthogonal
error goes to zero unconditionally of the signal iteration, i.e.,
$b_{t}\to0$ is always satisfied. Then, the problem is reduced to
the one-dimensional iteration studied in the previous section. Essentially,
this global property holds due to the fact that for $G(a,b\mid\eta,\frac{1}{2})$
is a concave increasing function with $G(a,0)=0$, and $\left.\frac{\dee}{\dee b}G(a,b)\right|_{b=0}<1$
\cite[Lemma 5]{wu2019EM}. As in the one-dimensional case, in the
unbalanced case, the lack of symmetry in case of $\delta<1/2$ breaks
down the concavity property of $G(a,b\mid\eta,\delta)$, and so a
different global property is required. This is achieved in the following
proposition, which states properties of $G(a,b\mid\eta,\delta)$ w.r.t.
the true parameters $(\eta,\delta)$, and specifically states that
$G(a,b\mid\eta,\delta)$ is dominated by $G(a,b\mid\eta,1/2)$, and
so the orthogonal iteration in case of $\delta<1/2$ converges faster
than in the case of $\delta=1/2$. 
\begin{prop}
\label{prop: Comparison theorem d>1}~ 
\begin{enumerate}
\item \label{enu: comparison theorem d>1  - G itertaion dominance}Monotonicity
w.r.t. $\delta$: If $a\geq0$ then $\delta\mapsto G(a,b\mid\eta,\delta)$
is an increasing function on $[0,1/2]$. 
\item \label{enu: comparison theorem d>1  - G itertaion explicit}Dominance
w.r.t. $\delta$: Let $C_{f}>0$ be given. There exists $\gl[C]_{G,\rho}^{(d)}(\gl[C]_{\theta},\gl[C]_{\beta},C_{f})>0$
and $n_{0}(\gl[C]_{\theta},\gl[C]_{\beta},C_{f})$ such that for any
$n\geq n_{0}$ 
\[
G(a,b\mid\eta,\delta)\leq G(a,b\mid\eta,\tfrac{1}{2})-\gl[C]_{G,\rho}^{(d)}\rho^{2}b\leq b\left(1-\frac{a^{2}+b^{2}}{2+4(a^{2}+b^{2})}-\gl[C]_{G,\rho}^{(d)}\rho^{2}\right)
\]
for all $(a,b,\eta)\in[0,C_{f}]^{2}\times[0,\gl[C]_{\theta}]$. 
\item Monotonicity w.r.t. $\eta$: For $a\geq0$, $\eta\mapsto G(a,b\mid\eta,\delta)$
is a decreasing function in $\eta\in[0,a+\frac{b^{2}}{a}]$. 
\item \label{enu: comparison theorem d>1  - G iteration dominance eta}Dominance
w.r.t. $\eta$: Let $\gl[C]_{G,\eta}^{(d)}=3\gl[C]_{\theta}$. Then,
\[
G(a,b\mid\eta,\delta)\leq G(a,b\mid0,\delta)+\gl[C]_{G,\eta}^{(d)}b\eta^{2}
\]
for all $(a,b,\eta)\in\mathbb{R}\times\mathbb{R}_{+}\times[0,\gl[C]_{\theta}]$. 
\end{enumerate}
\end{prop}

Given items \ref{enu: comparison theorem d>1  - G itertaion dominance}
and \ref{enu: comparison theorem d>1  - G itertaion explicit} of
Proposition \ref{prop: Comparison theorem d>1}, it is evident that
unconditional convergence of the orthogonal part of the iteration
to zero is assured in the unbalanced case (and the convergence is
only faster compared to the balanced case). After such convergence,
the problem is almost precisely reduced to the one-dimensional setting
in the signal iteration $\{a_{t}\}$, except for a small residual
additive error resulting from orthogonal iteration. However, as was
shown in the one-dimensional analysis, the unbalanced mean EM iteration
may tolerate small additive term (therein, this was due to the term
$\omega_{1}\rho$ which, unlike the error term $\omega_{1}\theta$
does not vanishes when $\|\theta\|\to0$), and so this additive term
does not prevent convergence.

We now turn to the empirical iteration: 
\begin{thm}[Empirical iteration, known weight, $d>1$, large $\rho_{*}$]
\label{thm: empirical d>1 known delta large rho}Assume that $\|\theta_{*}\|\leq\gl[C]_{\theta}$
and that $|\rho_{*}|\leq\gl[C]_{\rho}$, and that the high probability
event (\ref{eq: high probability event}) holds. Consider the empirical
mean EM iteration $\theta_{t}=f_{n}(\theta_{t-1},\delta\mid\theta_{*},\delta)$.
There exists $n_{0}$ and constants $\{\gl[C]_{i}^{(d)}\}$ which
depend on $(\gl[C]_{\theta},\gl[C]_{\rho})$ such that if $\rho_{*}>\gl[C]_{1}^{(d)}\sqrt{\omega}$
and $n\geq n_{0}$ then 
\[
\ell(\theta_{t},\theta_{*})\leq\gl[C]_{2}^{(d)}\min\left\{ \frac{\omega}{\rho},\frac{\omega}{\eta}\right\} 
\]
holds for all $t\geq\gl[T]_{\theta_{0}}^{(d)}=\gl[T]_{\theta_{0}}^{(1)}+\gl[T]_{G}^{(d)}$
where either $\theta_{0}=0$ or $\theta_{0}=\frac{1}{\rho}\E_{n}[X]$,
$\gl[T]_{\theta_{0}}^{(1)}$ is determined as in Theorem \ref{thm: empirical d=00003D00003D1 known delta}
by replacing $\omega_{1}\to\omega=\sqrt{\gl[C]_{\omega}\frac{d\log n}{n}}$\footnote{The constants determining $\gl[T]_{\theta_{0}}^{(1)}$ might also
be different than for $d=1$.} and where 
\[
\gl[T]_{G}^{(d)}\dfn\begin{cases}
1, & \eta\leq\rho\\
\frac{C_{3}^{(d)}}{\eta^{2}}\cdot\log\left(C_{4}^{(d)}\frac{\omega}{\eta}\right), & \eta\geq\rho
\end{cases}\,.
\]
\end{thm}

The proof of this theorem is based on splitting the analysis into
three regimes $0<\eta\lesssim\frac{\omega}{\rho}$, $\frac{\omega}{\rho}\lesssim\eta\lesssim\rho$
and $\eta\gtrsim\rho$. First, it is shown that when initializing
at either $\theta_{0}=0$ or $\theta_{0}=\frac{1}{\rho}\E_{n}[X]$,
the orthogonal iteration satisfies $b_{t}=O(\frac{\omega}{\rho})$,
and remains so for all iterations. In the $\eta\lesssim\frac{\omega}{\rho}$
regime, this is shown by the local behavior of $G(a,b\mid\eta,\delta)$
around $\eta\approx0$ (Proposition \ref{prop: Comparison theorem d>1},
item \ref{enu: comparison theorem d>1  - G iteration dominance eta}).
In the $\eta\gtrsim\frac{\omega}{\rho}$ this is proved by the dominance
relation to the balanced orthogonal iteration (Proposition \ref{prop: Comparison theorem d>1},
items \ref{enu: comparison theorem d>1  - G itertaion dominance}
and \ref{enu: comparison theorem d>1  - G itertaion explicit}), along
with a verification that $a_{t}$ remains positive for all iterations
(so that these dominance relations are in fact valid). Given that
$b_{t}=O(\frac{\omega}{\rho})$, the effect of the orthogonal iteration
on the signal iteration is negligible, and it is essentially reduced
to the one-dimensional iteration. The convergence time for this is
$\gl[T]_{\theta_{0}}^{(1)}$ (when setting the specific initialization
for $\theta_{0}$ \footnote{In fact, Theorem \ref{thm: empirical d>1 known delta large rho} is
valid for any $\theta_{0}$ for which $b_{t}=O(\frac{\omega}{\rho})$. }), and the resulting error for the signal iteration is $|a_{t}-\eta|=O(\min\{\frac{\omega}{\rho},\frac{\omega}{\eta}\})$.
If $\eta<\rho$ then this is also the error rate for $\theta_{*}$
as $|\theta-\theta_{*}|=O(|a_{t}-\eta|+b_{t})$. If $\eta>\rho$,
then the orthogonal iteration can be shown to decrease to $O(\frac{\omega}{\eta})$
after additional $\gl[T]_{G}^{(d)}$ iterations.

We next consider the case of $\rho_{*}$ which is too small to satisfy
the condition of Theorem \ref{thm: empirical d>1 known delta large rho}. 
\begin{prop}[Empirical iteration, known weight, $d>1$, small $\rho_{*}$]
\label{prop: empirical d>1 known delta small rho}Assume that $\|\theta_{*}\|\leq\gl[C]_{\theta}$
and that the high probability event (\ref{eq: high probability event})
holds. Further assume that the balanced EM weight iteration is run
$\theta_{t}=f_{n}(\theta_{t-1},\delta=\frac{1}{2}\mid\theta_{*},\delta=\frac{1}{2})$
with random initialization as in Theorem \ref{thm:balanced case}.
Let $\tilde{\theta}_{t}=s_{t}\theta_{t}$ where $s_{t}=\sgn\langle\E_{n}[X],\theta_{t}\rangle$.
Then, there exists $\gl[C]_{\rho}^{(d)}(\gl[C]_{\theta})>0$ such
that if 
\[
\frac{\omega}{\eta}\leq\rho_{*}\leq\gl[C]_{1}^{(d)}\sqrt{\omega}\leq\gl[C]_{\eta}^{(d)}\eta\,,
\]
then 
\[
\ell(\theta_{t},\theta_{*})\leq\gl[C]_{2}^{(d)}\frac{\omega}{\eta}
\]
holds for all $t\geq\frac{\log n}{\|\theta_{*}\|^{2}}.$ 
\end{prop}

Theorem \ref{thm: empirical d>1 known delta large rho} and Proposition
\ref{prop: empirical d>1 known delta small rho} together imply Theorem
\ref{thm:Main result}, which is the main result of the paper.

\subsection{The weight iteration for a fixed mean \label{subsec:The-weight-iteration}}

In previous sections, we have considered the mean iteration assuming
a known weight. In this section, we study the opposite extreme case,
and study the weight iteration assuming a fixed mean $\theta$, and
specifically, the case in which $\theta=\theta_{*}$ holds. In this
case, the log-likelihood is given by 
\[
\rho_{\text{MLE}}=\argmax_{\rho\in[0,1]}\sum_{i=1}^{n}\log\left(\frac{1+\rho}{2}e^{-\langle X_{i},\theta_{*}\rangle}+\frac{1-\rho}{2}e^{+\langle X_{i},\theta_{*}\rangle}\right).
\]
As apparent and also well-known, the log-likelihood is a concave function
of the unknown parameter $\rho$, and so, the EM algorithm is assured
to converge to the MLE \cite{wu1983convergence}. Alternatively, a
simple method-of-moments estimator $\rho_{\text{MoM}}=\frac{1}{\|\theta\|}\langle\hat{\theta},\E_{n}[X]\rangle$
can be readily shown to achieve the minimax error rate for this problem,
given roughly by $\min\{\frac{1}{\|\theta_{*}\|\sqrt{n}},1\}$ (see
Theorem \ref{thm: weight estimation minimax} in Appendix \ref{sec:Minimax-rates}).
Nonetheless, in this section we directly analyze the EM iteration
and provide statistical and computational guarantees similar to the
previous sections. Despite the favorable behavior mentioned above,
the analysis of the EM iteration is delicate, especially in the mismatched
case $\theta\neq\theta_{*}$. Understanding the EM iteration in this
setting may then further illuminate its basic features.

We thus assume in this section the model 
\[
P_{\rho}=\frac{1+\rho}{2}\cdot N(\theta_{*},1)+\frac{1-\rho}{2}\cdot N(-\theta_{*},1)
\]
where $\delta=\frac{1-\rho}{2}$ . As in Section \ref{subsec:The-mean-iteration d>1},
the weight iteration can be written as 
\begin{align*}
h(\rho,\theta\mid\theta_{*},\rho_{*}) & =\E\left[\frac{(1+\rho)e^{\|\theta\|V}-(1-\rho)e^{-\|\theta\|V}}{(1+\rho)e^{\|\theta\|V}+(1-\rho)e^{-\|\theta\|V}}\right]
\end{align*}
where $V\sim(\frac{1+\rho_{*}}{2})\cdot N(\langle\hat{\theta},\theta_{*}\rangle,1)+(\frac{1-\rho_{*}}{2})\cdot N(-\langle\hat{\theta},\theta_{*}\rangle,1)$
and $\hat{\theta}=\theta/\|\theta\|$. Similarly, the empirical iteration
will be denoted by $h_{n}(\rho,\theta)$. In addition, since $|\rho_{*}|\leq\gl[C]_{\rho}$
is assumed, we may also consider truncated iterations given by $[h(\rho,\theta\mid\theta_{*},\rho_{*})]_{\gl[C]_{\rho}}$where
\[
[t]_{\gl[C]_{\rho}}=\begin{cases}
-\gl[C]_{\rho}, & t<-\gl[C]_{\rho}\\
t, & -\gl[C]_{\rho}<t<\gl[C]_{\rho}\\
\gl[C]_{\rho}, & t>\gl[C]_{\rho}
\end{cases}\,.
\]
Fig. \ref{fig: weightiteration} illustrates EM iteration (based on
single runs of $n=10^{6}$ samples).\textbf{ } 
\begin{figure}
\centering{}\textbf{\includegraphics[scale=0.8]{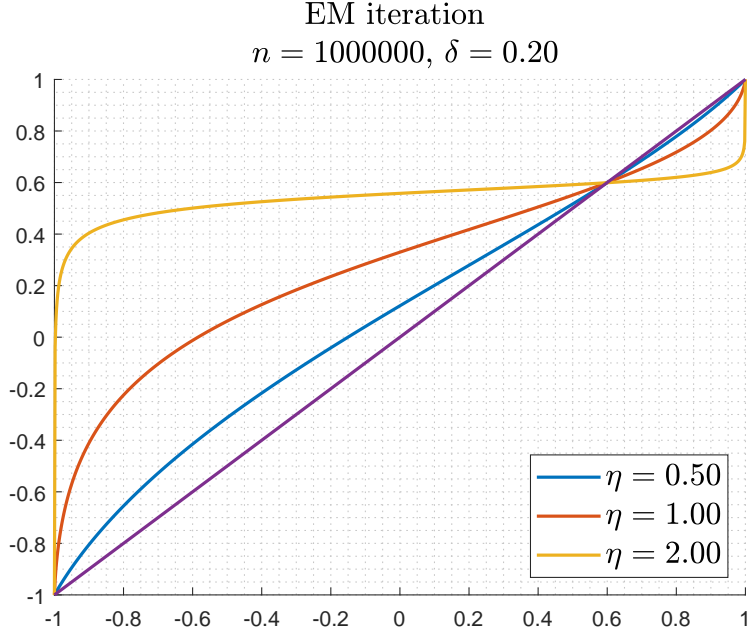}\caption{Illustration of $h(\rho,\eta\mid\eta,\rho_{*})$ for $\rho_{*}=0.6$
. \label{fig: weightiteration}}
} 
\end{figure}

\textbf{ }

As expected, for a given fixed $\theta$, the iteration is essentially
one-dimensional and does not depend on $d$. Note also that if $\langle\theta,\theta_{*}\rangle=0$
then $\rho_{*}$ is not identifiable, and, in accordance, the population
iteration is useless; indeed $h(\rho)=\rho$ for this case. Regarding
the population iteration, we have the following theorem: 
\begin{thm}[Population weight iteration, fixed mean]
\label{thm: population weight fixed mean}Assume that $\rho_{*}>0$
and that $\langle\theta,\theta_{*}\rangle\neq0$. The following holds: 
\begin{enumerate}
\item The iteration $h(\rho,\theta)$ has either two or three fixed points
in $[-1,1]$. The boundaries $\rho=\pm1$ are always fixed points.
There exists a third fixed point $\rho_{\#}\in(-1,1)$ if and only
if 
\begin{equation}
\left.\frac{\d}{\d\rho}h(\rho,\theta)\right|_{\rho=1}=e^{2\|\theta\|^{2}}\left[\left(\frac{1+\rho_{*}}{2}\right)e^{-2\langle\theta,\theta_{*}\rangle}+\left(\frac{1-\rho_{*}}{2}\right)\cdot e^{2\langle\theta,\theta_{*}\rangle}\right]>1\label{eq: condition on fixed point in weight iteration}
\end{equation}
and if it exists, it satisfies $\rho_{\#}\in(0,1)$ if $\langle\theta,\theta_{*}\rangle>0$
and $\rho_{\#}\in(-1,0)$ if $\langle\theta,\theta_{*}\rangle<0$.
Specifically, condition (\ref{eq: condition on fixed point in weight iteration})
holds if $\|\theta\|>|\langle\hat{\theta},\theta_{*}\rangle|$, and,
furthermore, if $\theta=\theta_{*}$ then $\rho_{\#}=\rho_{*}$. 
\item If $\rho_{\#}>0$ exists then the iteration $\rho_{t+1}=h(\rho_{t},\theta)$
converges monotonically upwards (resp. downwards) to $\rho_{\#}$
if $\rho_{0}\in(-1,\rho_{\#}]$ (resp. $\rho_{0}\in[\rho_{\#},1)$). 
\end{enumerate}
\end{thm}

The proof mainly utilizes the following properties of $\rho\mapsto h(\rho,\theta)$:
It increases monotonically from $h(-1,\theta)=-1$ to $h(1,\theta)=1$,
and in case there is a fixed point $\rho_{\#}\in(-1,1)$, its uniqueness
follows from the fact that $h(\rho,\theta)$ changes its curvature
only once as $\rho$ traverse from $-1$ to $1$ (from concave to
convex).

A rough characterization of the influence of mismatched $\theta$
can be derived as follows. Note that for the method-of-moments estimator,
a mismatch in the knowledge of $\theta_{*}$ when assuming the true
vector is $\theta\neq\theta_{*}$ results in bias in estimation, such
that, on the population level 
\[
\rho_{\text{MoM}}=\frac{1}{\|\theta\|}\langle\hat{\theta},\E[X]\rangle=\frac{1}{\|\theta\|}\langle\hat{\theta},\theta_{*}\rangle\cdot\rho_{*}\,.
\]
Thus, $\rho_{\text{MoM}}<\rho_{*}$ if and only if $\langle\hat{\theta},\theta_{*}\rangle<\|\theta\|$.
The next proposition shows the same effect for the EM iteration: 
\begin{prop}
\label{prop: population weight wrong mean effect}Assume that $\rho_{*}>0$
and that $\langle\hat{\theta},\theta_{*}\rangle>0$. Let $\rho_{\#}$
be the fixed point of $\rho\mapsto h(\rho,\theta)$ which satisfies
$\rho_{\#}\in(-1,1)$ (if such exists). Then, $\rho_{\#}<\rho_{*}$
if and only if $\langle\hat{\theta},\theta_{*}\rangle<\|\theta\|$. 
\end{prop}

The proof of the global property in Proposition \ref{prop: population weight wrong mean effect}
is again based on the variation diminishing property of the Gaussian
kernel, on consistency, and on exploring the location of the fixed
point as a function of the true parameter $\theta_{*}$ for a fixed
$\theta$.

The empirical weight iteration satisfies the following theorem: 
\begin{thm}[Empirical weight iteration, known mean]
\label{thm: sample weight known mean}Assume that $\|\theta_{*}\|\leq\gl[C]_{\theta}$
and that $|\rho_{*}|\leq\gl[C]_{\rho}$, and that the high probability
event (\ref{eq: high probability event}) holds. Consider the truncated
empirical weight iteration $\rho_{t}=[h_{n}(\rho_{t-1},\theta_{*}\mid\theta_{*},\rho_{*})]_{\gl[C]_{\rho}}$
when initialized with $\rho_{0}=0$. There exists $n_{0}(\gl[C]_{\theta},\gl[C]_{\rho})$
and constants $\{\gl[C]_{i}^{(\rho)}\}$ which depend on $(\gl[C]_{\theta},\gl[C]_{\rho})$
such that if $\|\theta_{*}\|>\gl[C]_{1}^{(\rho)}\frac{\omega_{1}}{\rho_{*}}$
and $n\geq n_{0}$ then 
\[
\ell(\rho_{t},\rho_{*})\leq\gl[C]_{2}^{(\rho)}\frac{\omega_{1}}{\|\theta_{*}\|}
\]
holds for all $t\geq\gl[T]^{(\rho)}=\frac{\gl[C]_{3}^{(\rho)}}{\|\theta_{*}\|^{2}}$. 
\end{thm}

The proof is based on bounding the empirical iteration with envelopes
of absolute error $\omega_{1}\|\theta_{*}\|$, and analyzing their
convergence. As might be expected, both the error bound and convergence
time diverge when $\|\theta_{*}\|\to0$; In the extreme case $\|\theta_{*}\|=0$,
$\rho_{*}$ is not identifiable at all. The error bound of the EM
iteration (and thus, also the MLE) matches that of the method-of-moments
estimator, and also the minimax error rate (Theorem \ref{thm: weight estimation minimax}
in Appendix \ref{sec:Minimax-rates}).

\subsection{An open problem: joint mean and weight estimation\label{subsec:Joint-mean-and}}

We have analyzed the EM algorithm for the model $P_{\theta,\delta}$
in case one of the parameters is known and the other is required to
be estimated. We next briefly discuss the more challenging scenario
in which both $\delta_{*}$ and $\theta_{*}$ are required to be jointly
estimated. In this case, each of the parameters serves as a nuisance
parameter for estimating the other one, and the exact statistical
and computational rates of EM remains an open problem. We nonetheless
briefly discuss several aspects of this problem.

For the idealized population version, it is straightforward to ensure
convergence, even far from the solution by a proper ``scheduling'',
i.e., not necessarily running both (\ref{eq: EM iteration mean})-(\ref{eq: EM iteration weight})
at each iteration. Specifically, a simple possible scheduling is ``freezing''
$\theta_{t}=\theta_{0}$ and running the weight iteration for $T_{0}$
steps until convergence, then freezing $\rho_{t}=\rho_{T_{0}}$ and
running the mean iteration until convergence, and so on. The initialization
and scheduling order will then affect convergence. If we set $\rho_{0}=0$
and run the balanced mean iteration $\theta_{t+1}=f(\theta_{t},\frac{1}{2}|\eta,\delta)$
it will converge to $\theta_{*}$. If we after this convergence we
will run the weight iteration $\rho_{t+1}=h(\rho_{t},\theta_{*})$
it will converge to $\rho_{*}$. Thus, this scheduling globally converges
to $(\theta_{*},\rho_{*})$. By contrast, while we have empirically
observed that initializing with a frozen $\theta_{0}$ also globally
converges, it is more challenging to establish via our methods. To
see this, suppose for simplicity that $d=1$. Note that Proposition
\ref{prop: population weight wrong mean effect} hints the importance
of assuring that $\theta_{0}>\theta_{*}\equiv\eta$ so that the weight
iteration $\rho_{t+1}=h(\rho_{t},\theta_{0})$ will have a fixed point
$\rho_{\#}<\rho_{*}$. If this condition does not hold, then the weight
iteration might not have a fixed point $\rho_{\#}$ in $(-1,1)$,
and the iteration will converge to the spurious fixed point of $\rho=1$.
Thus, we would like to initialize with a frozen $\theta_{0}$ such
that $|\theta_{0}|>\eta$. By our assumptions, this could be achieved,
by setting $\|\theta_{0}\|=\gl[C]_{\theta}.$ Next, we freeze $\rho_{t}$
at the obtained fixed point $\rho_{\#}$, and run the mean iteration.
The following property can be proved: Let $X\sim(1-\delta_{*})\cdot N(\eta,1)+\delta_{*}\cdot N(-\eta,1)$
with $\eta>0$ and assume that $\theta>0$. Then $f(\theta,\delta\mid\eta,\delta_{*})<f(\theta,\delta\mid\eta,\delta)$
if and only if $\delta_{*}>\delta$. Due to consistency and the convergence
properties of $f(\theta,\delta\mid\eta,\delta)$ (Theorem \ref{thm: population d=00003D00003D1 known delta}),
it can also be assured that $f(\theta,\delta\mid\eta,\delta_{*})$
has no fixed points in $(\eta,\infty)$, and has at least a single
fixed point in $(0,\eta]$ (which might not be unique). Upon convergence
to such a fixed point $\eta_{0}\leq\eta$, we may freeze it and run
the weight iteration $h(\rho,\eta_{0})$. At this phase, since $\eta_{0}<\eta$
it is not clear that the weight iteration will have a non-trivial
fixed point $\rho_{\#}\in(-1,1)$. A more delicate argument is required
to assure global convergence for a scheduling that begins with a phase
of frozen $\theta_{t}$. For the population iteration, we can always
choose to begin with $\rho=0$ that is provably globally converges,
but it is not clear if this scheduling is better in terms of the empirical
iteration.

For the empirical iteration, the result of running the balanced mean
EM iteration (or a spectral algorithm) can clearly be used to obtain
with high probability an initial guess $\theta_{0}$ with $\|\theta_{0}-\theta_{*}\|=\tilde{O}(\sqrt{\omega})$
if $\|\theta_{*}\|\lesssim\sqrt{\omega}$ and $\|\theta_{0}-\theta_{*}\|=\tilde{O}(\frac{\omega}{\|\theta_{*}\|})$
otherwise. In the former case, $\theta_{0}$ is not informative regarding
the direction of $\theta_{*}$, and so this initial guess is not expected
to be better than $\theta_{0}=0$. When $\|\theta_{*}\|\gtrsim\sqrt{\omega}$,
the initial guess has non-trivial angle with $\theta_{*}$, and so
it seems beneficial to initialize with that $\theta_{0}$. Furthermore,
the only possible case case in which EM algorithm can improve the
error rate is $\rho_{*}>\|\theta_{*}\|\gtrsim\sqrt{\omega}.$ A possible
direction to prove such a result, is to learn the stability of the
mean iteration w.r.t. error in the weight and vice-versa. Specifically,
using $\theta$ with $\|\theta-\theta_{*}\|\lesssim\frac{\omega}{\|\theta_{*}\|}$
in the iteration $\rho_{t+1}=h(\rho_{t},\theta\mid\theta_{*},\rho_{*})$
shifts the population fixed point by at most $O(\frac{\omega}{\|\theta_{*}\|^{2}})$,
but this is larger than the shift due to the empirical error which
is $O(\frac{\omega}{\|\theta_{*}\|})$. If, however, one can use $\rho$
with $|\rho-\rho_{*}|\lesssim\frac{\omega}{\|\theta_{*}\|}$ in the
empirical mean iteration $\theta_{t+1}=f_{n}(\theta_{t},\rho_{t}\mid\theta_{*},\rho_{*})$
then this mismatch can be shown to be negligible compared to the empirical
error. Thus, with this scheduling, the key point is how to finely
estimate $\theta_{*}$ so that its effect on the weight iteration
will be negligible. Nonetheless, if non trivial separation holds and
$\theta_{*}=\Omega(1)$ then running the balanced EM mean iteration
followed by the weight iteration leads to (nearly) optimal error rates.

\section{Proofs for Section \ref{subsec:The-mean-iteration d=00003D00003D1}}

\subsection{Population iteration}

The following lemma summarizes simple properties of the mean population
iteration for $d=1$. 
\begin{lem}
\label{lem: d=00003D00003D1 simple properties}Assume that $\eta\geq0$.
The following properties hold for $f(\theta)\equiv f(\theta,\delta\mid\eta,\delta)$: 
\begin{enumerate}
\item \label{enu: d=00003D00003D1 simple properties - iteration}Iteration:
$f(0)=(1-2\delta)^{2}\cdot\eta>0$, $f(\eta)=\eta$ (consistency)
and $\lim_{\theta\to\infty}f(\theta)\leq\eta+1<\gl[C]_{\theta}+1$. 
\item \label{enu: d=00003D00003D1 simple properties - first derivative}First
order derivative: $f(\theta)$ is increasing on $\mathbb{R}$ and
\[
f'(\theta)=\E\left[\frac{X^{2}}{\cosh^{2}(X\theta+\beta)}\right]=4\delta(1-\delta)\cdot\E\left[\frac{X^{2}}{\left((1-\delta)\cdot e^{X\theta}+\delta\cdot e^{-X\theta}\right)^{2}}\right]>0\,.
\]
At $\theta=0$ 
\[
f'(0)=4\delta(1-\delta)\cdot\left[\eta^{2}+1\right]\,,
\]
and furthermore, if $\theta\geq\eta$ then 
\[
f'(\theta)\leq2\sqrt{\delta(1-\delta)}e^{-\frac{1}{2}\eta^{2}}\leq1-\frac{1}{4}\cdot\max\{\min\{\eta^{2},1\},\rho^{2}\}\,.
\]
\item Second order derivative: 
\[
f''(\theta)=-2\E_{X}\left[\frac{X^{3}\tanh(X\theta+\beta)}{\cosh^{2}(X\theta+\beta)}\right]=-8\cdot\delta(1-\delta)\cdot\E\left[X^{3}\frac{\left((1-\delta)\cdot e^{X\theta}-\delta\cdot e^{-X\theta}\right)}{\left((1-\delta)\cdot e^{X\theta}+\delta\cdot e^{-X\theta}\right)^{3}}\right]\,.
\]
Furthermore, there exists $\gl[C]''(\gl[C]_{\theta},\gl[C]_{\rho})$
such that for all $\theta\in\mathbb{R}$ 
\[
\left|f''(\theta)\right|\leq\gl[C]''\cdot\max\{\theta,\rho\}\,.
\]
\end{enumerate}
\end{lem}

\begin{proof}
We only explicitly prove non-trivial properties, or ones which are
non-trivial extensions of \cite[Lemma 3]{wu2019EM}. Let $Z\sim N(0,1)$. 
\begin{enumerate}
\item The consistency property is well-known, but can also be proved explicitly:
\begin{align*}
f(\eta) & \trre[=,a]e^{-\eta^{2}/2}\cdot\E\left[Z\cdot\left((1-\delta)\cdot e^{Z\eta}-\delta\cdot e^{-Z\eta}\right)\right]\\
 & =e^{-\eta^{2}/2}\cdot\E\left[Ze^{Z\eta}\right]-\delta e^{-\eta^{2}/2}\cdot\E\left[Z\cdot\left(e^{Z\eta}+e^{-Z\eta}\right)\right]\\
 & \trre[=,b]e^{-\eta^{2}/2}\cdot\E\left[Ze^{Z\eta}\right]\trre[=,c]e^{-\eta^{2}/2}\cdot\E\left[\eta e^{Z\eta}\right]=\eta\,,
\end{align*}
where $(a)$ is by change of measure (see (\ref{eq: change of measure})
in Appendix \ref{subsec:Usful-results}), $(b)$ is by oddness of
the argument in the second expectation, $(c)$ is by Stein's identity
(see (\ref{eq: Stein lemma}) in Appendix \ref{subsec:Usful-results}). 
\item The bound on $f'(\theta)$ for $\theta>\eta$ holds since 
\begin{align}
f'(\theta) & \trre[=,a]e^{-\frac{1}{2}\eta^{2}}\cdot\E\left[Z^{2}\cdot\frac{4\delta(1-\delta)}{\left((1-\delta)\cdot e^{Z\theta}+\delta\cdot e^{-Z\theta}\right)^{2}}\cdot\left((1-\delta)\cdot e^{Z\eta}+\delta\cdot e^{-Z\eta}\right)\right]\nonumber \\
 & \trre[\leq,b]e^{-\frac{1}{2}\eta^{2}}\cdot\E\left[Z^{2}\cdot2\sqrt{\delta(1-\delta)}\cdot\frac{(1-\delta)\cdot e^{Z\eta}+\delta\cdot e^{-Z\eta}}{(1-\delta)\cdot e^{Z\theta}+\delta\cdot e^{-Z\theta}}\right]\nonumber \\
 & \trre[=,c]2\sqrt{\delta(1-\delta)}e^{-\frac{1}{2}\eta^{2}}\cdot\E\left[\frac{1}{2}Z^{2}\cdot\frac{(1-\delta)\cdot e^{Z\eta}+\delta\cdot e^{-Z\eta}}{(1-\delta)\cdot e^{Z\theta}+\delta\cdot e^{-Z\theta}}+\frac{1}{2}Z^{2}\cdot\frac{(1-\delta)\cdot e^{-Z\eta}+\delta\cdot e^{Z\eta}}{(1-\delta)\cdot e^{-Z\theta}+\delta\cdot e^{Z\theta}}\right]\nonumber \\
 & \trre[\leq,d]e^{-\frac{1}{2}\eta^{2}}\cdot2\sqrt{\delta(1-\delta)}\cdot\E\left[Z^{2}\right]\label{eq: mean iteration d=00003D00003D1 simple properties derivative bound}
\end{align}
where: $(a)$ is proved again by a change of measure; $(b)$ holds
since by the inequality of arithmetic and geometric means, for any
$a\geq0$ 
\[
\frac{4\delta(1-\delta)}{(1-\delta)\cdot a+\delta\cdot a^{-1}}\leq2\sqrt{\delta(1-\delta)}\,,
\]
$(c)$ follows since $Z\eqd-Z$; $(d)$ holds since 
\begin{equation}
\max_{\delta\in[0,\frac{1}{2}],\;a>b>1}\frac{1}{2}\cdot\frac{(1-\delta)\cdot b+\delta\cdot b^{-1}}{(1-\delta)\cdot a+\delta\cdot a^{-1}}+\frac{1}{2}\cdot\frac{(1-\delta)\cdot b^{-1}+\delta\cdot b}{(1-\delta)\cdot a^{-1}+\delta\cdot a}=1\,.\label{eq: symmetrized ratio}
\end{equation}
To show that (\ref{eq: symmetrized ratio}) holds, note that objective
function on the l.h.s. is a convex function of $b$ for a given $(a,\delta)$,
hence it is maximized for either $b=a$ or $b=1$. At $b=a$ the value
of the objective is $1$. At $b=1$ we maximize over $a>1$: 
\begin{align*}
 & \max_{a>1}\frac{1}{2}\cdot\frac{1}{(1-\delta)\cdot a+\delta\cdot a^{-1}}+\frac{1}{2}\cdot\frac{1}{(1-\delta)\cdot a^{-1}+\delta\cdot a}\\
 & =\max_{a>1}\frac{1}{2}\cdot\frac{a+a^{-1}}{(1-\delta)^{2}+\delta^{2}+\delta(1-\delta)\cdot\left(a+a^{-1}\right)}\\
 & =\max_{c>2}\frac{1}{2}\cdot\frac{c}{(1-2\delta)^{2}+\delta(1-\delta)\cdot c}=1\,,
\end{align*}
where $c=a+a^{-1}$ was used, and the fact that the function to be
maximized has a single maximum in $\mathbb{R}_{+}$ at $c=\frac{1-2\delta}{\sqrt{\delta}(1-\delta)}\leq2$.
Using $\delta=\frac{1-\rho}{2}$, we thus have $f'(\theta)=\sqrt{1-\rho^{2}}\cdot e^{-\frac{1}{2}\eta^{2}}$.
The final bound is obtained from $\sqrt{1-\rho^{2}}\leq1-\frac{\rho^{2}}{2}$
and $e^{-\frac{1}{2}\eta^{2}}\leq\max\{e^{-1},1-\frac{\eta^{2}}{4}\}$. 
\item The bound on the second derivative follows from $|\tanh(t)|\leq t$
and $\cosh(t)>1$. 
\end{enumerate}
\end{proof}
We next turn to prove Proposition \ref{prop: Comparison theorem d=00003D00003D1}.
To this end, we need the following technical lemma: 
\begin{lem}
\label{lem: s is always negative} Let $\beta>0$ be given, and let
\begin{equation}
s(u)\dfn-\tanh(u+\beta)+\tanh(u-\beta)-\frac{1}{2\delta\cosh^{2}(u+\beta)}+\frac{1}{2(1-\delta)\cosh^{2}(u-\beta)}\,.\label{eq: s function}
\end{equation}
Then, $\frac{\d}{\d u}s(u)=0$ has a unique solution, this solution
is negative, and $s(u)<0$ for all $u\in\mathbb{R}.$ 
\end{lem}

\begin{proof}
For $u<0$ the claim that $s(u)<0$ holds since both $\tanh(u-\beta)<\tanh(u+\beta)$
and $(1-\delta)\cosh^{2}(u-\beta)>\delta\cosh^{2}(u+\beta)$ hold
when $\beta>0$. For $u\geq0$, we begin by analyzing $\frac{\d}{\d u}s(u)$
and show that the real solution of $\frac{\d}{\d u}s(u)=0$ is negative
and unique. Using the double-argument identities $1+\cosh(2t)=2\cosh^{2}(t)$,
$\sinh(2t)=2\sinh(t)\cosh(t)$, the half-argument identity $\tanh(\frac{t}{2})=\frac{\sinh(t)}{\cosh(t)+1)}$
we obtain that the derivative is 
\begin{equation}
\frac{\d s(u)}{\d u}=-\frac{1-\frac{1}{\delta}\tanh(u+\beta)}{\cosh^{2}(u+\beta)}+\frac{1-\frac{1}{1-\delta}\tanh(u-\beta)}{\cosh^{2}(u-\beta)}\,,\label{eq: s function derivative}
\end{equation}
and thus get that $\frac{\d}{\d u}s(u)=0$ is equivalent to 
\[
\left[\frac{1+\cosh(2u-2\beta)}{1+\cosh(2u+2\beta)}\right]^{2}=\frac{1+\cosh(2u-2\beta)-\frac{1}{1-\delta}\sinh(2u-2\beta)}{1+\cosh(2u+2\beta)-\frac{1}{\delta}\sinh(2u+2\beta)}\,,
\]
or, by using $\exp(2\beta)=\frac{1-\delta}{\delta}$ and denoting
$\psi\dfn e^{2u}$, equivalent to 
\[
\frac{\left[2+\frac{1-\delta}{\delta}\psi^{-1}+\frac{\delta}{(1-\delta)}\psi\right]^{2}}{\left[2+\frac{\delta}{1-\delta}\psi^{-1}+\frac{1-\delta}{\delta}\psi\right]^{2}}=\frac{2+\frac{2-\delta}{\delta}\psi^{-1}-\frac{\delta^{2}}{(1-\delta)^{2}}\psi}{2+\frac{1+\delta}{1-\delta}\psi^{-1}-\frac{(1-\delta)^{2}}{\delta^{2}}\psi}\,.
\]
After further algebraic manipulations, the last display can be shown
to be equivalent to 
\[
(2\delta-1)\cdot\left[\delta\psi+1-\delta\right]\cdot\left[(1-\delta)\psi+\delta\right]\cdot\left[2\delta(1-\delta)\psi^{3}+3\delta(1-\delta)\psi^{2}+(1-2\delta)^{2}\psi-\delta(1-\delta)\right]=0.
\]
As $\delta\in(0,\frac{1}{2})$ and $\psi>0$, the only real solution
to this equation is the solution to 
\[
2\delta(1-\delta)\psi^{3}+3\delta(1-\delta)\psi^{2}+(1-2\delta)^{2}\psi-\delta(1-\delta)=0.
\]
The l.h.s. of the last display is an increasing function of $\psi\in\mathbb{R}_{+}$
with value $-\delta(1-\delta)$ at $\psi=0$ and a strictly positive
value at $\psi=1$. Thus, the above equation has a single real solution
which belongs to $(0,1)$. Hence, $\frac{\d}{\d u}s(u)=0$ has a unique
solution, and this solution is negative.

We next use this property $\frac{\d}{\d u}s(u)$ to show that $s(u)<0$
for all $u\geq0$. As $s(0)=4(2\delta-1)<0$ and $\lim_{u\to\infty}s(u)=0$,
the mean value theorem implies that $\frac{\d}{\d u}s(u)$ must be
strictly positive for some $u>0$. Since $\frac{\d}{\d u}s(u)\neq0$
for $u>0$, $s(u)$ must be increasing for all $u>0$. Since $\lim_{u\to\infty}s(u)=0$
holds, $s(u)\geq0$ is impossible for $u>0$. 
\end{proof}
\begin{proof}[Proof of Proposition \ref{prop: Comparison theorem d=00003D00003D1}]
Let $U\sim N(\eta,1)$. To prove the first property, we write 
\begin{align*}
 & f(\theta\mid\eta,\delta)+f(-\theta\mid\eta,\delta)\\
 & =\E\left[X\cdot\tanh(X\theta+\beta)-X\cdot\tanh(X\theta-\beta)\right]\\
 & =(1-2\delta)\cdot\E\left[U\cdot\tanh(U\theta+\beta)-U\cdot\tanh(U\theta-\beta)\right]\\
 & \geq0
\end{align*}
which holds since $u\mapsto u\cdot\tanh(u\theta+\beta)-u\cdot\tanh(u\theta-\beta)$
is an odd function, that is positive on $\mathbb{R}_{+}$.

To prove the second property, we write the iteration as 
\[
f(\theta\mid\eta,\delta)=\E\left[(1-\delta)\cdot U\tanh(U\theta+\beta)+\delta\cdot U\tanh(U\theta-\beta)\right],
\]
and then analyze its derivative w.r.t. $\delta$ 
\begin{align*}
\frac{\dee f(\theta\mid\eta,\delta)}{\dee\delta} & =\E\left[-U\tanh(U\theta+\beta)+U\tanh(U\theta-\beta)\right]\\
 & \hphantom{=}-\frac{1}{2\delta(1-\delta)}\E\left[\frac{(1-\delta)U}{\cosh^{2}(U\theta+\beta)}-\frac{\delta U}{\cosh^{2}(U\theta-\beta)}\right]\,.
\end{align*}
To prove the required property, we will show that 
\[
\frac{\dee f(\theta\mid\eta,\delta)}{\dee\delta}\;\;\begin{cases}
<0, & \theta<\eta\\
=0, & \eta=\theta\\
>0, & \theta>\eta
\end{cases}
\]
and to this end we split the analysis to the cases $\theta=0$, $\theta<0$
and $\theta\geq0$.

\paragraph*{Case $\theta=0$ }

In this case, trivially, $\left.\frac{\dee f(\theta\mid\eta,\delta)}{\dee\delta}\right|{}_{\theta=0}=-4(1-2\delta)\eta<0$
for $\eta>0$ and $\delta\in(0,\frac{1}{2})$.

\paragraph*{Case $\theta<0$ }

Let 
\[
q_{1}(u)\dfn-u\tanh(u\theta+\beta)+u\tanh(u\theta-\beta)
\]
and note that since $\tanh$ is monotonically increasing and negative
for $t<0$, it holds that $q_{1}(u)$ is an odd function, and $q_{1}(u)<0$
for all $u>0$. Thus $\E[q_{1}(U)]<0$. Also let 
\[
q_{2}(u)\dfn\frac{(1-\delta)u}{\cosh^{2}(u\theta+\beta)}-\frac{\delta u}{\cosh^{2}(u\theta-\beta)}\,.
\]
Since $\cosh$ is an even function with a unique minimum at $t=0$,
it holds that 
\[
q_{2}(u)=u\left[\frac{(1-\delta)}{\cosh^{2}(u\theta+\beta)}-\frac{\delta}{\cosh^{2}(u\theta-\beta)}\right]>0\,,
\]
and 
\begin{align*}
q_{2}(u)+q_{2}(-u) & =\frac{(1-\delta)u}{\cosh^{2}(u\theta+\beta)}-\frac{\delta u}{\cosh^{2}(u\theta-\beta)}-\frac{(1-\delta)u}{\cosh^{2}(u\theta-\beta)}+\frac{\delta u}{\cosh^{2}(u\theta+\beta)}\\
 & =u\left(\frac{1}{\cosh^{2}(u\theta+\beta)}-\frac{1}{\cosh^{2}(u\theta-\beta)}\right)>0
\end{align*}
for any $u>0$. Thus $q_{2}(u)>-q_{2}(-u)$ for $u>0$, and $\E[q_{2}(U)]>0$\textbf{
}(see Appendix \ref{subsec:Usful-results}). The required property
then follows since $\frac{\dee f(\theta\mid\eta,\delta)}{\dee\delta}=\E\left[q_{1}(U)\right]-\frac{1}{2\delta(1-\delta)}\E\left[q_{2}(U)\right]<0$.

\paragraph*{Case $\theta>0$ }

We follow the ideas outlined in the discussion following the statement
of the proposition. We note that $\frac{\dee f(\theta\mid\eta,\delta)}{\dee\delta}=\E\left[U\cdot s(U)\right]$
where $s(u)$ is as defined in (\ref{eq: s function}), and so 
\begin{equation}
\frac{\dee f(\theta\mid\eta,\delta)}{\dee\delta}=\left[\eta\cdot s(\theta\eta)\right]*\varphi(\eta)\label{eq: iteration derivative as Gaussian convolution}
\end{equation}
where $\varphi(\eta)\dfn\frac{1}{\sqrt{2\pi}}\cdot e^{-\eta^{2}/2}$
is the Gaussian kernel, and the convolution is w.r.t. $\eta$.

We begin by proving that $\eta\mapsto\frac{\dee f(\theta\mid\eta,\delta)}{\dee\delta}$
has at least a single zero-crossing in $\mathbb{R}_{+}$ by showing
that for $\eta=0$ and as $\eta\to\infty$: 
\[
\left.\frac{\dee f(\theta\mid\eta,\delta)}{\dee\delta}\right|_{\eta=0}>0,\;\;\frac{\dee f(\theta\mid\eta,\delta)}{\dee\delta}\uparrow0\text{ as }\eta\to\infty.
\]
At $\eta=0$, using the definitions of $q_{1}(u)$ and $q_{2}(u)$,
and recalling that $Z\sim N(0,1)$ 
\begin{align*}
\left.\frac{\dee f(\theta\mid\eta,\delta)}{\dee\delta}\right|_{\eta=0} & =\E\left[q_{1}(Z)\right]-\frac{1}{2\delta(1-\delta)}\E\left[q_{2}(Z)\right]\\
 & \trre[=,a]-\frac{1}{2\delta(1-\delta)}\E\left[q_{2}(Z)\right]\\
 & =-\frac{1}{2\delta(1-\delta)}\E\left[q_{2}(Z)+q_{2}(-Z)\right]\\
 & =-\frac{1}{2\delta(1-\delta)}\E\left[Z\left(\frac{1}{\cosh^{2}(Z\theta+\beta)}-\frac{1}{\cosh^{2}(Z\theta-\beta)}\right)\right]\\
 & \trre[>,b]0
\end{align*}
where $(a)$ is since $q_{1}(u)$ is an odd function, and $(b)$ is
since for $\theta>0$ and any $u\in\mathbb{R}$ 
\[
u\left(\frac{1}{\cosh^{2}(u\theta+\beta)}-\frac{1}{\cosh^{2}(u\theta-\beta)}\right)<0\,.
\]

For $\eta\to\infty$, we note that the first term in the limit of
$\frac{\dee f(\theta\mid\eta,\delta)}{\dee\delta}$ is 
\begin{align*}
 & \lim_{\eta\to\infty}\E\left[-(Z+\eta)\tanh\left((Z+\eta)\theta+\beta\right)+(Z+\eta)\tanh\left((Z+\eta)\theta-\beta\right)\right]\\
 & =\E\left[\lim_{\eta\to\infty}\left(-(Z+\eta)\left[\tanh\left((Z+\eta)\theta+\beta\right)-\tanh\left((Z+\eta)\theta-\beta)\right)\right]\right)\right]
\end{align*}
by dominated convergence theorem. Then, by L'Hôpital 
\begin{align*}
 & \lim_{\eta\to\infty}(z+\eta)\left[\tanh\left((z+\eta)\theta+\beta\right)-\tanh\left((z+\eta)\theta-\beta\right)\right]\\
 & =\lim_{\eta\to\infty}\frac{\left[\tanh\left((z+\eta)\theta+\beta\right)-\tanh\left((z+\eta)\theta-\beta\right)\right]}{(z+\eta)^{-1}}\\
 & =\lim_{\eta\to\infty}\theta\frac{\left[\frac{1}{\cosh^{2}\left((z+\eta)\theta+\beta\right)}-\frac{1}{\cosh^{2}\left((z+\eta)\theta-\beta\right)}\right]}{-(z+\eta)^{-2}}=0
\end{align*}
since $\cosh(t)>e^{t}/2$ for $t>0$. The second term in the limit
of $\frac{\dee f(\theta\mid\eta,\delta)}{\dee\delta}$ can be analyzed
similarly and also equals zero. The fact that the limit of $\frac{\dee f(\theta\mid\eta,\delta)}{\dee\delta}$
to $0$ is from below, can be deduced from $\eta\cdot s(\theta\eta)<0$
for all $\eta>0$ (Lemma \ref{lem: s is always negative}) and the
convolution relation (\ref{eq: iteration derivative as Gaussian convolution}).

Next, we prove that the zero-crossing of $\eta\mapsto\frac{\dee f(\theta\mid\eta,\delta)}{\dee\delta}$
in $\mathbb{R}_{+}$ is unique. The function $\eta\mapsto\eta\cdot s(\theta\eta)$
has a unique zero-crossing at $\eta=0$ since Lemma \ref{lem: s is always negative}
implies that $s(\theta\eta)<0$ for all $\eta>0$ . Furthermore, for
any given $\theta$, $\eta\mapsto\eta\cdot s(\theta\eta)$ is a bounded
function. Indeed, $\eta\cdot s(\theta\eta)$ is clearly bounded for
$|\eta|\leq1$. For $|\eta|>1$, note that using $\tanh(t)=1-\frac{e^{-t}}{\cosh(t)}$,
it holds that for any $\eta>0$ 
\begin{align*}
 & \left|\tanh(\theta\eta+\beta)-\tanh(\theta\eta-\beta)\right|\\
 & =e^{-\theta\eta}\left|\frac{e^{-\beta}}{\cosh(\theta\eta+\beta)}-\frac{e^{\beta}}{\cosh(\theta\eta-\beta)}\right|\\
 & \leq e^{-\theta\eta}\cdot2e^{\beta},
\end{align*}
and analogous result holds for $\eta<0$. Also, using $\cosh(t)\geq1+\frac{t^{2}}{2}$,
it holds that 
\[
\left|\left(\frac{1}{2\delta\cosh^{2}(\theta\eta+\beta)}-\frac{1}{2(1-\delta)\cosh^{2}(\theta\eta-\beta)}\right)\right|\leq\frac{1}{2\delta\left(1+\frac{\theta\eta+\beta}{2}\right)^{2}}+\frac{1}{2(1-\delta)\left(1+\frac{\theta\eta-\beta}{2}\right)^{2}}\,.
\]
Hence, $\left|\eta\cdot s(\theta\eta)\right|\leq|\eta|\cdot[e^{-\theta|\eta|}\cdot2e^{\beta}+\frac{1}{\delta\theta^{2}\eta^{2}}]$
which is bounded for all $|\eta|>1$.

The variation diminishing property of the Gaussian kernel (Proposition
\ref{prop: Gaussian variation diminishing}, Appendix \ref{subsec:Totally-positive-kernels}),
and the convolution relation (\ref{eq: iteration derivative as Gaussian convolution})
imply that $\frac{\dee f(\theta\mid\eta,\delta)}{\dee\delta}$ has
at most a single zero-crossing as a function of $\eta$. From all
the above, $\frac{\dee f(\theta\mid\eta,\delta)}{\dee\delta}$ has
exactly a single zero-crossing for some $\eta>0$. The consistency
property implies that $\left.\frac{\dee f(\theta\mid\eta,\delta)}{\dee\delta}\right|_{\theta=\eta}=0$,
and so this zero-crossing must occur at $\eta=\theta$. From this,
(\ref{eq: derivative of the iteration wrt delta for positive theta})
follows. 
\end{proof}
\begin{proof}[Proof of theorem \ref{thm: population d=00003D00003D1 known delta}]
Recall that $\pm\eta$ and $0$ are the only fixed points of the
balanced iteration $f(\theta,\frac{1}{2}\mid\eta,\frac{1}{2})$, and
that $f(|\theta|,\frac{1}{2}\mid\eta,\frac{1}{2})>|\theta|$ for $0<|\theta|<\eta$,
and $f(|\theta|,\frac{1}{2}\mid\eta,\frac{1}{2})<|\theta|$ for $|\theta|>\eta$.
The claim then follows from Proposition \ref{prop: Comparison theorem d=00003D00003D1},
item \ref{enu: one-dimensioanl iteration delta derivative dominance}.
The last two claims follow from Proposition \ref{prop: Properties of one dimensional iterations},
items \ref{enu: general convergence in one-dim - convergence of monotonic}
and \ref{enu: general convergence in one-dim - convergence dominance}. 
\end{proof}

\subsection{Empirical iteration }
\begin{proof}[Proof of Theorem \ref{thm: empirical d=00003D00003D1 known delta}]
We analyze the empirical iteration $f_{n}(\theta)\equiv f_{n}(\theta,\delta\mid\eta,\delta)$.
From Lemma \ref{lem: d=00003D00003D1 simple properties} and assuming
the high probability event (\ref{eq: high probability event}), it
holds that 
\[
|f_{n}(\theta)|\leq\gl[C]_{\theta}+1+\omega\leq\gl[C]_{\theta}+2
\]
for all $n$ sufficiently large, and that 
\begin{equation}
f_{n}(\theta)\geq f_{-}(\theta)\dfn f(\theta)-\max\{|\theta|,\rho\}\cdot\omega\,,\label{eq: lower envelope d=00003D00003D1}
\end{equation}
and 
\begin{equation}
f_{n}(\theta)\leq f_{+}(\theta)\dfn f(\theta)+\max\{|\theta|,\rho\}\cdot\omega\,,\label{eq: upper envelope d=00003D00003D1}
\end{equation}
where we abbreviate here $\omega\equiv\omega_{1}=\sqrt{\gl[C]_{\omega}\frac{\log n}{n}}$
and $f_{\pm}(\theta)$ will be referred to as the lower ($-$) and
upper ($+$) envelopes. We consider the empirical iteration $\theta_{t+1}=f_{n}(\theta_{t})$
as well as the lower and upper envelopes iterations $\theta_{t+1}^{\pm}=f_{\pm}(\theta_{t}^{\pm})$,
all which are initialized at the same point, to wit, $\theta_{0}=\theta_{0}^{\pm}$.
We begin by thoroughly analyzing the initialization $\theta_{0}=0$
and then briefly discuss the initialization $\theta_{0}=\frac{1}{\rho}\E_{n}[X]$
(which is similar and simpler). In the first step of the proof, we
show that $\{\theta_{t}\}$ and $\{\theta_{t}^{\pm}\}$ all converge
monotonically to fixed points. In the second step, we analyze the
convergence time and the distance between the fixed points. We split
the analysis into three different regimes for $\eta$.

\paragraph*{Fixed points}

We show that $\{\theta_{t}\}$ and $\{\theta_{t}^{\pm}\}$ converge
monotonically to fixed points, which we denote, respectively, by $\eta_{n}$
and $\eta_{\pm}$. We use several intuitive properties of convergence
of one-dimensional iterations, which are formally stated and proved
in Proposition \ref{prop: Properties of one dimensional iterations},
items \ref{enu: general convergence in one-dim - existence of a fixed point}
and \ref{enu: general convergence in one-dim - convergence of monotonic}.

For the empirical iteration $f_{n}(\theta)$, since 
\[
f'_{n}(\theta)=\E_{n}\left[\frac{X^{2}}{\cosh^{2}(X\theta+\beta)}\right]>0
\]
and $f_{n}(\theta)$ is bounded by assumption, $\{\theta_{t}\}$ converges
monotonically to a fixed point $\eta_{n}$, and is either increasing
or decreasing according to the sign of $f_{n}(0)$.

For the upper envelope $f_{+}(\theta)$, recall from Lemma \ref{lem: d=00003D00003D1 simple properties}
that $f(\theta)$ is increasing and bounded. So $\lim_{\theta\to\infty}f'(\theta)=0$.
Hence, $f_{+}(\theta)$ is increasing, and for $n>n_{0}(\gl[C]_{\omega})$
it holds that $\lim_{\theta\to\infty}f'_{+}(\theta)<1$. Thus $\{\theta_{t}^{+}\}$
is increasing and converges to a fixed point $\eta_{+}$.

For the lower envelope $f_{-}(\theta)$, first note that there exists
$n_{1}(\gl[C]_{\theta},\gl[C]_{\rho})$ such that $f_{-}(\theta)$
is increasing for all $\theta\in[-\gl[C]_{\theta},\gl[C]_{\theta}]$
since 
\begin{equation}
f_{-}'(\theta)\geq f'(\theta)-\omega\geq\min_{0\leq\eta\leq\gl[C]_{\theta}}\E\left[\frac{X^{2}}{\cosh^{2}(|X|\gl[C]_{\theta}+\beta)}\right]-\omega>0\,.\label{eq: one dimensional lower envelops is increasing -- condition}
\end{equation}
If $f_{-}(0)>0$ then since $f(\theta)$ has a unique fixed point
$\eta$ in $[0,\infty)$, $f_{-}(\theta)$ must have a fixed point
in $[0,\eta]$, and no fixed points in $[\eta,\infty)$, and $\{\theta_{t}^{-}\}$
is increasing to one of the fixed points in $[0,\eta)$. If $f_{-}(0)<0$
then as the negative fixed points of $f(\theta)$ are confined to
$[-\eta,0]$ (Theorem \ref{thm: population d=00003D00003D1 known delta})
similar reasoning as for the upper envelope leads to the conclusion
that $\{\theta_{t}^{-}\}$ is decreasing and converges to some fixed
point $\eta_{-}<0$. Furthermore, since the negative fixed points
of $f(\theta)$ are confined to $[-\eta,0]$ (Theorem \ref{thm: population d=00003D00003D1 known delta})
and since $f(\theta)\geq-f(-\theta)$ for all $\theta>0$ (Proposition
\ref{prop: Comparison theorem d=00003D00003D1}, item \ref{enu: one-dimensioanl iteration oddness dominance})
the minimal negative fixed point $\underline{\eta}_{-}<0$ satisfies
$|\underline{\eta}_{-}|\leq|\eta_{+}|$.

\paragraph*{Stochastic error and convergence time}

We now prove bounds on the stochastic error and on the required number
of iterations for convergence. We will use constants $C_{1},C_{2},C_{3}>0$
which satisfy relations that will be specified throughout the proof.
Assume that $\rho\geq C_{1}\sqrt{\omega}$. We split the analysis
to three regimes for $\eta$ given by $[0,\frac{\omega}{\rho}]$,
$[\frac{\omega}{\rho},C_{2}\rho]$ and $[C_{2}\rho,\gl[C]_{\theta}]$
where $C_{1}\geq\sqrt{1/C_{2}}$ is assumed so that these are three
non-empty intervals. For simplicity, we assume that $C_{2}\leq1$
(its value will eventually be chosen to be sufficiently small).

\uline{Case 1 - }%
\mbox{%
$\eta\in[\frac{\omega}{\rho},C_{2}\rho]$%
}\uline{:} For $\theta\geq\eta$, Lemma \ref{lem: d=00003D00003D1 simple properties}
implies that 
\[
f'_{+}(\theta)\leq1-\frac{\rho^{2}}{4}+\omega\,.
\]
Thus, assuming\textbf{ $C_{1}\geq\sqrt{12}$ }then $f'_{+}(\theta)\leq1-\frac{\rho^{2}}{6}.$
For $0\leq\theta<\eta$, we have 
\[
\left|f''(\theta)\right|\leq\gl[C]''\cdot\max\{\eta,\rho\}\leq\gl[C]''\rho\,,
\]
and using $f'_{+}(\theta)=f'_{+}(\eta)-\int_{\theta}^{\eta}f''(\tilde{\theta})\d\tilde{\theta}$,
it holds that 
\begin{align*}
f'_{+}(\theta) & \leq1-\frac{\rho^{2}}{4}+\gl[C]''\rho(\eta-\theta)+\omega\\
 & \leq1-\frac{\rho^{2}}{4}+\gl[C]''\eta\rho+\omega\\
 & \leq1-\frac{\rho^{2}}{4}+\gl[C]''C_{2}\rho^{2}+\omega\,.
\end{align*}
Assuming that $C_{1}\geq\sqrt{12}$ and $C_{2}\leq\frac{1}{12\cdot\gl[C]''}$
we get that $f'_{+}(\theta)\leq1-\frac{\rho^{2}}{12}$. Thus, $f_{+}(\theta)$
is a contraction for $\theta\in[0,\infty)$. Hence, so is $f_{-}(\theta)$
(as $0\leq f'_{-}(\theta)\leq f'_{+}(\theta)$). Furthermore, it holds
that $f_{+}(0)>f_{-}(0)=f(0)-\rho\omega>0$ for all $n\geq n_{2}(\gl[C]_{\theta},\gl[C]_{\rho})$.
Thus both $\theta_{t}^{\pm}$ are increasing and converge to fixed
points $\eta_{\pm}>0$ where $\eta_{+}\geq\eta\geq\eta_{-}$ and satisfy
$\eta_{\pm}-\theta_{t}^{\pm}\leq\eta_{\pm}(1-\frac{\rho^{2}}{12})^{t}$
(Proposition \ref{prop: Properties of one dimensional iterations},
item \ref{enu: general convergence in one-dim - convergence time of contraction}).
We next analyze the errors $\epsilon_{-}=\eta-\eta_{-}>0$ and $\epsilon_{+}=\eta_{+}-\eta>0$.
For the error of the lower envelope, let $\theta\in[\eta_{-},\eta]$
and recall that $\eta_{-}\leq\eta\leq C_{2}\rho$. Then, $f''(\theta)\geq-\gl[C]''\rho$
and so 
\[
f'(\eta)\geq f'(\theta)-\gl[C]''\rho(\eta-\theta)\,,
\]
as well as $f_{-}(\theta)>f(\theta)-\rho\omega.$ Hence 
\begin{align*}
\eta-\rho\omega & =f_{-}(\eta)\\
 & =f_{-}(\eta_{-})+\int_{\eta_{-}}^{\eta}f_{-}'(\theta)\cdot\d\theta\\
 & \leq\eta_{-}+\int_{\eta_{-}}^{\eta}f'(\theta)\cdot\d\theta\\
 & \leq\eta_{-}+f'(\eta)\epsilon_{-}+\gl[C]''\rho\left(\eta\epsilon_{-}-\frac{\eta^{2}-\eta_{-}^{2}}{2}\right)\\
 & \leq\eta_{-}+f'(\eta)\epsilon_{-}+\gl[C]''C_{2}\rho^{2}\epsilon_{-}\,.
\end{align*}
The above implies $\epsilon_{-}(1-f'(\eta)-C_{2}\rho^{2})\leq\rho\omega$
and since $f'(\eta)\leq1-\frac{\rho^{2}}{4}$ then if $C_{2}\leq\frac{1}{8}$
we obtain $\eta-\eta_{-}\leq8\frac{\omega}{\rho}$. Since $\eta_{-}-\theta_{t}^{-}\leq\eta_{-}(1-\frac{\rho^{2}}{12})^{t}\leq\gl[C]_{\theta}(1-\frac{\rho^{2}}{12})^{t}$
it holds that $\eta_{-}-\theta_{t}^{-}\leq8\frac{\omega}{\rho}$ for
all $t\geq\frac{12}{\rho^{2}}\log\left(\frac{\gl[C]_{\theta}}{8}\cdot\frac{\rho}{\omega}\right)$
(Proposition \ref{prop: Properties of one dimensional iterations},
item \ref{enu: general convergence in one-dim - convergence time of contraction})
and so also $\eta-\theta_{t}^{-}\leq16\frac{\omega}{\rho}$. The analysis
for $\eta_{+}$ is similar: 
\begin{align*}
\eta_{+} & =f_{+}(\eta_{+})\\
 & =f_{+}(\eta)+\int_{\eta}^{\eta_{+}}f_{+}'(\theta)\cdot\d\theta\\
 & =\eta+\rho\omega+\epsilon_{+}\left(f'(\eta)+\omega\right)\,.
\end{align*}
Then, since $\eta\leq C_{2}\rho$ and $f'(\theta)\leq1-\frac{\rho^{2}}{4}$
it holds that $\epsilon_{+}(\frac{\rho^{2}}{4}-\omega)\leq\rho\omega$
for all $\theta>\eta$. Since $\rho\geq C_{1}\sqrt{\omega}$, assuming
$C_{1}\geq\sqrt{8}$ we get $\eta_{+}-\eta\leq8\frac{\omega}{\rho}$.
Furthermore, for all $n\geq n_{3}(\gl[C]_{\omega})$, $\eta_{+}\leq2\eta$
and so for all $t\geq\frac{12}{\rho^{2}}\log\left(\frac{\gl[C]_{\theta}}{4}\frac{\rho}{\omega}\right)$
it holds that $|\eta-\theta_{t}^{+}|\leq16\frac{\omega}{\rho}$.

\uline{Case 2 - }%
\mbox{%
$\eta\in[C_{2}\rho,\gl[C]_{\theta}]$%
}\uline{:} The convergence has two phases. The time spent in phase
$1$ (resp. phase $2$) until the required convergence is assured
will be denoted by $T_{1}$ (resp. $T_{2}$). 
\begin{enumerate}
\item We show that there exists \textbf{$C_{3}\leq\frac{1}{2}$} sufficiently
small and $C_{1}$ sufficiently large such that $f_{-}(\theta)>\theta+\frac{1}{6}\rho^{2}\eta$
holds for all $\theta\in[0,C_{3}\rho]$ (note that $\rho>2\theta$
is assured). In turn, this inequality is satisfied if 
\begin{equation}
f(\theta)>\theta+\omega(\theta+\rho)+\frac{1}{6}\rho^{2}\eta\label{eq: Case 2 one-dim empirical convergence, required inequality}
\end{equation}
holds. Since 
\begin{align*}
f'(\theta) & =f'(0)+\int_{0}^{\theta}f''(\tilde{\theta})\cdot\d\tilde{\theta}\\
 & \geq(1-\rho^{2})(1+\eta^{2})-\frac{\gl[C]''}{2}\theta^{2}-\gl[C]''\rho\theta\,,
\end{align*}
and as $f(0)=\rho^{2}\eta$, we get that 
\begin{align*}
f(\theta) & =f(0)+\int_{0}^{\theta}f'(\tilde{\theta})\cdot\d\tilde{\theta}\\
 & \geq\rho^{2}\eta+(1-\rho^{2})(1+\eta^{2})\theta-\frac{\gl[C]''}{6}\theta^{3}-\frac{\gl[C]''}{2}\rho\theta^{2}\,.
\end{align*}
Thus, (\ref{eq: Case 2 one-dim empirical convergence, required inequality})
is satisfied if the following inequality holds: 
\begin{equation}
\frac{5}{6}\rho^{2}\eta+(1-\rho^{2})\eta^{2}\theta>\rho^{2}\theta+\frac{\gl[C]''}{6}\theta^{3}+\frac{\gl[C]''}{2}\rho\theta^{2}+\omega\theta+\omega\rho.\label{eq: Case 2 one-dim empirical convergence, required inequality - derivation}
\end{equation}
Since clearly $(1-\rho^{2})\eta^{2}\theta>0$, this inequality can
be assured to hold for all $n>n_{4}(\gl[C]_{\omega},C_{1})$ large
enough, by proper choice of constants. Specifically, by ``allocating''
$\frac{1}{6}\rho^{2}\eta$ for each of the five additive terms on
the r.h.s. of (\ref{eq: Case 2 one-dim empirical convergence, required inequality - derivation}),
and using the assumption $\rho>C_{1}\sqrt{\omega}$, the inequality
(\ref{eq: Case 2 one-dim empirical convergence, required inequality - derivation})
is satisfied for $C_{3}\leq\min\left\{ \frac{C_{2}}{6},\frac{4C_{2}}{\gl[C]''},\frac{2C_{2}}{3\gl[C]''}\right\} $
(for the first three terms) and $C_{1}\geq\max\left\{ \sqrt{\frac{6C_{3}}{C_{2}}},\sqrt{\frac{6}{C_{2}}}\right\} $
(for the fourth and fifth terms). Therefore, as long as $\theta_{t}\leq C_{3}\rho$,
it holds that $\theta_{t+1}-\theta_{t}\geq\frac{1}{6}\rho^{2}\eta$.
So, initializing from $\theta_{0}=0$, it holds that $\theta_{t}^{-}\geq C_{3}\rho$
for all $t\geq T_{1}=\frac{6C_{3}}{\rho\eta}$ where $T_{1}\leq\frac{1}{\omega}$.
Naturally, this holds for $\theta_{t}^{+}$ too. 
\item At this phase, the convergence behaves similarly to the balanced case.
Specifically, as $C_{3}\leq1$ and since $\theta>C_{3}\rho$ then
\[
f_{+}(\theta)\leq f(\theta)+\frac{1}{C_{3}}\theta\omega
\]
and 
\[
f_{-}(\theta)\geq f(\theta)+\frac{1}{C_{3}}\theta\omega\,.
\]
Using Theorem \ref{thm: population d=00003D00003D1 known delta},
the convergence of the envelopes with $f(\theta)\pm\frac{1}{C_{3}}\theta\omega$
is faster than the convergence of the envelopes of the balanced iterations
$f(\theta,\frac{1}{2}\mid\eta,\frac{1}{2})\pm\frac{1}{C_{3}}\theta\omega$.
Thus, the one-dimensional analysis and result of \cite[Theorem 3]{wu2019EM}
holds. Specifically, \cite[Theorem 3]{wu2019EM} demonstrates the
existence of constants $\{c_{i}\}$ such that if the balanced EM is
initialized at $\theta_{0}=C_{3}\rho$, assuming that $\eta>c_{4}\sqrt{\omega}$
it holds that $|\eta_{\pm}-\eta|\leq c_{2}\frac{\omega}{\eta}$ for
all $t\geq T_{2}=\frac{c_{3}}{\eta^{2}}\log\left(\frac{1}{C_{3}\rho\omega}\right)$.
The condition $\eta>c_{4}\sqrt{\omega}$ is satisfied by requiring
that $C_{1}>\frac{c_{4}}{C_{2}}$. 
\end{enumerate}
\uline{Case 3 - }%
\mbox{%
$\eta\in[0,\frac{\omega}{\rho}]$%
}\uline{:} For the upper envelope, as in case 2, $f_{+}(\theta)$
is increasing and satisfies $f_{+}(0)>0$ and thus $\theta_{t}^{+}$
is increasing and converges to a fixed point $\eta_{+}>0$. In addition,
similar analysis shows that $\eta_{+}-\eta\leq8\frac{\omega}{\rho}$.
Thus, for all $t>1$ 
\[
|\theta_{t}^{+}-\eta|\leq|\theta_{t}^{+}|+|\eta|\leq|\eta_{+}|+|\eta|\leq10\frac{\omega}{\rho}\,.
\]
So, the error is $O(\frac{\omega}{\rho})$ for all iterations. For
the lower envelope, by the assumption on $n>n_{1}$ in (\ref{eq: one dimensional lower envelops is increasing -- condition}),
$f_{-}(\theta)$ is increasing for $\theta\in\mathbb{B}(\gl[C]_{\theta})$.
If $f_{-}(0)>0$ then $\theta_{t}^{-}$ will converge to a fixed point
$0<\eta_{-}\leq\eta_{+}$ and similar analysis as for the upper envelope
implies that $|\theta_{t}^{-}-\eta|\leq10\frac{\omega}{\rho}$ for
all $t>1$. Otherwise, if $f_{-}(0)<0$, $\{\theta_{t}^{-}\}$ is
decreasing and converges to a fixed point $\eta_{-}$. As was shown
in the analysis of the fixed points, it must hold that $|\eta_{-}|\leq|\eta_{+}|$
and so the last bound on $|\theta_{t}^{-}-\eta|$ is valid in this
case too.

The result for $\theta_{0}=0$ then follows from summarizing all three
cases, and determining the constants in the following order: $C_{2}$
to be sufficiently small, then $C_{3}$ sufficiently small, and finally
$C_{1}$ sufficiently large.

We now discuss the case $\theta_{0}=\frac{1}{\rho}\E_{n}[X]$. Since
$f(0)=\rho^{2}\eta$ and $f_{n}(0)=\rho\E_{n}[X]$, and under the
high probability event $|f_{n}(0)-f(0)|\leq\omega\rho$ we have that
$|\theta_{0}-\eta|\leq\frac{\omega}{\rho}$ which is in fact already
within the error rate obtained in Cases 1 and 3 above. By repeating
the same arguments in those cases, the error is $O(\frac{\omega}{\rho})$
for all subsequent iterations. In Case 2, the first phase is unnecessary
since in this case $\eta>C_{2}\rho$ and so $\eta-\frac{\omega}{\rho}>C_{3}\rho$
as long as $C_{3}<\frac{C_{2}}{2}$ and $C_{1}\geq\sqrt{\frac{1}{C_{2}}}$.
Thus, if $\theta_{0}\leq\eta$ only the second phase in the analysis
above occurs. If $\theta_{0}>\eta$ then the analysis of the balanced
iteration \cite[Theorem 3]{wu2019EM} is similarly intact. 
\end{proof}

\section{Proofs for Section \ref{subsec:The-mean-iteration d>1} }

\subsection{Population iteration}

The next lemma summarizes basic properties of $F(\cdot)$ and $G(\cdot)$
which are useful for the unbalanced iteration analysis. 
\begin{lem}[Properties of $F$ and $G$ as functions of $(a,b)$]
\label{lem: properties of F and G} Assume that $b\geq0$ and $\delta\in[0,\frac{1}{2})$.
Then: 
\begin{enumerate}
\item \label{enu: simple properties d>1 monotonicity}Monotonicity: $a\mapsto F(a,b)$
and $b\mapsto G(a,b)$ are monotonically increasing functions. 
\item \label{enu: simple properties d>1 positivity}Positivity: $F(a,b)>0$
for $a\geq0$, and $G(a,b)\geq0=G(a,0)$. 
\item \label{enu: simple properties d>1 strict positivity}Strict positivity
of $F(0,b)$: For any $C_{f}>0$ there exists $\gl[C]_{F,0}>0$ which
depends on $(C_{f},\gl[C]_{\beta})$ such that $\min_{b\in[0,C_{f}]}F(0,b)\geq\gl[C]_{F,0}\rho^{2}\eta$. 
\item \label{enu: simple properties d>1 boundedness}Boundedness: $|F(a,b)|\leq\eta+1$
and $G(a,b)\leq\eta+1$. 
\item \label{enu: simple properties d>1 upper bound on first derivatives}Upper
bounded first derivatives: 
\[
\left|\frac{\dee F(a,b)}{\dee a}\right|\leq1+\eta^{2},\;\left|\frac{\dee F(a,b)}{\dee b}\right|\leq\sqrt{1+\eta^{2}},\;\left|\frac{\dee G(a,b)}{\dee a}\right|\leq\sqrt{1+\eta^{2}},\;\left|\frac{\dee G(a,b)}{\dee b}\right|\leq1.
\]
\item \label{enu: simple properties d>1 lower bound on F first derivatives}Lower
bounded first derivative: Let $C_{f}>0$ be given. There exists $\gl[C]''_{F}(\gl[C]_{\theta},C_{f},\gl[C]_{\beta})>0$
such that 
\[
\min_{a,b\in[0,C_{f}]^{2}}\frac{\dee F(a,b)}{\dee a}\geq\gl[C]''_{F}\,.
\]
\item Derivative at $b=0$: For $a\in[0,\eta]$ 
\[
\left.\frac{\dee G(a,b)}{\dee b}\right|_{b=0}\leq4\delta(1-\delta)
\]
and for $a\geq\eta$ 
\[
\left.\frac{\dee G(a,b)}{\dee b}\right|_{b=0}\leq2\sqrt{\delta(1-\delta)}e^{-\frac{1}{2}\eta^{2}}\leq1-\frac{1}{4}\cdot\max\{\min\{\eta^{2},1\},\rho^{2}\}\,.
\]
\item \label{enu: simple properties d>1 F mixed derivative at b=00003D00003D0}Crossed
derivative at $b=0$: 
\[
\left.\frac{\dee F(a,b)}{\dee b}\right|_{b=0}=0\,.
\]
\item \label{enu: simple properties d>1 upperbound on F second mixed derivative at b=00003D00003D0}Upper
bounded crossed second order derivatives at $b=0$: 
\[
\left|\left.\frac{\dee^{2}F(a,b)}{\dee b^{2}}\right|_{b=0}\right|\leq a(1+\eta^{2})+\rho\gl[\overline{C}]_{\beta}(1+\eta)
\]
and the same bound holds for $\frac{\dee^{2}G(a,b)}{\dee b\dee a}=\frac{\dee^{2}F(a,b)}{\dee b^{2}}$. 
\item \label{enu: simple properties d>1 upperbound on F third mixed derivative at b=00003D00003D0}Upper
bounded crossed third order derivatives: There exists $\gl[C]'''_{F}(\gl[C]_{\theta})>0$
such that 
\[
\left|\frac{\dee^{3}F(a,b)}{\dee b^{3}}\right|\leq\gl[C]'''_{F}\,.
\]
\end{enumerate}
\end{lem}

\begin{proof}
~ 
\begin{enumerate}
\item We have 
\begin{align*}
\frac{\dee F(a,b)}{\dee a} & =\E\left[\frac{V^{2}}{\cosh^{2}(aV+bW+\beta)}\right]>0\,,
\end{align*}
and similarly, 
\begin{align*}
\frac{\dee G(a,b)}{\dee b} & =\E\left[\frac{W^{2}}{\cosh^{2}(aV+bW+\beta)}\right]>0\,.
\end{align*}
\item $F(a,b)\geq0$ for $a\geq0$ since $a\mapsto F(a,b)$ is increasing
and 
\begin{equation}
F(0,b)=\E[V]\cdot\E[\tanh(bW+\beta)]>0\label{eq: F at zero a}
\end{equation}
where the inequality is because $\E[V]=(1-2\delta)\eta>0$ and since
$\tanh$ is odd and increasing and $W$ is symmetric. Similarly, $G(a,b)\geq0=G(a,0)=\E[W]\cdot\E[\tanh(aV+\beta)]=0$
because $b\mapsto G(a,b)$ is increasing. 
\item The minimal value of $\min_{b\in[0,C_{f}]}\E\left[\tanh(bW+\beta)\right]$
is obtained for $b=C_{f}$ as 
\begin{align}
\frac{\dee\E\left[\tanh(bW+\beta)\right]}{\dee b} & =\E\left[\frac{W}{\cosh(bW+\beta)^{2}}\right]\nonumber \\
 & =\E\left[\frac{W}{\cosh(bW+\beta)^{2}}-\frac{W}{\cosh(bW-\beta)^{2}}\mid W>0\right]\nonumber \\
 & <0\,.\label{eq: derivative of tanh wrt to b analysis}
\end{align}
We next analyze the minimal value $q(\beta)\dfn\E[\tanh(C_{f}W+\beta)]$.
It holds that $q(0)=\E[\tanh(C_{f}W)]=0$ and that 
\[
q'(\beta)\dfn\frac{\d q(\beta)}{\d\beta}=\E\left[\frac{1}{\cosh^{2}(C_{f}W+\beta)}\right]
\]
and so $q'(0)=\E[\frac{1}{\cosh^{2}(C_{f}W)}]>0$ and only depends
on $C_{f}$. Also, 
\begin{align*}
q''(\beta)\dfn\frac{\d^{2}q(\beta)}{\d\beta^{2}} & =\E\left[\frac{-2\tanh(C_{f}W+\beta)}{\cosh^{2}(C_{f}W+\beta)}\right]\\
 & =\frac{1}{C_{f}}\E\left[\frac{W}{\cosh^{2}(C_{f}W+\beta)}\right]\\
 & <0
\end{align*}
by Stein's identity (see (\ref{eq: Stein lemma}) in Appendix \ref{subsec:Usful-results}),
and where the inequality is as in (\ref{eq: derivative of tanh wrt to b analysis}).
Hence, $\beta\mapsto q(\beta)$ is a concave function at $\beta\in\mathbb{R}_{+}$,
and so for all $\beta\in[0,\gl[C]_{\beta}]$, $q(\beta)$ is lower
bounded by the straight line connecting $q(0)$ and $q(\gl[C]_{\beta})$,
to wit, 
\[
\E[\tanh(C_{f}W+\beta)]=q(\beta)\geq\frac{q(\gl[C]_{\beta})}{\gl[C]_{\beta}}\beta=C_{1}\beta\geq C_{1}\gl[\underline{C}]_{\beta}\rho\,.
\]
The claim then follows from (\ref{eq: F at zero a}) and $\E[V]=\rho\eta>0$. 
\item $|F(a,b)|\leq\E\left[|V\tanh(aV+bW+\beta)|\right]\leq\E|V|\leq\eta+\sqrt{2/\pi}$,
and $|G(a,b)|\leq\E\left[|W\tanh(aV+bW+\beta)|\right]\leq\E|W|\leq\sqrt{2/\pi}$. 
\item We only show 
\[
\left|\frac{\dee F(a,b)}{\dee b}\right|\leq\E\left[\frac{|WV|}{\cosh^{2}(aV+bW+\beta)}\right]\leq\E\left[|WV|\right]\leq\sqrt{\E\left[W^{2}\right]\E\left[V^{2}\right]}
\]
since $\cosh(t)\geq1$. The other bounds are proved similarly. 
\item Let $U\sim N(\eta,1)$, $U\ind W$. Then, for any $a,b\in[0,C_{f}]^{2}$
\begin{align*}
\frac{\dee F(a,b)}{\dee a} & =\E\left[\frac{V^{2}}{\cosh^{2}(aV+bW+\beta)}\right]\\
 & =\E\left[\frac{2V^{2}}{1+\cosh(2aV+2bW+2\beta)}\right]\\
 & =\E\left[\frac{2(1-\delta)U^{2}}{1+\cosh(2aU+2bW+2\beta)}+\frac{2\delta U^{2}}{1+\cosh(2aU-2bW-2\beta)}\right]\\
 & \trre[\geq,a]\E\left[\frac{(1-\delta)U^{2}}{\cosh(2aU+2bW+2\beta)}+\frac{\delta U^{2}}{\cosh(2aU-2bW-2\beta)}\right]\\
 & \trre[\geq,b]\E\left[\frac{U^{2}}{\cosh(2aU+2bW+2\beta)\cdot\cosh(2aU-2bW-2\beta)}\right]\\
 & =\E\left[\frac{2U^{2}}{\cosh(4aU)+\cosh(4bW+4\beta)}\right]\\
 & \trre[\geq,c]\E\left[\frac{2U^{2}}{\cosh(4C_{f}U)+\cosh(4bW+4\beta)}\right]\\
 & \trre[\geq,d]\E\left[\frac{4U^{2}}{2\cosh(4C_{f}U)+\exp\left(4b\left|W\right|+4\beta\right)}\right]\\
 & \geq\E\left[\frac{2U^{2}}{2\cosh(4C_{f}U)+\exp\left(4C_{f}\left|W\right|+4\beta\right)}\right]\\
 & \dfn\gl[C]''_{F}>0
\end{align*}
where $(a)$ and $(b)$ are since $\cosh(t)\geq1$, $(c)$ is since
$\cosh(t)$ is an even function, and increasing for $t\geq0$, $(d)$
is since $\cosh(t)\geq\frac{1}{2}e^{|t|}$. 
\item We have 
\[
\frac{\dee G(a,b)}{\dee b}=\E\left[\frac{W^{2}}{\cosh(aV+bW+\beta)^{2}}\right]
\]
and 
\begin{align*}
\left.\frac{\dee G(a,b)}{\dee b}\right|_{b=0} & =\E\left[\frac{W^{2}}{\cosh(aV+\beta)^{2}}\right]\\
 & =\E\left[\frac{1}{\cosh(aV+\beta)^{2}}\right]\\
 & =4\delta(1-\delta)\E\left[\frac{1}{\left((1-\delta)e^{aV}+\delta e^{-aV}\right)^{2}}\right]\,.
\end{align*}
At this point, for $a>\eta$, the proof follows the same steps of
the proof of Lemma \ref{lem: d=00003D00003D1 simple properties},
item \ref{enu: d=00003D00003D1 simple properties - first derivative}
and thus omitted. For $a\in[0,\eta]$ we let $Z\sim N(0,1)$ 
\begin{align*}
\left.\frac{\dee G(a,b)}{\dee b}\right|_{b=0} & \trre[=,a]4\delta(1-\delta)e^{-\eta^{2}/2}\cdot\E\left[\frac{(1-\delta)e^{\eta Z}+\delta e^{-\eta Z}}{\left((1-\delta)e^{aV}+\delta e^{-aV}\right)^{2}}\right]\\
 & \trre[=,b]4\delta(1-\delta)e^{-\eta^{2}/2}\E\left[\frac{1}{2}\frac{(1-\delta)e^{\eta Z}+\delta e^{-\eta Z}}{\left((1-\delta)e^{aV}+\delta e^{-aV}\right)^{2}}+\frac{1}{2}\frac{(1-\delta)e^{-\eta Z}+\delta e^{\eta Z}}{\left((1-\delta)e^{-aV}+\delta e^{aV}\right)^{2}}\right]\\
 & \trre[\leq,c]4\delta(1-\delta)e^{-\eta^{2}/2}\cdot\E\left[e^{\eta Z}\right]\\
 & =4\delta(1-\delta)\,,
\end{align*}
where $(a)$ is by a change of measure (see (\ref{eq: change of measure})
in Appendix \ref{subsec:Usful-results}) and $(b)$ is by the symmetry
of $N(0,1)$. The inequality $(c)$ can be proved pointwise as follows:
We denote $\psi_{\eta}=e^{\eta Z}$ and $\psi_{a}=e^{aZ}$, and may
assume that $Z>0$ if we show it holds for all $\delta\in[0,1]$.
Thus, it remains to show that 
\begin{equation}
\max_{\psi_{\eta}\geq\psi_{a}>1}\left[\frac{(1-\delta)\psi_{\eta}+\delta\psi_{\eta}^{-1}}{\left((1-\delta)\psi_{a}+\delta\psi_{a}^{-1}\right)^{2}}+\frac{(1-\delta)\psi_{\eta}^{-1}+\delta\psi_{\eta}}{\left((1-\delta)\psi_{a}^{-1}+\delta\psi_{a}\right)^{2}}-2\psi_{\eta}\right]\leq0\,.\label{eq: d>1 simple properties inequality to prove}
\end{equation}
We prove this inequality by showing that it holds for $\psi_{\eta}=\psi_{a}$
and then show that the term in the l.h.s. of (\ref{eq: d>1 simple properties inequality to prove})
is non-increasing in $\psi_{\eta}$ for $\psi_{\eta}\in(\psi_{a},\infty)$.
We first verify the inequality (\ref{eq: d>1 simple properties inequality to prove})
for $\psi_{\eta}=\psi_{a}$. In this case 
\begin{align*}
 & \max_{\psi_{a}>1}\left[\frac{1}{(1-\delta)\psi_{a}+\delta\psi_{a}^{-1}}+\frac{1}{(1-\delta)\psi_{a}^{-1}+\delta\psi_{a}}-2\psi_{a}\right]\\
 & =\max_{\psi_{a}>1}\psi_{a}\cdot\left[\frac{1}{(1-\delta)\psi_{a}^{2}+\delta}+\frac{1}{(1-\delta)+\delta\psi_{a}^{2}}-2\right]
\end{align*}
and the term in brackets is non-positive (its maximal value is $0$
obtained by $\psi_{a}=1$). To prove that the l.h.s. of (\ref{eq: d>1 simple properties inequality to prove})
is monotonic w.r.t. $\psi_{\eta}$, we next differentiate w.r.t. to
$\psi_{\eta}$: 
\begin{align}
 & \frac{\dee}{\dee\psi_{\eta}}\left[\frac{(1-\delta)\psi_{\eta}+\delta\psi_{\eta}^{-1}}{\left((1-\delta)\psi_{a}+\delta\psi_{a}^{-1}\right)^{2}}+\frac{(1-\delta)\psi_{\eta}^{-1}+\delta\psi_{\eta}}{\left((1-\delta)\psi_{a}^{-1}+\delta\psi_{a}\right)^{2}}-2\psi_{\eta}\right]=\nonumber \\
 & =\left(\frac{(1-\delta)-\frac{\delta}{\psi_{\eta}^{2}}}{\left[(1-\delta)\psi_{a}+\delta\psi_{a}^{-1}\right]^{2}}+\frac{\delta-\frac{1-\delta}{\psi_{\eta}^{2}}}{\left[(1-\delta)\psi_{a}^{-1}+\delta\psi_{a}\right]^{2}}\right)-2\nonumber \\
 & \leq\left(\frac{1-\delta}{\left[(1-\delta)\psi_{a}+\delta\psi_{a}^{-1}\right]^{2}}+\frac{\delta}{\left[(1-\delta)\psi_{a}^{-1}+\delta\psi_{a}\right]^{2}}\right)-2\,.\label{eq: d>1 simple properties inequality to prove - derivative}
\end{align}
This last term in (\ref{eq: d>1 simple properties inequality to prove - derivative})
is symmetric w.r.t. $\delta$ so we may return to assume $\delta\in[0,\frac{1}{2}]$,
which along $\psi_{a}\geq1$ satisfies $(1-\delta)\psi_{a}+\delta\psi_{a}^{-1}>1$
and $(1-\delta)\psi_{a}^{-1}+\delta\psi_{a}\geq2\sqrt{\delta(1-\delta)}$.
With these properties, we may further upper bound (\ref{eq: d>1 simple properties inequality to prove - derivative})
as 
\[
\left(1-\delta+\frac{1}{4(1-\delta)}\right)-2<0\,.
\]
\item We have 
\[
\left.\frac{\dee F(a,b)}{\dee b}\right|_{b=0}=\E\left[\frac{WV}{\cosh^{2}(aV+\beta)}\right]=0\,.
\]
\item We have 
\[
\frac{\dee^{2}G(a,b)}{\dee a\dee b}=\frac{\dee^{2}F(a,b)}{\dee b^{2}}=-2\E\left[\frac{W^{2}V\tanh(aV+bW+\beta)}{\cosh^{2}(aV+bW+\beta)}\right]
\]
and so using $|\tanh(t)|\leq t$ and $\cosh(t)\geq1$ 
\begin{align*}
\left|\left.\frac{\dee^{2}F(a,b)}{\dee b^{2}}\right|_{b=0}\right| & \leq\E\left[\frac{|V|\cdot|aV+\beta|}{\cosh^{2}(aV+\beta)}\right]\leq a\E\left[|V|^{2}\right]+\beta\E\left[|V|\right]\,.
\end{align*}
\item We have 
\begin{align*}
\left|\frac{\dee^{3}F(a,b)}{\dee b^{3}}\right| & =\left|-2\E\left[W^{3}V\frac{\left(1-2\sinh^{2}(aV+bW+\beta)\right)}{\cosh^{4}(aV+bW+\beta)}\right]\right|\\
 & \leq2\E\left[\left|W^{3}V\right|\right]\\
 & \leq2\sqrt{\E\left[W^{6}\right]\E\left[V^{2}\right]}\\
 & \leq2\sqrt{15(1+\eta^{2})}
\end{align*}
since $\left|\frac{1-2\sinh^{2}(t)}{\cosh^{4}(t)}\right|$ is maximized
at $t=0$ and its maximal value is $1$. 
\end{enumerate}
\end{proof}
We next turn to prove Proposition \ref{prop: Comparison theorem d>1}: 
\begin{proof}[Proof of Proposition \ref{prop: Comparison theorem d>1}]
~ 
\begin{enumerate}
\item By Stein's identity (see (\ref{eq: Stein lemma}) in Appendix \ref{subsec:Usful-results}),
for $U\sim N(a\eta,a^{2}+b^{2})$ 
\[
\frac{G(a,b\mid\eta,\delta)}{b}=\E\left[\frac{1-\delta}{\cosh^{2}(U+\beta)}+\frac{\delta}{\cosh^{2}(U+\beta)}\right]\,.
\]
Then, 
\begin{align*}
\frac{\dee\left[\frac{1}{b}G(a,b\mid\eta,\delta)\right]}{\dee\delta} & =\E\left[\frac{1-\frac{1}{1-\delta}\tanh(U-\beta)}{\cosh^{2}(U-\beta)}-\frac{1-\frac{1}{\delta}\tanh(U+\beta)}{\cosh^{2}(U+\beta)}\right]\\
 & =\E\left[s'(U)\right]\\
 & =\E\left[r\left(a\overline{U}\right)\right]\\
 & \dfn q(\eta)\,,
\end{align*}
where $s(u)$ was defined in (\ref{eq: s function}), and its derivative
is denoted by $s'(u)\equiv\frac{\d s}{\d u}$ (as in (\ref{eq: s function derivative})),
$r(u)\dfn s'(au)$ and $\overline{U}\sim N(\eta,1+\frac{b^{2}}{a^{2}})$.
Recall that Lemma \ref{lem: s is always negative} implies that $s'(u)$
has at most a single zero-crossing at some $u<0$. As $a>0$, $r(u)$
also has a single zero-crossing at some $u<0$. Hence, from Proposition
\ref{prop: Gaussian variation diminishing}, the total positivity
of the Gaussian kernel implies that $\eta\mapsto q(\eta)$ has at
most a single zero-crossing point as a function of $\eta\in\mathbb{R}$
(note that for the sake of the proof we allow $\eta<0$). We next
show that the zero crossing must occur for $\eta<0$. We do so by
evaluating $q(\eta)$ for $\eta=0$ and for large $\eta$. For $\eta=0$,
$\overline{U}\sim N(0,1+\frac{b^{2}}{a^{2}})\eqd-\overline{U}$ and
so 
\begin{equation}
q(0)=\left(\frac{1}{1-\delta}+\frac{1}{\delta}\right)\E\left[\frac{\tanh(a\overline{U}+\beta)}{\cosh^{2}(a\overline{U}+\beta)}\right]>0\label{eq: expectation of tanh over cosh squared}
\end{equation}
since $t\mapsto\frac{\tanh(t)}{\cosh^{2}(t)}$ is an odd function,
positive (resp. negative) on $\mathbb{R}_{+}$ (resp. $\mathbb{R}_{-}$).
Furthermore, noting that $r(0)=\frac{\frac{1}{\delta}\tanh(\beta)-\frac{1}{1-\delta}\tanh(\beta)}{\cosh^{2}(\beta)}>0$,
and using again the single zero-crossing of $r(u)$ at some $u<0$,
we have that $r(u)>0$ for all $u\in\mathbb{R}_{+}$. Since $q(\eta)=r(\eta)*\varphi(\eta\mid1+\frac{b^{2}}{a^{2}})$
where $\varphi(\eta\mid\sigma^{2})$ is the Gaussian kernel with variance
$\sigma^{2}$, i.e., $\varphi(\eta\mid\sigma^{2})\dfn(2\pi\sigma^{2})^{-1/2}\cdot e^{-\eta^{2}/(2\sigma^{2})}$,
there exists some $\eta_{0}>0$ such that $q(\eta)>0$ for all $\eta>\eta_{0}$
(note that $q(\eta)$ is bounded because $r(u)$ is). Since $q(\eta)>0$
for $\eta=0$ and all $\eta>\eta_{0}$, $q(\eta)$ must have an even
number of zero-crossing in $\mathbb{R}_{+}$. Since it cannot have
more than one single crossing, it does not have any. Hence, $q(\eta)=\frac{1}{b}\frac{\dee\left[G(a,b\mid\eta,\delta)\right]}{\dee\delta}>0$
for all $\eta>0$, and thus $G(a,b\mid\eta,\delta)\leq G(a,b\mid\eta,\frac{1}{2})$. 
\item Let $\tilde{V}\sim N(\eta,1)$. We analyze the partial derivatives
of 
\begin{align*}
G(a,b\mid\eta,\delta) & =\E\left[W\cdot\left((1-\delta)\tanh(a\tilde{V}+bW+\beta)+\delta\tanh(-a\tilde{V}+bW+\beta)\right)\right]\\
 & =\E\left[W\cdot\left((1-\delta)\tanh(a\tilde{V}+bW+\beta)+\delta\tanh(a\tilde{V}+bW-\beta)\right)\right]
\end{align*}
w.r.t. $\delta$ around $\delta=\frac{1}{2}$ (i.e., $\beta=0$) and
then use Taylor expansion for $\frac{1}{b}G(a,b\mid\eta,\delta)$.
For brevity, we denote $U=a\tilde{V}+bW\sim N(a\eta,a^{2}+b^{2})$. 
\begin{enumerate}
\item \uline{First derivative:} Taking partial derivative w.r.t. $\delta$\footnote{Note that this form is different from the form used in the previous
item, and is before applying Stein's identity.} 
\begin{align*}
\frac{\dee\left[G(a,b\mid\eta,\delta)\right]}{\dee\delta} & =\E\left[W\cdot\left(\tanh(U-\beta)-\tanh(U+\beta)\right)\right]\\
 & \hphantom{=}+\E\left[W\cdot\left(-\frac{1}{2\delta\cosh^{2}(U+\beta)}+\frac{1}{2(1-\delta)\cosh^{2}(U-\beta)}\right)\right]
\end{align*}
we get $\left.\frac{\dee\left[G(a,b\mid\eta,\delta)\right]}{\dee\delta}\right|_{\delta=\frac{1}{2}}=0$. 
\item \uline{Second derivative:} Taking the next partial derivative w.r.t.
$\delta$ 
\begin{equation}
\frac{\dee^{2}G(a,b\mid\eta,\delta)}{\dee\delta^{2}}=\frac{1}{2\delta^{2}(1-\delta)^{2}}\E\left[W\cdot\left(\frac{(1-\delta)\left[1-\tanh(U+\beta)\right]}{\cosh^{2}(U+\beta)}+\frac{\delta\left[1-\tanh(U-\beta)\right]}{\cosh^{2}(U-\beta)}\right)\right]\label{eq: second derivative of G wrt delta}
\end{equation}
and letting $\overline{U}=U-a\eta\sim N(0,a^{2}+b^{2})$, 
\begin{align*}
\left.\frac{\dee^{2}\left[\frac{1}{b}G(a,b\mid\eta,\delta)\right]}{\dee\delta^{2}}\right|_{\delta=\frac{1}{2}} & =\frac{4}{b}\cdot\E\left[W\cdot\frac{\left[1-\tanh(U)\right]}{\cosh^{2}(U)}\right]\\
 & =\frac{4}{b}\cdot\E\left[\frac{\left[1-\tanh(U)\right]}{\cosh^{2}(U)}\cdot\E\left[W\mid U\right]\right]\\
 & =\frac{4}{a^{2}+b^{2}}\cdot\E\left[\overline{U}\cdot\frac{\left[1-\tanh(\overline{U}+a\eta)\right]}{\cosh^{2}(\overline{U}+a\eta)}\right]\\
 & \dfn\frac{4}{a^{2}+b^{2}}\cdot A_{0}(a,b)\,,
\end{align*}
where the equality holds since $\E[W\mid U]=\frac{b}{a^{2}+b^{2}}(U-a\eta)$
and where $A_{0}(a,b)$ was implicitly defined. We next show that
$A_{0}(a,b)<0$. To this end, first assume that both $a>0$ and $b>0$,
and let $h(t)\dfn\frac{1-\tanh(t)}{\cosh^{2}(t)}.$ It holds that
$h(t)\geq0$ for all $t\in\mathbb{R}$, $h(t)\leq h(-t)$ for $t\geq0$,
and $h(t)$ has unique maximum at $t=-\log\sqrt{2}<0$. In addition,
for any $\overline{u}>0$, it holds that $h(-\overline{u}+a\eta)>h(\overline{u}+a\eta)$.
Indeed, if $-\overline{u}+a\eta\geq0$ then this is true since $h(t)$
is strictly decreasing for $t\geq0$, and if $-\overline{u}+a\eta<0$
then $h(-\overline{u}+a\eta)>h(\overline{u}-a\eta)>h(\overline{u}+a\eta)$.
Now, the conditional version of the expectation defining $A_{0}(a,b)$,
when conditioned on $|\overline{U}|=\overline{u}>0$ satisfies 
\[
\E\left[\overline{U}\cdot h(\overline{U}+a\eta)\mid|\overline{U}|=\overline{u}\right]=\frac{1}{2}\left[\overline{u}\cdot h(\overline{u}+a\eta)-\overline{u}\cdot h(-\overline{u}+a\eta)\right]<0\,,
\]
and so $A_{0}(a,b)<0$. Therefore, any $(a,b,\eta)\in(0,C_{f}]^{2}\times[0,\gl[C]_{\theta}]$
satisfies 
\[
\left.\frac{\dee^{2}\left[\frac{1}{b}G(a,b\mid\eta,\delta)\right]}{\dee\delta^{2}}\right|_{\delta=\frac{1}{2}}\dfn\Gamma(a,b,\eta)<0\,.
\]
We may now consider the cases $a=0$ or $b=0$. If $a=0$ but $b\neq0$
or vice-versa, similar analysis to before shows that $\Gamma(a,b,\eta)<0$.
For $(a,b)=(0,0)$ we use Stein's identity (see (\ref{eq: Stein lemma})
in Appendix \ref{subsec:Usful-results}) to obtain 
\[
\frac{4}{a^{2}+b^{2}}A_{0}(a,b)=4\cdot\E\left[\frac{e^{-2(\overline{U}+a\eta)}-2}{\cosh^{4}(\overline{U}+a\eta)}\right]
\]
and so $\left.\frac{1}{b}\frac{\dee^{2}\left[G(a,b\mid\eta,\delta)\right]}{\dee\delta^{2}}\right|_{\delta=\frac{1}{2}}=-4$
for $(a,b)=(0,0)$. Hence, there exists $C_{2}(C_{f},\gl[C]_{\theta},\gl[C]_{\beta})>0$
such that 
\[
\max_{(a,b,\eta)\in[0,C_{f}]^{2}\times[0,\gl[C]_{\theta}]}\Gamma(a,b,\eta)=-C_{2}<0\,.
\]
\item \uline{Third derivative:} We show that its absolute value is upper
bounded. As apparent from form (\ref{eq: second derivative of G wrt delta}),
the second derivative $\frac{\dee^{2}\left[G(a,b;\delta)\right]}{\dee\delta^{2}}$
can be written as the sum of two terms of the same form. We show how
to bound the derivative of the first, and as it is similar, omit the
bounding of the second. Recalling that $\overline{U}\sim N(0,a^{2}+b^{2})$,
the first term is, 
\[
\frac{1}{2\delta^{2}(1-\delta)}\E\left[W\frac{\left[1-\tanh(U+\beta)\right]}{\cosh^{2}(U+\beta)}\right]=\frac{1}{2\delta^{2}(1-\delta)}\frac{b}{a^{2}+b^{2}}\E\left[\overline{U}\cdot\frac{\left[1-\tanh(\overline{U}+a\eta+\beta)\right]}{\cosh^{2}(\overline{U}+a\eta+\beta)}\right]
\]
using again $\E[W\mid U]=\frac{b}{a^{2}+b^{2}}(U-a\eta)$. Hence $\frac{1}{b}\cdot\frac{\dee^{3}\left[G(a,b;\delta)\right]}{\dee\delta^{3}}$
has two terms of a similar form, the first of them is 
\begin{align*}
 & \frac{1}{a^{2}+b^{2}}\cdot\frac{\dee}{\dee\delta}\left\{ \frac{1}{2\delta^{2}(1-\delta)}\E\left[\overline{U}\cdot\frac{\left[1-\tanh(\overline{U}+a\eta+\beta)\right]}{\cosh^{2}(\overline{U}+a\eta+\beta)}\right]\right\} \\
 & =-\frac{1}{a^{2}+b^{2}}\cdot\frac{(1-\frac{3}{2}\delta)}{\delta^{3}(1-\delta)^{2}}\cdot\E\left[\overline{U}\cdot\frac{\left[1-\tanh(\overline{U}+a\eta+\beta)\right]}{\cosh^{2}(\overline{U}+a\eta+\beta)}\right]\\
 & -\frac{1}{a^{2}+b^{2}}\cdot\frac{1}{\delta^{3}(1-\delta)^{2}}\E\left[\overline{U}\cdot\frac{e^{-2\overline{U}-2a\eta-2\beta}-2}{\left[1+\cosh(2\overline{U}+2a\eta+2\beta)\right]^{2}}\right]\\
 & \dfn-\frac{(1-\frac{3}{2}\delta)}{\delta^{3}(1-\delta)^{2}}A_{1}(a,b)-\frac{1}{\delta^{3}(1-\delta)^{2}}A_{2}(a,b)\,,
\end{align*}
where $A_{1}(a,b),A_{2}(a,b)$ where implicitly defined. The multiplicative
factors $\frac{(1-\frac{3}{2}\delta)}{\delta^{3}(1-\delta)^{2}}$
and $\frac{1}{\delta^{3}(1-\delta)^{2}}$ are upper bounded since
$\delta$ is assume to be bounded away from zero ($\beta<\gl[C]_{\beta}$),
and so we focus on $A_{1}(a,b),A_{2}(a,b)$. Further, 
\[
A_{1}(a,b)=\frac{1}{a^{2}+b^{2}}\cdot\E\left[\left|\overline{U}\cdot\frac{\left[1-\tanh(\overline{U}+a\eta+\beta)\right]}{\cosh^{2}(\overline{U}+a\eta+\beta)}\right|\right]<2\frac{1}{a^{2}+b^{2}}\cdot\E\left|\overline{U}\right|<\frac{2}{\sqrt{a^{2}+b^{2}}}
\]
and 
\begin{align*}
A_{2}(a,b) & =\frac{1}{a^{2}+b^{2}}\cdot\E\left[\overline{U}\cdot\frac{e^{-2\overline{U}-2a\eta-2\beta}-2}{\left[1+\cosh(2\overline{U}+2a\eta+2\beta)\right]^{2}}\right]\\
 & \leq\frac{4}{a^{2}+b^{2}}\cdot\E\left[|\overline{U}|\cdot\frac{e^{-2\overline{U}-2a\eta-2\beta}}{\left[2+e^{2\overline{U}+2a\eta+2\beta}+e^{-2\overline{U}-2a\eta-2\beta}\right]^{2}}\right]\\
 & \hphantom{=}+\frac{1}{a^{2}+b^{2}}\cdot\E\left[\left|\overline{U}\right|\cdot\frac{2}{\left[1+\cosh(2\overline{U}+2a\eta+2\beta)\right]^{2}}\right]\\
 & \leq\frac{6}{a^{2}+b^{2}}\cdot\E\left|\overline{U}\right|<\frac{6}{\sqrt{a^{2}+b^{2}}}\,.
\end{align*}
Thus, if either $a>0$ or $b>0$ then both $A_{1}(a,b)<\infty$ and
$A_{2}(a,b)<\infty$. It remains to consider limits to $(a,b)=(0,0)$.
The limits for $A_{1}(a,b)$ can be shown to be finite by an analysis
similar to the one made for the second derivative $\frac{\dee^{2}G(a,b\mid\eta,\delta)}{\dee\delta^{2}}$.
For $A_{2}(a,b)$, using Stein's identity (see (\ref{eq: Stein lemma})
in Appendix \ref{subsec:Usful-results}) 
\begin{align*}
A_{2}(a,b) & =\frac{1}{a^{2}+b^{2}}\E\left[\overline{U}\frac{e^{-2\overline{U}-2a\eta-2\beta}-2}{\left[1+\cosh(2\overline{U}+2a\eta+2\beta)\right]^{2}}\right]\\
 & =-2\cdot\E\left[\frac{e^{-2\overline{U}-2a\eta-2\beta}}{\left[1+\cosh(2\overline{U}+2a\eta+2\beta)\right]^{2}}\right]\\
 & \hphantom{=}-4\cdot\E\left[\frac{\left(e^{-2\overline{U}-2a\eta-2\beta}-2\right)\sinh(2\overline{U}+2a\eta+2\beta)}{\left[1+\cosh(2\overline{U}+2a\eta+2\beta)\right]^{3}}\right]
\end{align*}
and so $|A_{2}(0,0)|<2e^{-2\beta}+4\left|e^{-2\beta}-2\right|\sinh(2\beta)<\infty$.
Hence, there exists $C_{3}(C_{f},\gl[C]_{\beta})$ such that 
\[
\sup_{(a,b,\eta)\in[0,C_{f}]^{2}\times[0,\gl[C]_{\theta}]}\left|\frac{\dee^{3}\left[\frac{1}{b}G(a,b\mid\eta,\delta)\right]}{\dee\delta^{3}}\right|\leq C_{3}\,.
\]
\end{enumerate}
From the analysis of the derivatives, and recalling that $\rho=2(\frac{1}{2}-\delta)$,
for any $(a,b,\eta)\in[0,C_{f}]^{2}\times[0,\gl[C]_{\theta}]$ it
holds that
\[
\frac{G(a,b\mid\eta,\delta)}{b}\leq\frac{G(a,b\mid\eta,\frac{1}{2})}{b}-\frac{C_{2}}{8}\cdot\rho^{2}+\frac{C_{3}}{48}\rho^{3}.
\]
If $\gl[C]_{\beta}$ is such that $\rho\leq\frac{C_{2}}{3C_{3}}=\overline{\rho}$
then $\frac{1}{b}G(a,b\mid\eta,\delta)\leq\frac{1}{b}G(a,b\mid\eta,\frac{1}{2})-\frac{C_{2}}{16}\cdot\rho^{2}$.
Otherwise, by dominance of $G$ w.r.t $\delta$ (item \ref{enu: comparison theorem d>1  - G itertaion dominance})\textbf{
} 
\begin{align*}
\frac{1}{b}G(a,b\mid\eta,\delta) & \leq\frac{1}{b}G\left(a,b\mid\eta,\tfrac{1-\overline{\rho}}{2}\right)\leq\frac{G(a,b\mid\eta,\frac{1}{2})}{b}-\frac{C_{2}}{8}\overline{\rho}^{2}\\
 & \leq\frac{1}{b}G(a,b\mid\eta,\tfrac{1}{2})-C_{4}\rho^{2}
\end{align*}
for some constant $C_{4}$ (which depends on $\gl[C]_{\beta}$). Taking
$\gl[C]_{1}^{(d)}=\min(\frac{1}{8}C_{2},C_{4})$ 
\begin{align*}
G(a,b\mid\eta,\delta) & \leq G(a,b\mid\eta,\tfrac{1}{2})-\gl[C]_{1}^{(d)}\rho^{2}b\\
 & \leq b\left(1-\frac{a^{2}+b^{2}}{2+4(a^{2}+b^{2})}-\gl[C]_{1}^{(d)}\rho^{2}\right)
\end{align*}
where the upper bound on $G(a,b\mid\eta,\frac{1}{2})$ was obtained
in the analysis of the balanced iteration \cite[Lemma 5, item 8]{wu2019EM}. 
\item Let $Z\sim N(0,1)$. By Stein's identity for $W$ (see (\ref{eq: Stein lemma})
in Appendix \ref{subsec:Usful-results}), and a change of measure
(see (\ref{eq: change of measure}) in Appendix \ref{subsec:Usful-results})
\begin{align*}
G(a,b\mid\eta,\delta) & =\E\left[\frac{b}{\cosh^{2}(aV+bW+\beta)}\right]\\
 & =e^{-\eta^{2}/2}\E\left[\frac{b}{\cosh^{2}(aZ+bW+\beta)}\left((1-\delta)e^{\eta Z}+\delta e^{-\eta Z}\right)\right]\,.
\end{align*}
Then, 
\begin{align}
\frac{\dee G(a,b\mid\eta,\delta)}{\dee\eta} & =-\eta\cdot G(a,b\mid\eta,\delta)+e^{-\eta^{2}/2}\cdot\E\left[\frac{bZ}{\cosh(aZ+bW+\beta)^{2}}\left((1-\delta)e^{\eta Z}-\delta e^{-\eta Z}\right)\right]\nonumber \\
 & =-2be^{-\eta^{2}/2}\cdot\E\left[\frac{\tanh(aZ+bW+\beta)}{\cosh^{2}(aZ+bW+\beta)}\left((1-\delta)e^{\eta Z}-\delta e^{-\eta Z}\right)\right]\,,\label{eq: derivative of G wrt to eta}
\end{align}
by Stein's identity for $Z$. Letting $U=aZ+bW+\beta\sim N(\beta,a^{2}+b^{2})$,
we have that $Z|U\sim N(\frac{a(U-\beta)}{a^{2}+b^{2}},\frac{b^{2}}{a^{2}+b^{2}})$,
and so 
\begin{multline*}
\frac{\dee G(a,b\mid\eta,\delta)}{\dee\eta}\\
=-2\exp\left(-\frac{\eta^{2}a^{2}}{2(a^{2}+b^{2})}\right)\E\left[\frac{b\tanh(U)}{\cosh^{2}(U)}\left((1-\delta)\exp\left(\frac{\eta a}{a^{2}+b^{2}}(U-\beta)\right)-\delta\exp\left(-\frac{\eta a}{a^{2}+b^{2}}(U-\beta)\right)\right)\right]\,.
\end{multline*}
Letting 
\[
p_{+}\dfn(1-\delta)\exp\left(-\frac{\eta a}{a^{2}+b^{2}}\beta\right),\;p_{-}\dfn\delta\exp\left(\frac{\eta a}{a^{2}+b^{2}}\beta\right)\,,
\]
then under the assumption $\frac{\eta a}{a^{2}+b^{2}}<1$ and using
$\exp(\beta)=\sqrt{(1-\delta)/\delta}$ it holds that $p_{+}\geq p_{-}$.
Now, 
\[
h(u)=\frac{\tanh(u)}{\cosh^{2}(u)}\left(p_{+}\exp\left(\frac{\eta au}{a^{2}+b^{2}}\right)-p_{-}\exp\left(-\frac{\eta au}{a^{2}+b^{2}}\right)\right)
\]
satisfies that for $u\geq0$, $h(u)\geq0$ and $|h(u)|\geq|h(-u)|$.
Thus, we deduce that 
\[
\frac{\dee G(a,b\mid\eta,\delta)}{\dee\eta}=-2b\cdot e^{-\frac{\eta^{2}a^{2}}{2(a^{2}+b^{2})}}\E\left[h(U)\right]\leq0
\]
(see the Gaussian average of odd function property in Appendix (\ref{subsec:Usful-results})). 
\item We show that $\left.\frac{1}{b}\frac{\dee G(a,b\mid\eta,\delta)}{\dee\eta}\right|_{\eta=0}<0$
and that $\frac{1}{b}\frac{\dee^{2}G(a,b\mid\eta,\delta)}{\dee\eta^{2}}$
is uniformly bounded (over all $(a,b,\eta$)), and the result then
follows from Taylor expansion. 
\begin{enumerate}
\item \uline{First derivative:} Let $\overline{U}\sim N(0,a^{2}+b^{2})$.
If $b=0$ then $\left.\frac{\dee G(a,b\mid\eta,\delta)}{\dee\eta}\right|_{\eta=0}=0$
and so we next assume $b>0$. From (\ref{eq: derivative of G wrt to eta})
and Stein's identity (see (\ref{eq: Stein lemma}) in Appendix \ref{subsec:Usful-results})
\begin{align*}
\left.\frac{\dee G(a,b\mid\eta,\delta)}{\dee\eta}\right|_{\eta=0} & =-2(1-2\delta)\cdot b\E\left[\frac{\tanh(\overline{U}+\beta)}{\cosh^{2}(\overline{U}+\beta)}\right]\\
 & =(1-2\delta)\cdot b\frac{1}{a^{2}+b^{2}}\E\left[\frac{\overline{U}}{\cosh^{2}(\overline{U}+\beta)}\right]\\
 & <0
\end{align*}
since for $u>0,\beta>0$ it holds that $\cosh(u+\beta)>\cosh(-u+\beta)>0$
and $\overline{U}\eqd-\overline{U}$. 
\item \uline{Second derivative:} Let $Z\sim N(0,1)$. Taking the next
partial derivative w.r.t. $\eta$ in (\ref{eq: derivative of G wrt to eta})
\begin{align*}
\frac{\dee^{2}G(a,b\mid\eta,\delta)}{\dee\eta^{2}} & =-\eta\frac{\dee G(a,b\mid\eta,\delta)}{\dee\eta}-2be^{-\eta^{2}/2}\cdot\E\left[Z\frac{\tanh(aZ+bW+\beta)}{\cosh^{2}(aZ+bW+\beta)}\left((1-\delta)e^{\eta Z}+\delta e^{-\eta Z}\right)\right]\,.
\end{align*}
The absolute value of the first term is bounded by $2\gl[C]_{\theta}b$
since 
\begin{align*}
\left|\frac{\dee G(a,b\mid\eta,\delta)}{\dee\eta}\right| & \leq2be^{-\eta^{2}/2}\E\left[(1-\delta)e^{\eta Z}+\delta e^{-\eta Z}\right]=2
\end{align*}
(using $\left|\frac{\tanh(t)}{\cosh^{2}(t)}\right|\leq1$). The absolute
value of the second term is bounded by $4\gl[C]_{\theta}b$ since
\begin{align*}
 & \E\left[\left|Z\frac{\tanh(aZ+bW+\beta)}{\cosh^{2}(aZ+bW+\beta)}\left((1-\delta)e^{\eta Z}+\delta e^{-\eta Z}\right)\right|\right]\\
 & \leq\E\left[\left|Z\right|\left((1-\delta)e^{\eta Z}+\delta e^{-\eta Z}\right)\right]\\
 & =\E\left[\left|Z\right|e^{\eta Z}\right]\\
 & \leq2\cdot\E\left[Ze^{\eta Z}\right]\\
 & \trre[=,a]2\cdot\eta\E\left[e^{\eta Z}\right]\\
 & =2\eta e^{\eta^{2}/2}
\end{align*}
where $(a)$ follows from Stein's identity. Thus, $\left|\frac{\dee^{2}G(a,b)}{\dee\eta^{2}}\right|\leq6\gl[C]_{\theta}b$. 
\end{enumerate}
\end{enumerate}
\end{proof}
We may now prove that the population mean iteration converges. 
\begin{proof}[Proof of Theorem \ref{thm: population d=00003D00003D1 known delta}]
If $a_{0}\geq0$ then $a_{t}\geq0$ for all $t>1$ (Lemma \ref{lem: properties of F and G},
item \ref{enu: simple properties d>1 positivity}).\textbf{ }Consider
the upper envelope iteration $b_{t+1}^{+}=G(a_{t},b_{t}^{+}\mid\eta,\frac{1}{2})$,
where $b_{0}^{+}=b_{0}.$ Since $b\mapsto G(a,b\mid\eta,\delta)$
is increasing for $a>0$ (Lemma \ref{lem: properties of F and G},
item \ref{enu: simple properties d>1 monotonicity}), Proposition
\ref{prop: Comparison theorem d>1}, item \ref{enu: comparison theorem d>1  - G itertaion dominance}
and induction imply that $b_{t}^{+}\geq b_{t}$ for all $t\geq1$.
It follows from the analysis of the balanced iteration \cite[Lemma 5, item 8]{wu2019EM}
that (see Proposition \ref{prop: Comparison theorem d>1}, item \ref{enu: comparison theorem d>1  - G itertaion explicit}
and its proof) that $b_{t}^{+}\to0$. Thus also $b_{t}\to0$ as $t\to\infty$.
As $\left|\frac{\dee F(a,b)}{\dee b}\right|\leq\sqrt{1+\eta^{2}}\leq\sqrt{1+\gl[C]_{\theta}}$
is uniformly bounded (Lemma \ref{lem: properties of F and G}, item
\ref{enu: simple properties d>1 upper bound on first derivatives}),
for any given $\epsilon>0$, there exists $t>0$ such that 
\[
\left|F(a_{t},b_{t})-F(a_{t},0)\right|\leq\epsilon\,,
\]
where $F(a_{t},0)=f(\theta\mid\eta,\delta)$, i.e., the population
mean iteration in $d=1$. Theorem \ref{thm: empirical d=00003D00003D1 known delta}
shows that convergence is assured for any given sufficiently small
absolute error, and that the error $|\theta_{t}-\eta|$ tends to zero
as $\epsilon\to0$. 
\end{proof}

\subsection{Empirical iteration}
\begin{proof}[Proof of Theorem \ref{thm: empirical d>1 known delta large rho}]
We analyze the empirical iteration $\theta_{t+1}=f_{n}(\theta_{t})\equiv f_{n}(\theta_{t},\delta\mid\theta_{*},\delta)$.
We will assume that $\rho\geq C_{1}\sqrt{\omega}$ and specify conditions
on $C_{1}$ along the proof. As for the population iteration (Lemma
\ref{lem:Decomposition to signal and orthogoanl iteration}), we may
write 
\[
\theta_{t}=a_{t}\cdot\hat{\theta}_{*}+b_{t}\cdot\xi_{t}
\]
where $\eta=\|\theta_{*}\|$, $\hat{\theta}_{*}=\frac{\theta_{*}}{\eta}$,
$\xi_{t}\perp\eta$ and $\|\xi_{t}\|=1$ such that $\spa\{\theta_{*},\xi_{t}\}=\spa\{\theta_{*},\theta_{t}\}$
and $b_{t}\geq0$. Assuming the high probability event (\ref{eq: high probability event})
holds, we have that 
\[
\|f_{n}(\theta)-f(\theta)\|\leq\max\{\eta,\rho\}\cdot\omega
\]
and so for the signal iteration 
\begin{align*}
a_{t+1} & =\langle\theta_{t+1},\hat{\theta}_{*}\rangle=\langle f_{n}(\theta_{t}),\hat{\theta}_{*}\rangle\\
 & \leq F(a_{t},b_{t})+\max\left\{ \left|a_{t}\right|+b_{t},\rho\right\} \cdot\omega\\
 & \leq F(a_{t},b_{t})+\left(\left|a_{t}\right|+b_{t}+\rho\right)\cdot\omega\\
 & \dfn F_{+}(a_{t},b_{t})\,,
\end{align*}
and, similarly, 
\begin{align*}
a_{t+1} & \geq F(a_{t},b_{t})-\max\left\{ \left|a_{t}\right|+b_{t},\rho\right\} \cdot\omega\\
 & \geq F(a_{t},b_{t})-\left(\left|a_{t}\right|+b_{t}+\rho\right)\cdot\omega\\
 & \dfn F_{-}(a_{t},b_{t})\,.
\end{align*}
In the same spirit, for the orthogonal iteration , it holds that 
\[
b_{t+1}\leq G(a_{t},b_{t})+\max\left\{ \left|a_{t}\right|+b_{t},\rho\right\} \cdot\omega\,.
\]

We split the analysis into two regimes of $\eta\lesssim\frac{\omega}{\rho}$
and $\eta\gtrsim\frac{\omega}{\rho}$. In the former regime, the iteration
dwells around $\|\theta_{t}\|\lesssim\frac{\omega}{\rho}$, though
the corresponding signal iteration $a_{t}$ might be negative. In
the later regime, it is assured that $a_{t}\geq0$ for all $t$ (given
that $a_{0}\geq0$), and so properties such as dominance of the orthogonal
iteration may be used. As a preliminary step, we show that the iteration
is bounded:

\uline{Step 0 (Boundedness):} We prove that for all $n$ sufficiently
large, it holds that $|a_{t}|,b_{t}\leq C_{f}$ for all $t\geq1$
if $C_{f}\geq\max\{2\gl[C]_{\theta},\rho/2,4(\gl[C]_{\theta}+1)\}$.
By induction, when the iteration is initialized with $\theta_{0}=0$
then $|a_{0}|=b_{0}=0$. When $\theta_{0}=\frac{1}{\rho}\E_{n}[X]$
then assuming the high probability event (\ref{eq: high probability event})
it holds that $\|\theta_{0}\|\leq\eta+\frac{\omega}{\rho}\leq2\gl[C]_{\theta}$
for all $n>n_{0}(\gl[C]_{\theta},\gl[C]_{\omega})$ (see the end of
the proof of Theorem \ref{thm: empirical d=00003D00003D1 known delta}).
Thus also $|a_{0}|,b_{0}\leq2\gl[C]_{\theta}$. Note that for all
$n>n_{1}(\gl[C]_{\omega})$ it holds that $\omega\leq1/4$. For the
induction step, assume that $|a_{t}|,b_{t}\leq C_{f}$ for some $t$.
Then, using Lemma \ref{lem: properties of F and G}, item \ref{enu: simple properties d>1 boundedness}
\begin{align*}
a_{t+1} & \leq F(a_{t},b_{t})+\max\left\{ \left|a_{t}\right|+b_{t},\rho\right\} \cdot\omega\\
 & \leq\gl[C]_{\theta}+1+2C_{f}\omega\\
 & \leq C_{f}\,,
\end{align*}
and a similar lower bound on $a_{t+1}$ holds, as well as a similar
upper bound $b_{t+1}$. We henceforth assume that $|a_{t}|,b_{t}\leq C_{f}$.

\paragraph*{Large signal case}

Assume that $\eta>C_{0}\frac{\omega}{\rho}$ where $C_{0}>0$ is to
be specified later on.

\uline{Step 1 (Orthogonal iteration and Positivity):} We prove
by induction that there exists $c_{1}>0$ to be specified (large enough)
such that if $a_{0}\geq0$ and $b_{0}\leq c_{1}\frac{\omega}{\rho}$
then $a_{t}\geq0$ and $b_{t}\leq c_{1}\frac{\omega}{\rho}$ for all
$t$. For $t=0$, this is satisfied when initializing with both $\theta_{0}=0$
since, trivially, $a_{0}=b_{0}=0$, and also when initializing with
$\theta_{0}=\frac{1}{\rho}\E_{n}[X]$, since under the high probability
event (\ref{eq: high probability event}) 
\[
\|\frac{1}{\rho}\E_{n}[X]-\theta_{*}\|=\frac{1}{\rho^{2}}\|f_{n}(0)-f(0)\|\leq\frac{\omega}{\rho}\,.
\]
Taking $C_{0}>2$ (say) implies that $a_{0}>0$ and $b_{0}\leq\frac{\omega}{\rho}$.
We assume that $a_{t}\geq0$ and that $b_{t}\leq c_{1}\frac{\omega}{\rho}$
and show that these properties continue to hold after iteration $t+1$.
We first consider the orthogonal iteration which is analyzed through
its upper bound. There are two differences compared to the balanced
case \cite[Theorem 5]{wu2019EM}: 1) The slope of the upper bound
on $b_{t}$ has additional $-\gl[C]_{G,\rho}^{(d)}\rho^{2}$ term
(which improves the bound on $b_{t}$ and improves convergence to
low values). 2) The empirical error has an additional $\omega\rho$
term which deteriorates the bound on $b_{t}$. Nonetheless, the first
effect dominates the second. Specifically, from Proposition \ref{prop: Comparison theorem d>1},
item \ref{enu: comparison theorem d>1  - G itertaion explicit},\footnote{For simplicity of later notation, the constant $\gl[C]_{G,\rho}^{(d)}$
was reduced to $\frac{\gl[C]_{G,\rho}^{(d)}}{2}$.} since $a_{t}>0$ and $\eta>C_{0}\frac{\omega}{\rho}$ was assumed,
\begin{align}
b_{t+1} & \leq b_{t}\left(1-\frac{a_{t}^{2}+b_{t}^{2}}{2+4(a_{t}^{2}+b_{t}^{2})}-\frac{\gl[C]_{G,\rho}^{(d)}}{2}\rho^{2}\right)+\omega\left(a_{t}+b_{t}+\rho\right).\label{eq: bound on G iteration empirical iteration proof}\\
 & \trre[\leq,a]b_{t}\left(1+\omega-\frac{\gl[C]_{G,\rho}^{(d)}}{2}\rho^{2}\right)-\frac{b_{t}^{3}}{C_{2}}+\sup_{0\leq a_{t}\leq C_{f}}\left(\omega a_{t}-\frac{a_{t}^{2}}{C_{2}}\right)+\omega\rho\nonumber \\
 & \trre[\leq,b]b_{t}\left(1-\frac{\gl[C]_{G,\rho}^{(d)}}{2}\rho^{2}\right)-\frac{b_{t}^{3}}{C_{2}}+\frac{C_{2}\omega^{2}}{4b_{t}}+\omega\rho\nonumber 
\end{align}
where in $(a)$ we have used $C_{2}=2+8C_{f}^{2}$, and $(b)$ holds
by requiring that $C_{1}\geq\sqrt{\frac{2}{\gl[C]_{G,\rho}^{(d)}}}$.
We assume w.l.o.g. that $\gl[C]_{G,\rho}^{(d)}\leq1$ as otherwise,
we may weaken the bound by setting $\gl[C]_{G,\rho}^{(d)}=1$. Now,
we may show that $b_{t+1}\leq c_{1}\frac{\omega}{\rho}$: Let $c_{2}>0$
be a constant to be specified. 
\begin{itemize}
\item If $0\leq b_{t}\leq c_{2}\omega$ then the bound 
\[
b_{t+1}\leq b_{t}\left(1-\frac{\gl[C]_{G,\rho}^{(d)}}{2}\rho^{2}\right)+\omega C_{f}+\omega\rho\leq c_{1}\frac{\omega}{\rho}
\]
holds as long as $c_{2}+\omega(C_{f}+\rho)\leq\frac{c_{1}}{\rho}$
(condition I). 
\item If $c_{2}\omega\leq b_{t}\leq c_{1}\frac{\omega}{\rho}$ we may use
the bound 
\[
b_{t+1}\leq b_{t}\left(1-\frac{\gl[C]_{G,\rho}^{(d)}}{2}\rho^{2}\right)-\frac{b_{t}^{3}}{C_{2}}+\frac{C_{2}\omega^{2}}{4b_{t}}+\omega\rho\dfn h(b_{t})\,.
\]
Under the assumption $\rho\geq C_{1}\sqrt{\omega}$ it holds that
$\frac{\omega^{2}}{\rho^{2}}\leq\frac{\omega}{C_{1}^{2}}\to0$ as
$n\to\infty$ Thus, we may assume that $c_{1}<\frac{\rho^{2}}{\omega^{2}}$
and $c_{2}\geq\sqrt{C_{2}}$ (condition II), so that for all $n>n_{2}(c_{1},C_{1},C_{2},\gl[C]_{\omega})$
\[
\frac{\d h}{\d b}=\left(1-\frac{\gl[C]_{G,\rho}^{(d)}}{2}\rho^{2}\right)-\frac{3b^{2}}{C_{2}}-\frac{C_{2}\omega^{2}}{4b^{2}}\geq\frac{1}{2}-\frac{3c_{1}^{2}\omega^{2}}{C_{2}\rho^{2}}-\frac{C_{2}}{4c_{2}^{2}}>0\,.
\]
In this event, $b\mapsto h(b)$ is increasing, and thus if $\frac{C_{2}}{4c_{1}}+1-c_{1}\frac{\gl[C]_{G,\rho}^{(d)}}{2}<0$
(condition III) then 
\begin{align*}
b_{t+1} & \leq\max_{c_{2}\omega\leq b\leq c_{1}\frac{\omega}{\rho}}h(b)=h(c_{1}\frac{\omega}{\rho})\\
 & \leq c_{1}\frac{\omega}{\rho}+\left(\frac{C_{2}}{4c_{1}}+1-c_{1}\frac{\gl[C]_{G,\rho}^{(d)}}{2}\right)\omega\rho\\
 & \leq c_{1}\frac{\omega}{\rho}
\end{align*}
holds. We choose $c_{2}$ such that condition II holds, and then choose
$c_{1}\geq1$ and large enough so that conditions I and III will hold
(note that $n_{2}$ may be affected by these choices). 
\end{itemize}
Thus, we have proved that $b_{t+1}\leq c_{1}\frac{\omega}{\rho}$
in the next iteration. We next show that $a_{t+1}>0$. Using the boundedness
from step 0 and Lemma \ref{lem: properties of F and G}, items \ref{enu: simple properties d>1 monotonicity}
and \ref{enu: simple properties d>1 strict positivity} 
\begin{align*}
a_{t+1} & \geq F(a_{t},b_{t})-\max\left\{ \left|a_{t}\right|+b_{t},\rho\right\} \cdot\omega\\
 & \trre[\geq,a]F(a_{t},b_{t})-\omega a_{t}-c_{1}\frac{\omega^{2}}{\rho}-\rho\omega\\
 & \trre[\geq,b]F(0,b_{t})+\gl[C]''_{F}a_{t}-\omega a_{t}-c_{1}\frac{\omega^{2}}{\rho}-\rho\omega\\
 & \trre[\geq,c](\gl[C]''_{F}-\omega)a_{t}+\rho\left(\gl[C]_{F,0}\rho\eta-c_{1}\frac{\omega^{2}}{\rho^{2}}-\omega\right)\\
 & >0\,,
\end{align*}
where $(a)$ is by the induction assumption, $(b)$ is by the uniform
bound on $\frac{\dee F}{\dee a}$ in Lemma \ref{lem: properties of F and G}
item \ref{enu: simple properties d>1 lower bound on F first derivatives}
and Taylor expansion, $(c)$ is by the lower bound on $F(0,b)$ in
Lemma \ref{lem: properties of F and G} item \ref{enu: simple properties d>1 strict positivity},
and the final inequality holds as long as $n$ is sufficiently large
so that $\omega\leq\gl[C]''_{F}$, and $\frac{\omega}{\rho}<1$, when
requiring that $C_{0}\geq\frac{c_{1}+1}{\gl[C]_{F,0}}$.

\uline{Step 2 (Signal iteration):} From the previous step, we may
assume that $a_{t}\geq0$ and $b_{t}\leq c_{1}\frac{\omega}{\rho}$
for all $t\geq0$. Lemma \ref{lem: properties of F and G} items \ref{enu: simple properties d>1 F mixed derivative at b=00003D00003D0},
\ref{enu: simple properties d>1 upperbound on F second mixed derivative at b=00003D00003D0}
and \ref{enu: simple properties d>1 upperbound on F third mixed derivative at b=00003D00003D0}
characterize the derivative of $F(a,b)$ w.r.t. $b$ around $b=0$.
Using these claims and the global assumptions $\eta\leq\gl[C]_{\theta}$
and $\beta\leq\gl[\overline{C}]_{\beta}\rho$, Taylor expansion of
$F(a,b)$ around $b=0$ results 
\[
F_{+}(a,b)\leq F(a,0)+\left[C_{3}\cdot(|a|+\rho)b^{2}+\gl[C]'''_{F}b^{3}+\omega(|a|+b+\rho)\right]\,,
\]
and 
\[
F_{-}(a,b)\geq F(a,0)-\left[C_{3}\cdot(|a|+\rho)b^{2}+\gl[C]'''_{F}b^{3}+\omega(|a|+b+\rho)\right]\,.
\]
Observe that for $b=0$ it holds that $F(a,0)=f(\theta,\delta\mid\eta,\delta)$
is simply the one-dimensional iteration with absolute mean $\eta$.
Furthermore, the envelopes $F_{\pm}(a,b)$ are the same as one-dimensional
envelopes $f_{\pm}(\theta)$ with $\theta\equiv a$ except for excess
error $C_{3}\cdot(|a|+\rho)b^{2}+\gl[C]'''_{F}b^{3}+\omega b$ which
results from the orthogonal error when $b\neq0$. From the previous
step, we have that $b_{t}\leq c_{1}\frac{\omega}{\rho}\leq\frac{c_{1}}{C_{1}}\sqrt{\omega}\leq\frac{c_{1}\rho}{C_{1}^{2}}$
for all $t$ and thus we may evaluate the terms of this excess error
under this assumption. Then, $C_{3}\cdot|a|b^{2}\leq C_{3}\frac{c_{1}}{C_{1}}\omega|a|$,
$C_{3}\cdot\rho b^{2}\leq C_{3}\frac{c_{1}^{2}}{C_{1}^{2}}\omega\rho$,
$\gl[C]'''_{F}b^{3}\leq\gl[C]'''_{F}\frac{c_{1}^{3}}{C_{1}^{2}}\omega\rho$,
and $\omega b\leq\frac{c_{1}}{C_{1}^{2}}\omega\rho$. Hence, the excess
error is no more than $C_{4}(|a|+\rho)\omega$ for some $C_{4}>0$.
Since the one-dimensional envelopes have error $\max\{|a|,\rho\}\omega\geq\frac{1}{2}(|a|+\rho)\omega$
and $\max\{|a|,\rho\}\omega\leq(|a|+\rho)\omega$ (see (\ref{eq: lower envelope d=00003D00003D1})
and (\ref{eq: upper envelope d=00003D00003D1})), the excess error
due to $b\neq0$ only contributes to increasing the factor multiplying
$(|a|+\rho)\omega$. Thus, taking $F_{\pm}(a_{t},b_{t})$ as envelopes
of one dimensional iteration for $a_{t}$ (with $b_{t}$ acting as
bounded disturbance), we obtain that, orderwise, the statistical error
and convergence time of $a_{t}$ are the same as for $\theta_{t}$
in the one-dimensional empirical iteration given in Theorem \ref{thm: empirical d=00003D00003D1 known delta}.
Thus, after the convergence time specified in Theorem \ref{thm: empirical d=00003D00003D1 known delta},
the error is $|a_{t}-\eta|\leq C_{5}\min\left\{ \frac{\omega}{\rho},\frac{\omega}{\eta}\right\} $
where $C_{5}\leq\gl[C]_{2}^{(1)}(1+C_{4})$ (Note, however, that $\omega=\sqrt{\gl[C]_{\omega}\frac{d\log n}{n}}$with
$d>1$). We denote this convergence time by $T_{a}$.

\uline{Step 3 (Refinement of the orthogonal iteration):} At the
end of the previous step, it was shown that the error in $|a_{t}-\eta|\lesssim\min\left\{ \frac{\omega}{\rho},\frac{\omega}{\eta}\right\} $,
whereas in the proceeding step it was shown that $b_{t}\lesssim\frac{\omega}{\rho}$.
We next refine the latter bound to $b_{t}\lesssim\frac{\omega}{\eta}$
in case $\eta>\rho$. Assume that $C_{f}\geq\frac{2}{3}C_{5}C_{l}$
so that after the previous step it holds that $a_{t}\in[\frac{\eta}{2},2\eta]$
for all $t\geq T_{a}$. Thus, it holds that 
\[
\frac{a^{2}+b^{2}}{2+4(a^{2}+b^{2})}\geq\frac{\eta^{2}/4}{2+4(4\gl[C]_{\theta}^{2}+c_{1}\frac{\omega^{2}}{\rho^{2}})}\geq C_{7}\eta^{2}
\]
for some $C_{7}>0$. Utilizing the bound of Proposition \ref{prop: Comparison theorem d>1}
item \ref{enu: comparison theorem d>1  - G itertaion explicit}, we
get 
\begin{align*}
b_{t+1} & \leq b_{t}\left(1-C_{7}\eta^{2}-\gl[C]_{1}^{(d)}\rho^{2}\right)+\max\left\{ \left|a_{t}\right|+b_{t},\rho\right\} \cdot\omega\\
 & \leq b_{t}\left(1-C_{8}(\eta^{2}+\rho^{2})\right)+C_{9}\eta\cdot\omega
\end{align*}
where $C_{8}\leq C_{7}+\gl[C]_{1}^{(d)}$ and $C_{9}\leq\max\{2\gl[C]_{\theta},1+\frac{c_{1}}{C_{1}^{2}}\}$.
From convergence properties of one-dimensional iterations, (Proposition
\ref{prop: Properties of one dimensional iterations}, item \ref{enu: general convergence in one-dim - convergence time of contraction plus constant})
$b_{t}\leq\frac{8C_{9}}{C_{8}}\cdot\frac{\omega}{\eta}$ for all $t\geq T_{a}+T_{b}$
where 
\[
T_{b}\leq\frac{2}{C_{8}\eta^{2}}\cdot\log\left(\frac{8C_{9}}{C_{8}}\frac{\omega}{\eta}\right)\,.
\]

\paragraph*{Small signal case}

Assume that $\eta<C_{0}\frac{\omega}{\rho}$. As the signal is small,
we show that the EM iteration remains small for all iterations. There
are only two steps - an analysis of the orthogonal iteration (without
assuming that $a_{t}$ is positive), and then the signal iteration.

\uline{Step 1 (Orthogonal iteration):}\textbf{ }In this case $a_{t}<0$
is possible, and so we cannot use the dominance relation to the balanced
iteration (Proposition \ref{prop: Comparison theorem d>1} item \ref{enu: comparison theorem d>1  - G itertaion dominance}).
However, since the signal is small, we may use Taylor expansion to
relate $G(a,b\mid\eta,\delta)$ to $G(a,b\mid0,\delta)$ (Proposition
\ref{prop: Comparison theorem d>1} item \ref{enu: comparison theorem d>1  - G iteration dominance eta}),
and then use the fact that for $\eta=0$, the weight $\delta$ does
not affect the iteration at all. Specifically, Proposition \ref{prop: Comparison theorem d>1},
items \ref{enu: comparison theorem d>1  - G itertaion dominance}
and \ref{enu: comparison theorem d>1  - G iteration dominance eta},
along with the assumption $\rho>C_{1}\sqrt{\omega}$ while requiring
that $C_{1}\geq\left(\frac{2\gl[C]_{G,\eta}^{(d)}}{\gl[C]_{1}^{(d)}}\right)^{1/4}\sqrt{C_{0}}$
imply that 
\begin{align*}
G(a,b\mid\eta,\delta) & \leq G(a,b\mid0,\delta)+b\gl[C]_{G,\eta}^{(d)}C_{0}^{2}\frac{\omega^{2}}{\rho^{2}}\\
 & \leq b\left(1-\frac{a^{2}+b^{2}}{2+4(a^{2}+b^{2})}-\gl[C]_{G,\rho}^{(d)}\rho^{2}+\gl[C]_{G,\eta}^{(d)}C_{0}^{2}\frac{\omega^{2}}{\rho^{2}}\right)\\
 & \leq b\left(1-\frac{a^{2}+b^{2}}{2+4(a^{2}+b^{2})}-\frac{\gl[C]_{G,\rho}^{(d)}}{2}\rho^{2}\right)
\end{align*}
for all $(a,b,\eta)\in[-C_{f},C_{f}]^{2}\times[0,\gl[C]_{\theta}]$.
Thus, the bound (\ref{eq: bound on G iteration empirical iteration proof})
holds for this case too (note that the error $\omega(a_{t}+b_{t}+\beta)$
with mixed signs for $a_{t}<0$ and $b_{t}>0$ is only lower than
both being positive). Hence, for all $(a,b,\eta)\in[-C_{f},C_{f}]^{2}\times[0,\gl[C]_{\theta}]$
it holds that 
\[
b_{t+1}\leq b_{t}\left(1-\frac{\gl[C]_{G,\rho}^{(d)}}{2}\rho^{2}\right)-\frac{b_{t}^{3}}{C_{2}}+\frac{C_{2}\omega^{2}}{4b_{t}}+\omega\rho.
\]
Similar analysis to the large signal case then yields $b_{t}\leq c_{1}\frac{\omega}{\rho}$
for all $t$.

\uline{Step 2 (Signal iteration):} The analysis is similar to the
large signal case, which shows that $|a_{t}-\eta|\leq C_{5}\frac{\omega}{\rho}$
for $t\geq T_{a}$. The conclusion then follows since for some $C_{6}>0$
\[
\left|\theta_{t}-\theta_{*}\right|\leq\left|a_{t}-\eta\right|+b_{t}\leq C_{6}\frac{\omega}{\rho}\,.
\]
\end{proof}
We complete the proof of Theorem \ref{thm:Main result} with the case
of $\rho\lesssim\sqrt{\omega}$: 
\begin{proof}[Proof of Proposition \ref{prop: empirical d>1 known delta small rho}]
The balanced iteration $f_{n}(\theta,\delta=\frac{1}{2}\mid\theta_{*},\delta=\frac{1}{2})=\E_{n}[X\cdot\tanh(X\theta))]$
is insensitive to the actual signs generating the samples $(X_{1},\ldots,X_{n})$,
and thus the convergence result of Theorem \ref{thm:balanced case}
holds, where we note that its error guarantee is w.r.t. $\ell_{0}$.
It is thus remain to show that $s_{t}$ correctly adjusts the sign
of $\theta_{t}$. Recall that $\rho=1-2\delta$, and let $\varepsilon=\frac{1}{\rho}\E_{n}[X]-\theta_{*}$.
Under the high probability event (\ref{eq: high probability event})
\[
\|\varepsilon\|=\|\frac{1}{\rho^{2}}\E_{n}[X]-\frac{1}{\rho^{2}}\E[X]\|=\frac{1}{\rho^{2}}\|f_{n}(\theta,\delta\mid\theta_{*},\delta)-f(\theta,\delta\mid\theta_{*},\delta)\|\leq\frac{\omega}{\rho}\,.
\]
By the guarantees of the balanced iteration, there exists $c_{0}$
such that for $t$ large enough (as specified in the theorem) $\|\theta_{t}-\theta_{*}\|\leq c_{0}\frac{\omega}{\eta}$.
Thus, if $\eta>c_{1}\sqrt{\omega}$ for properly large $c_{1}$, then
$|\langle\theta_{t},\theta_{*}\rangle|\geq\frac{\omega}{\rho}\|\theta_{t}\|\geq|\langle\theta_{t},\varepsilon_{n}\rangle|$.
Hence, 
\[
s_{t}=\sgn\langle\theta_{t},\frac{1}{\rho}\E_{n}[X]\rangle=\sgn\left(\langle\theta_{t},\theta_{*}\rangle+\langle\theta_{t},\varepsilon_{n}\rangle\right)=\sgn\langle\theta_{t},\theta_{*}\rangle\,.
\]
\end{proof}

\section{Proofs for Section \ref{subsec:The-weight-iteration}}

\subsection{Population iteration}

We begin with basic properties: 
\begin{lem}
\label{lem: weight simple properties} Assume that $\rho_{*}\geq0$
and that $\langle\theta,\theta_{*}\rangle>0$. Then: 
\begin{enumerate}
\item \label{enu: weight iteration simple properties - iteration}Iteration:
The iteration is consistent $h(\rho_{*},\theta_{*})=\rho_{*},$ at
the boundaries $\lim_{\rho\uparrow1}h(\rho,\theta)=1$ and $\lim_{\rho\downarrow-1}h(\rho,\theta)=-1$.
At $\rho=0$ 
\[
h(0,\theta)\geq\rho_{*}\left[1-e^{-\langle\theta,\theta_{*}\rangle/2}\right]\,.
\]
\item First order derivative: It holds that $\frac{\d}{\d\rho}h(\rho,\theta)>0$
and so $\rho\mapsto h(\rho,\theta)$ is increasing on $[-1,1]$. At
the left boundary $\lim_{\rho\downarrow-1}\frac{\dee}{\dee\rho}h(\rho,\theta)>1$,
and at the right boundary 
\[
\lim_{\rho\uparrow1}\frac{\dee}{\dee\rho}h(\rho,\theta)=e^{2\|\theta\|^{2}}\left[\left(\frac{1+\rho_{*}}{2}\right)e^{-2\langle\theta,\theta_{*}\rangle}+\left(\frac{1-\rho_{*}}{2}\right)\cdot e^{2\langle\theta,\theta_{*}\rangle}\right]
\]
for which $\lim_{\rho\uparrow1}\frac{\dee}{\dee\rho}h(\rho,\theta)>1$
if $\|\theta\|>|\langle\hat{\theta},\theta_{*}\rangle|$. 
\item \label{enu: weight iteration simple properties - second order derivative}Second
order derivative: There exists $\overline{\rho}\leq\rho_{*}$ such
that $h(\rho)$ is strictly concave on $[-1,\overline{\rho}]$ and
strictly convex on $[\overline{\rho},1]$. Consequently, $\frac{\dee h(\rho,\theta)}{\dee\rho}$
is strictly decreasing on $[-1,\overline{\rho}]$ and strictly increasing
on $[\overline{\rho},1]$. 
\item \label{enu: weight iteration simple properties - derivative from zero to rho_star}Contractivity
at $\rho\in[0,\rho_{*}]$: If $\|\theta\|>|\langle\hat{\theta},\theta_{*}\rangle|$
then 
\[
\max_{\rho\in[0,\rho_{*}]}\frac{\dee h(\rho,\theta)}{\dee\rho}\leq e^{-|\langle\hat{\theta},\theta_{*}\rangle|^{2}/2}\cdot\max\left\{ \frac{5}{6},1-\frac{\|\theta\|^{2}}{6}\right\} \,.
\]
\item \label{enu: weight iteration simple properties - second derivative at theta_star}Bounded
second derivative for $\theta=\theta_{*}$: $\max_{\rho\in(0,\gl[C]_{\rho})}\frac{\dee^{2}h(\rho,\theta_{*})}{\dee\rho^{2}}\leq\gl[C]''_{h}\cdot\eta^{2}$
for $\gl[C]''_{h}=\frac{32+36(4\gl[C]_{\theta}^{2}+1)}{(1-\gl[C]_{\rho}^{2})^{2}}$. 
\end{enumerate}
\end{lem}

\begin{proof}
Let $Z\sim N(0,1)$ and $U\sim N(\eta,1)$ where $\eta=\langle\hat{\theta},\theta_{*}\rangle$. 
\begin{enumerate}
\item Consistency for $\theta=\theta_{*}$ is well-known and can be proved
as in the proof of Lemma \ref{lem: d=00003D00003D1 simple properties},
item \ref{enu: d=00003D00003D1 simple properties - iteration}.\textbf{
}The limits at the boundaries are immediate. At $\rho=0$ 
\begin{align*}
h(0,\theta) & =\E\left[\tanh\left(\|\theta\|V\right)\right]\\
 & =\left(\frac{1+\rho_{*}}{2}\right)\cdot\E\left[\tanh\left(\|\theta\|U\right)\right]+\left(\frac{1-\rho_{*}}{2}\right)\cdot\E\left[\tanh\left(-\|\theta\|U\right)\right]\\
 & =\rho_{*}\E\left[\tanh\left(\|\theta\|U\right)\right]\\
 & \geq\rho_{*}\left[1-e^{-\langle\theta,\theta_{*}\rangle/2}\right]\,.
\end{align*}
where the last inequality is from (\ref{eq: lower bound on expected tanh})
in Appendix \ref{subsec:Usful-results}. 
\item By direct computation 
\[
h'(\rho)\dfn\frac{\dee h(\rho,\theta)}{\dee\rho}=\frac{1}{1-\rho^{2}}\cdot\E\left[\frac{1}{\cosh^{2}\left(\|\theta\|V+\beta_{\rho}\right)}\right]=\E\left[\frac{1}{\left[(\frac{1+\rho}{2})e^{\|\theta\|V}+(\frac{1-\rho}{2})e^{-\|\theta\|V}\right]^{2}}\right]>0
\]
and so the iteration is monotonically increasing. The limit at $\rho=-1$
\begin{align*}
\lim_{\rho\downarrow-1}h'(\rho) & =\E\left[e^{2\|\theta\|V}\right]=\left(\frac{1+\rho_{*}}{2}\right)\cdot\E\left[e^{2\|\theta\|U}\right]+\left(\frac{1-\rho_{*}}{2}\right)\cdot\E\left[e^{-2\|\theta\|U}\right]\\
 & =e^{2\|\theta\|^{2}}\left[\left(\frac{1+\rho_{*}}{2}\right)e^{2\langle\theta,\theta_{*}\rangle}+\left(\frac{1-\rho_{*}}{2}\right)\cdot e^{-2\langle\theta,\theta_{*}\rangle}\right]\\
 & >1
\end{align*}
where the last inequality is since $\left(\frac{1+\rho_{*}}{2}\right)s+\left(\frac{1-\rho_{*}}{2}\right)s^{-1}$
for $s\in[1,\infty)$ is minimized for $s=1$ (assuming $\rho_{*}>0$).
Similarly, the limit at $\rho=1$ 
\begin{align}
\lim_{\rho\uparrow1}h'(\rho) & =\E\left[e^{-2\|\theta\|V}\right]=\left(\frac{1+\rho_{*}}{2}\right)\cdot\E\left[e^{-2\|\theta\|U}\right]+\left(\frac{1-\rho_{*}}{2}\right)\cdot\E\left[e^{2\|\theta\|U}\right]\nonumber \\
 & =e^{2\|\theta\|^{2}}\left[\left(\frac{1+\rho_{*}}{2}\right)e^{-2\langle\theta,\theta_{*}\rangle}+\left(\frac{1-\rho_{*}}{2}\right)\cdot e^{2\langle\theta,\theta_{*}\rangle}\right]\,.\label{eq: derivative of h at rho=00003D00003D1}
\end{align}
\item By direct computation 
\[
h''(\rho)\dfn\frac{\dee^{2}h(\rho,\theta)}{\dee\rho^{2}}=\frac{2}{(1-\rho^{2})^{2}}\cdot\E\left[\frac{\rho-\tanh\left(\|\theta\|V+\beta_{\rho}\right)}{\cosh^{2}\left(\|\theta\|V+\beta_{\rho}\right)}\right]\,,
\]
and evidently, $\lim_{\rho\uparrow1}h''(\rho)>0$. At $\rho=0$, 
\[
h''(0)=-2\cdot\E\left[\frac{\tanh\left(\|\theta\|V\right)}{\cosh^{2}\left(\|\theta\|V\right)}\right]<0\,,
\]
since $\P[V=v]>\P[V=-v]$ for any $v>0$ (see (\ref{eq: GM positive vs negative probability})
in Appendix \ref{subsec:Usful-results}) and $\tanh$ is an odd function.
To show that $h(\rho)$ changes its curvature from convex to concave
as $\rho$ increases from $-1$ to $1$ only a single time at some
$\overline{\rho}$, we note that: 
\[
h'''(\rho)\dfn\frac{\dee^{3}h(\rho,\theta)}{\dee\rho^{3}}=\frac{3}{(1-\rho^{2})^{2}}\cdot\E\left[\frac{\cosh\left(2\|\theta\|V\right)-1}{\cosh^{4}\left(\|\theta\|V+\beta_{\rho}\right)}\right]>0\,.
\]
Thus, $h''(\rho)$ is monotonically increasing. The fact that $\overline{\rho}\leq\rho_{*}$
follows from $h''(\rho_{*})>0$ but we omit the full proof since this
property is inconsequential for further analysis. 
\item We prove the claimed bound at the edge points of the interval $[0,\rho_{*}]$,
and then the same bound holds at the interior of the interval since
the property of $h''(\rho)$ stated in item \ref{enu: weight iteration simple properties - second order derivative}
implies that $\max_{\rho\in[0,\rho_{*}]}h'(\rho)$ is bounded by its
values at the edge points. At $\rho=0$, let $V_{\frac{1}{2}}\sim\frac{1}{2}\cdot N(\eta,1)+\frac{1}{2}\cdot N(-\eta,1)$
be a balanced version of $V$, where we recall that $\eta=\langle\hat{\theta},\theta_{*}\rangle$
here. Then, 
\begin{align}
h'(0) & =\E\left[\frac{1}{\cosh^{2}\left(\|\theta\|V\right)}\right]\nonumber \\
 & \trre[=,a]\E\left[\frac{1}{\cosh^{2}(\|\theta\|V_{\frac{1}{2}})}\right]\nonumber \\
 & \trre[=,b]e^{-\eta^{2}/2}\cdot\E\left[\frac{1}{\cosh(\|\theta\|Z)}\cdot\frac{e^{\eta Z}+e^{-\eta Z}}{e^{\|\theta\|Z}+e^{-\|\theta\|Z}}\mid Z>0\right]\nonumber \\
 & \trre[\leq,c]e^{-\eta^{2}/2}\cdot\E\left[\frac{1}{\cosh(\|\theta\|Z)}\right]\nonumber \\
 & \trre[\leq,d]e^{-\eta^{2}/2}\cdot\left(1-\frac{\|\theta\|^{2}}{2(1+2\|\theta\|^{2})}\right)\nonumber \\
 & \leq e^{-\eta^{2}/2}\cdot\begin{cases}
\frac{5}{6}, & \|\theta\|\leq1\\
1-\frac{\|\theta\|^{2}}{6}, & \|\theta\|>1
\end{cases}\,,\label{eq: weight iteration bound on first derivative at zero}
\end{align}
where $(a)$ is since $\cosh^{2}(t)$ is even, $(b)$ is by a change
of measure (see (\ref{eq: change of measure}) in Appendix \ref{subsec:Usful-results})
and symmetry, $(c)$ is since $\max_{a\geq b\geq1}\frac{b+b^{-1}}{a+a^{-1}}=\max\left\{ 1,\max_{a\geq1}\frac{2}{a+a^{-1}}\right\} \leq1$
(e.g., by the inequality of arithmetic and geometric means), and $(d)$
is by \cite[eq.(125) and Lemma 24]{wu2019EM}. At $\rho=\rho_{*}$
it holds that 
\begin{align*}
h'(\rho_{*}) & =\E\left[\frac{1}{\left[\left(\frac{1+\rho_{*}}{2}\right)e^{\|\theta\|V}+\left(\frac{1-\rho_{*}}{2}\right)e^{-\|\theta\|V}\right]^{2}}\right]\\
 & \trre[=,a]e^{-\eta^{2}/2}\cdot\E\left[\frac{1}{\left(\frac{1+\rho_{*}}{2}\right)e^{\|\theta\|Z}+\left(\frac{1-\rho_{*}}{2}\right)e^{-\|\theta\|Z}}\cdot\frac{\left(\frac{1+\rho_{*}}{2}\right)e^{\eta Z}+\left(\frac{1-\rho_{*}}{2}\right)e^{-\eta Z}}{\left(\frac{1+\rho_{*}}{2}\right)e^{\|\theta\|Z}+\left(\frac{1-\rho_{*}}{2}\right)e^{-\|\theta\|Z}}\right]\\
 & \trre[\leq,b]e^{-\eta^{2}/2}\cdot\E\left[\frac{1}{\left(\frac{1+\rho_{*}}{2}\right)e^{\|\theta\|Z}+\left(\frac{1-\rho_{*}}{2}\right)e^{-\|\theta\|Z}}\right]\\
 & =e^{-\eta^{2}/2}\cdot\E\left[\frac{1}{\cosh(\|\theta\|Z+\beta)}\right]\\
 & \trre[\leq,c]e^{-\eta^{2}/2}\cdot\E\left[\frac{1}{\cosh(\|\theta\|Z)}\right]\\
 & \trre[\leq,d]e^{-\eta^{2}/2}\cdot\begin{cases}
\frac{5}{6}, & \|\theta\|\leq1\\
1-\frac{\|\theta\|^{2}}{6}, & \|\theta\|>1
\end{cases}\,,
\end{align*}
where $(a)$ is again by a change of measure, $(b)$ is as in (\ref{eq: mean iteration d=00003D00003D1 simple properties derivative bound})
assuming $\|\theta\|\geq\eta=\langle\hat{\theta},\theta_{*}\rangle$,
$(c)$ is since the argument inside the expectation is an even function
of $\beta$, and using arguments similar to (\ref{eq: expectation of tanh over cosh squared})
to show that 
\[
\frac{\dee}{\dee\beta}\E\left[\frac{1}{\cosh(\|\theta\|Z+\beta)}\right]=-\E\left[\frac{\tanh(\|\theta\|Z+\beta)}{\cosh(\|\theta\|Z+\beta)}\right]<0\,,
\]
and $(d)$ is as in (\ref{eq: weight iteration bound on first derivative at zero}). 
\item To bound the second derivative let $\eta=\|\theta_{*}\|$ and $V\sim(1-\delta_{*})\cdot N(\eta,1)+\delta_{*}\cdot N(-\eta,1)$.
Then, for any $\rho\in[0,\gl[C]_{\rho})$ 
\begin{align*}
h''(\rho) & =\frac{2}{(1-\rho^{2})^{2}}\cdot\E\left[\frac{\rho-\tanh\left(\eta V+\beta_{\rho}\right)}{\cosh^{2}\left(\eta V+\beta_{\rho}\right)}\right]\dfn\frac{2}{(1-\rho^{2})^{2}}\cdot a(\eta)
\end{align*}
and using $\psi_{\pm}=\eta^{2}+\eta Z\pm\beta_{\rho}$ with $Z\sim N(0,1)$
we may write 
\[
a(\eta)=(1-\delta_{*})\cdot\E\left[\frac{\rho-\tanh(\psi_{+})}{\cosh^{2}(\psi_{+})}\right]+\delta_{*}\cdot\E\left[\frac{\rho+\tanh(\psi_{-})}{\cosh^{2}(\psi_{-})}\right]\,.
\]
We will bound $a(\eta)$ by its Taylor expansion around $\eta=0$.
Note that $a(0)=0$ (since $\rho=\tanh(\beta_{\rho})$), and that
the first derivative is 
\begin{align*}
a'(\eta)=\frac{\dee a(\eta)}{\dee\eta} & =-4(1-\delta_{*})\cdot\E\left[(2\eta+Z)\cdot\frac{2+\rho\sinh(2\psi_{+})-\cosh(2\psi_{+})}{\left[1+\cosh^{2}\left(2\psi_{+}\right)\right]^{2}}\right]\\
 & \hphantom{=}+4\delta_{*}\cdot\E\left[(2\eta+Z)\cdot\frac{2+\rho\sinh(2\psi_{-})-\cosh(2\psi_{-})}{\left[1+\cosh^{2}\left(2\psi_{-}\right)\right]^{2}}\right]
\end{align*}
and so $a'(0)=0$. Next we upper bound the second derivative $a''(\eta)=\frac{\dee a^{2}(\eta)}{\dee\eta^{2}}$.
As $a'(\eta)$ in the last display is comprised from a mixture of
two expectations, we only bound the first (and the second one can
be bounded similarly by the same bound). So, 
\begin{align*}
 & \frac{\dee}{\dee\eta}\E\left[(2\eta+Z)\cdot\frac{2+\rho\sinh(2\psi_{+})-\cosh(2\psi_{+})}{\left[1+\cosh^{2}\left(2\psi_{+}\right)\right]^{2}}\right]\\
 & =2\cdot\E\left[\frac{2+\rho\sinh(2\psi_{+})-\cosh(2\psi_{+})}{\left[1+\cosh^{2}\left(2\psi_{+}\right)\right]^{2}}\right]\\
 & \hphantom{=}+\E\left[(2\eta+Z)^{2}\cdot\frac{\rho+\rho\cosh(2\psi_{+})-5\sinh(2\psi_{+})+\cosh(2\psi_{+})\sinh(2\psi_{+})-\rho\sinh^{2}(2\psi_{+})}{\left[1+\cosh^{2}\left(2\psi_{+}\right)\right]^{3}}\right]\,.
\end{align*}
Using $|\sinh(t)|\leq|\cosh(t)|$ and the triangle inequality, the
absolute value of the above expression is bounded from above by 
\[
2(3+\rho)+(3\rho+6)\cdot\E\left[(2\eta+Z)^{2}\right]\leq8+9(4\eta^{2}+1).
\]
Hence, $a''(\eta)\leq32+36(4\gl[C]_{\theta}^{2}+1)$ for all $\eta\in\gl[C]_{\theta}$
and the result follows from Taylor expansion. 
\end{enumerate}
\end{proof}
We may now prove the convergence of the population iteration. 
\begin{proof}[Proof of Theorem \ref{thm: population weight fixed mean}]
For brevity, we denote the iteration by $h(\rho)$. Let $h'(\rho)=\frac{\dee}{\dee\rho}h(\rho,\theta)$
and recall that Lemma \ref{lem: weight simple properties} states
that: $h'(-1)>1$; that there exists a $\overline{\rho}$ such that
$h'(\rho)$ is strictly decreasing in $(-1,\overline{\rho})$ and
strictly increasing in $(\overline{\rho},1)$; that $\rho=\pm1$ are
fixed points of $h(\rho)=h(\rho,\theta)$. Note also that an explicit
expression for $h'(1)$ is given in (\ref{eq: derivative of h at rho=00003D00003D1}).
We show that: 
\begin{enumerate}
\item If $h'(1)\leq1$ then $h(\rho)$ has no fixed points in $(-1,1)$. 
\item If $h'(1)>1$ then $h(\rho)$ has a unique fixed point $\rho_{\#}\in(-1,1)$. 
\end{enumerate}
In the second case, we may deduce that $h(\rho)>\rho$ for $\rho\in(-1,\rho_{\#})$
and $h(\rho)<\rho$ for $\rho\in(\rho_{\#},1)$. By Lemma \ref{lem: weight simple properties}
item 2, $h(\rho)$ is increasing, and so Proposition \ref{prop: Properties of one dimensional iterations},
item \ref{enu: general convergence in one-dim - convergence of monotonic}\textbf{
}and analogous arguments imply that the iteration $\rho_{t+1}=h(\rho_{t})$
will converge monotonically upwards (resp. downwards) to $\rho_{\#}$
if $\rho_{0}\in(-1,\rho_{\#}]$ (resp. $\rho_{0}\in[\rho_{\#},1)$).
By consistency (Lemma \ref{lem: weight simple properties}) $\rho_{\#}=\rho_{*}$
for $\theta=\theta_{*}$.

\uline{Case }%
\mbox{%
$h'(1)<1$%
}\uline{:} By the properties mentioned above, it must be that there
exists $\tilde{\rho}\in(-1,1]$ such that 
\begin{equation}
h'(\rho)\;\begin{cases}
>1, & -1\leq\rho\leq\tilde{\rho}\\
=1, & \rho=\tilde{\rho}\\
<1, & \tilde{\rho}<\rho\leq1
\end{cases}\,.\label{eq: rho derivative options first case}
\end{equation}
Assume by contradiction that $\rho_{1}\in(-1,1)$ is a fixed point.
Further assume that there are no other fixed points in $(-1,\rho_{1})$.\footnote{The fixed points of $h(\rho)$ must be isolated; see Proposition \ref{prop: Properties of one dimensional iterations},
item \ref{enu: general convergence in one-dim - alternating slope}.} Since $h(-1)>1$, Proposition \ref{prop: Properties of one dimensional iterations},
item \ref{enu: general convergence in one-dim - alternating slope}
implies that $h'(\rho_{1})\leq1$. Hence, (\ref{eq: rho derivative options first case})
implies that $\rho_{1}\geq\tilde{\rho}$, and so $h'(\rho)<1$ for
all $\rho\in(\rho_{1},1)$. But this implies 
\begin{equation}
h(1)=h(\rho_{1})+\int_{\rho_{1}}^{1}h'(\rho)\d\rho\leq\rho_{1}+(1-\rho_{1})\cdot\max_{\rho\in[\rho_{1},1]}h'(\rho)<1\label{eq: weight population iteration proof contradiction with integral}
\end{equation}
which contradicts the property $h(1)=1$.

\uline{Case }%
\mbox{%
$h'(1)>1$%
}\uline{:} In this case $h(\rho)<\rho$ for $\rho$ close enough
to $\rho=1$ from below, and as $h'(-1)>1$ then $h(\rho)>\rho$ for
$\rho$ close enough to $\rho=-1$ from above. By the intermediate
value theorem for $h(\rho)-\rho$ for $\rho\in[-1,1]$, there must
exists at least a single fixed point $\rho_{\#}$ for $h(\rho)$ in
$(-1,1)$. We show that $\rho_{\#}$ is unique. By the properties
mentioned above, it must be that there exists $-1<\tilde{\rho}_{-}<\tilde{\rho}_{+}<1$
such that 
\begin{equation}
h'(\rho)\;\begin{cases}
>1, & -1\leq\rho<\tilde{\rho}_{-}\\
=1, & \rho=\tilde{\rho}_{-}\\
<1 & \tilde{\rho}_{-}<\rho<\tilde{\rho}_{-}\\
=1, & \rho=\tilde{\rho}_{+}\\
>1, & \tilde{\rho}_{+}<\rho\leq1
\end{cases}\label{eq: rho derivative options second case}
\end{equation}
(and also $h'(\rho)$ decreases in $(-1,\overline{\rho})$ and increases
in $(\overline{\rho},1)$ where $\overline{\rho}\in(\tilde{\rho}_{-},\tilde{\rho}_{+})$).
If there are multiple fixed points in $(-1,1)$ then we denote by
$\rho_{\#}$ the minimal one. Since $h'(-1)>1$ Proposition \ref{prop: Properties of one dimensional iterations},
item \ref{enu: general convergence in one-dim - alternating slope}
implies that $h'(\rho_{\#})\leq1$, and, in fact, a similar argument
to (\ref{eq: weight population iteration proof contradiction with integral})
together with (\ref{eq: rho derivative options second case}) show
that $h'(\rho_{\#})<1$. We next separately show that there are no
fixed points in $(-1,\rho_{\#})$ and in $(\rho_{\#},1)$: 
\begin{itemize}
\item Assume by contradiction that $\rho_{1}\in(\rho_{\#},1)$ is a fixed
point, and further assume that there are no other fixed points in
$(\rho_{\#},\rho_{1})$. By Proposition \ref{prop: Properties of one dimensional iterations},
item \ref{enu: general convergence in one-dim - alternating slope},
it holds that $h'(\rho_{1})\geq1$. Since $h'(\rho)$ has (strictly)
increased from $h'(\rho_{\#})<1$ to $h'(\rho_{1})\geq1$, (\ref{eq: rho derivative options second case})
implies that $h'(\rho)$ is strictly increasing on $(\rho_{1},1)$.
Hence, $h(\rho)>\rho$ for all $\rho\in(\rho_{1},1)$. However, as
$\rho=1$ is a fixed point, and $h'(1)>1$, continuity of $h(\rho)$
implies that there exists $\rho_{1}<\tilde{\rho}<1$ such that $h(\rho)<\rho$
for $\rho\in(\tilde{\rho},1)$; a contradiction. 
\item Assume by contradiction that $\rho_{1}\in(-1,\rho_{*})$ is a fixed
point, and further assume that $\rho_{1}$ is such that there are
no other fixed points in $(-1,\rho_{1})$. Since $h(-1)>1$, Proposition
\ref{prop: Properties of one dimensional iterations}, item \ref{enu: general convergence in one-dim - alternating slope},
implies that $h'(\rho_{1})\leq1$. We consider separately the cases
$h'(\rho_{1})<1$ and $h'(\rho_{1})=1$. First, if $h'(\rho_{1})<1$,
then there exists $\overline{\rho}$ such that $h(\rho)<\rho$ for
$(\rho_{1},\overline{\rho})$. Since $h'(\rho_{\#})<1$ it holds that
there exists $\grave{\rho}$ such that $h(\rho)>\rho$ for $\rho\in(\grave{\rho},\rho_{*})$.
By the mean value theorem, there must exist at least one more fixed
point $\rho_{2}\in(\rho_{1},\rho_{\#})$ for which $h'(\rho_{2})>1$
(Proposition \ref{prop: Properties of one dimensional iterations},
item \ref{enu: general convergence in one-dim - alternating slope}).
Since $h'(\rho)$ has increased from $h'(\rho_{1})<1$ to $h(\rho_{2})>1$,
(\ref{eq: rho derivative options second case}) implies that we must
have that $\tilde{\rho}_{+}\leq\rho_{2}$. But since $\rho_{\#}>\rho_{2}$
(\ref{eq: rho derivative options second case}) implies that $h'(\rho_{\#})>1$;
a contradiction since it holds $h'(\rho_{\#})<1$. Second, if $h'(\rho_{1})=1$,
then if, in addition, $\rho_{1}=\tilde{\rho}_{+}$ then $h'(\rho)$
is strictly increasing on $(\rho_{1},1)$ which will result $h'(\rho_{\#})>1$;
a contradiction. If $h'(\rho_{1})=1$ and $\rho_{1}=\tilde{\rho}_{-}$
then a similar proof to the case $h'(\rho_{1})<1$ holds verbatim. 
\end{itemize}
\end{proof}
\begin{proof}[Proof of Proposition \ref{prop: population weight wrong mean effect}]
As discussed in the beginning of Section \ref{subsec:The-weight-iteration},
we may assume $d=1$, with $\eta=\langle\hat{\theta},\theta_{*}\rangle$.
Per the statement of the theorem, we assume that $\theta>0$ and that
there exists a fixed point $\rho_{\#}\in(-1,1)$. According to Theorem
\ref{thm: population weight fixed mean}, this fixed point must be
unique and so $h(\rho,\theta\eta,\rho_{*})>\rho$ for $\rho\in(-1,\rho_{\#})$
and $h(\rho,\theta\mid\eta,\rho_{*})<\rho$ for $\rho\in(\rho_{\#},1)$.
Consequently, the location of $\rho_{\#}$ with respect to $\rho_{*}$
may be determined by comparing $h(\rho_{*},\theta\mid\eta,\rho_{*})$
to $=h(\rho_{*},\eta\mid\eta,\rho_{*})=\rho_{*}$. Specifically, it
suffices to show that $h(\rho_{*},\theta\mid\eta,\rho_{*})>\rho_{*}$
for $\theta\in(0,\eta)$ and $h(\rho_{*},\theta\mid\eta,\rho_{*})<\rho_{*}$
for $\theta>\eta$. In the former case, this implies that $\rho_{\#}>\rho_{*}$
and in the latter case, this implies that $\rho_{\#}<\rho_{*}$ (and
that such $\rho_{\#}\in(-1,1)$ exists). To show that property, we
take similar strategy as in the analysis of the mean iteration (Proposition
\ref{prop: Comparison theorem d=00003D00003D1}), and prove this ``global''
property by exploring $h(\rho_{*},\theta\mid\eta,\rho_{*})$ as a
function of $\eta$ for a fixed $\theta$. We thus denote it here
explicitly as $k(\theta\mid\eta)\dfn h(\rho_{*},\theta\mid\eta,\rho_{*})$.
Thus, it boils down to show that for $\theta>0$ 
\begin{equation}
\begin{cases}
k(\theta\mid\eta)<k(\theta\mid\theta)=\rho_{*}, & \theta>\eta\\
k(\theta\mid\eta)>k(\theta\mid\theta)=\rho_{*}, & \theta<\eta
\end{cases}\,.\label{eq: dominance relation weight iteration}
\end{equation}
To this end, note that 
\begin{align}
 & k(\theta\mid\eta)-k(\theta\mid\theta)\nonumber \\
 & =\E\left\{ (1-\delta_{*})\left[\tanh\left(\theta U+\beta_{\delta_{*}}\right)-1\right]-\delta_{*}\left[\tanh\left(\theta U-\beta_{\delta_{*}}\right)-1\right]\right\} \nonumber \\
 & =\E\left[s(\theta U)\right]=s(\theta\eta)*\varphi(\eta)\label{eq: convolution relation for weight iteration}
\end{align}
where $\varphi(\eta)=\frac{1}{\sqrt{2\pi}}e^{-\eta^{2}/2}$ is the
Gaussian kernel and 
\[
s(u)\dfn(1-\delta_{*})\left[\tanh\left(u+\beta_{\delta_{*}}\right)-1\right]-\delta_{*}\left[\tanh\left(u-\beta_{\delta_{*}}\right)-1\right]\,.
\]
Note that $k(\theta\mid\eta)-k(\theta\mid\theta)=0$ for $\eta=\theta$.
We will show that this is a unique zero-crossing point of $\eta\mapsto k(\theta\mid\eta)-k(\theta\mid\theta)$
by analyzing $s(u)$. The function $s(u)$ has a single zero-crossing
point at $u=0$ since: (a) It can be shown by some simple algebra
that the unique root of $s(u)=0$ is $u=0$, and (b) $s(u)$ changes
from negative to positive at $u=0$ since $\lim_{u\to-\infty}s(u)=-2\rho_{*}<0$
and 
\[
\left.\frac{\d s(u)}{\d u}\right|_{u=0}=(1-2\delta_{*})\frac{1}{\cosh^{2}(\beta_{\delta_{*}})}>0\,.
\]
Thus, $s=0$ is a unique zero-crossing point of $s(u)$. As in the
proof of Proposition \ref{prop: Comparison theorem d=00003D00003D1},
the convolution relation (\ref{eq: convolution relation for weight iteration})
and the variation diminishing property of the Gaussian kernel (Proposition
\ref{prop: Gaussian variation diminishing} in Appendix \ref{subsec:Totally-positive-kernels})
imply that $\eta\mapsto k(\theta\mid\eta)-k(\theta\mid\theta)$ has
at most a single zero-crossing point as a function of $\eta\in\mathbb{R}$
(note that for the sake of the proof we allow $\eta<0$). Clearly,
this zero-crossing point can only be at $\eta=\theta$. To show that
this is indeed a zero crossing point, we show that $k(\theta\mid\eta)-k(\theta\mid\theta)$
changes from negative to positive from $\eta=0$ to $\eta\to\infty$.
Indeed, for $\eta=0$ 
\[
k(\theta\mid\eta=0)-k(\theta\mid\theta)=\E\left[\tanh\left(\theta Z+\beta_{\delta_{*}}\right)\right]-\rho_{*}<0
\]
because $k(\theta\mid\eta=0)-k(\theta\mid\theta)|_{\theta=0}=\tanh(\beta_{\delta_{*}})-\rho_{*}=0$
and 
\begin{align*}
\frac{\dee\left[k(\theta\mid\eta=0)-k(\theta\mid\theta)\right]}{\dee\theta} & =\E\left[\frac{Z}{\cosh^{2}\left(\theta Z+\beta_{\delta_{*}}\right)}\right]<0
\end{align*}
(by conditioning on $|Z|$ and using $\delta_{*}\leq\frac{1}{2}$
and $\theta,\beta_{\delta_{*}}\geq0$ ). For $\eta\to\infty$, since
$s(u)\geq0$ for all $u>0$, (\ref{eq: convolution relation for weight iteration})
implies that $k(\theta\mid\eta)-k(\theta\mid\theta)>0$ for all $\eta$
large enough. Thus, $\eta=\theta$ is a unique zero-crossing point
of $k(\theta\mid\eta)-k(\theta\mid\theta)$ and (\ref{eq: dominance relation weight iteration})
holds. 
\end{proof}

\subsection{Empirical iteration}
\begin{proof}[Proof of Theorem \ref{thm: sample weight known mean}]
Let $h_{n}(\rho)\equiv h_{n}(\rho,\theta_{*})$. Under the high probability
event (\ref{eq: high probability event}), it holds that $h_{n}(\rho)$
is sandwiched between the envelopes 
\[
h_{n}(\rho)\leq h(\rho)+\eta\omega_{1}\dfn h_{+}(\rho)
\]
and 
\[
h_{n}(\rho)\geq h(\rho)-\eta\omega_{1}\dfn h_{-}(\rho)
\]
and where $\eta=\|\theta_{*}\|$. Thus, for the weight iteration,
the empirical error is absolute, i.e., comprised of an additive term
$\eta\omega_{1}$ which does not depend on $\rho$. Furthermore, the
truncated empirical iteration $[h_{n}(\rho)]_{\gl[C]_{\rho}}$ is
bounded by the truncated envelopes $[h_{\pm}(\rho)]_{\gl[C]_{\rho}}$.
The truncation does not affect the analysis of the lower envelope,
but will be used for the error analysis of the upper envelope.

By repeating the arguments in the proof of Theorem \ref{thm: population weight fixed mean},
it can be shown that $h_{-}(\rho)$ has only two fixed points in $[-1,1]$,
denoted here by $\rho_{-}$ and $\underline{\rho}$, such that $\rho_{-}\uparrow\rho_{*}$
and $\underline{\rho}\downarrow-1$ as $\omega_{1}\to0$ (or, $n\to\infty$).\footnote{Indeed, the proof Theorem \ref{thm: population weight fixed mean}
mostly uses the first order derivative $h'(\rho)$ which is not changed
by absolute errors. } Hence, $h(\rho)>\rho$ for $\rho\in(\underline{\rho},\rho_{-})$,
and if the iteration is initialized in $\rho_{0}\in(\underline{\rho},\rho_{-})$
it will converge to $\rho_{-}$. We next show that this holds for
initialization at $\rho_{0}=0$. It is readily verified that $\eta\mapsto\frac{1-e^{-\eta^{2}/2}}{\eta^{2}}$
is an even function of $\eta$ which is strictly decreasing for $\eta>0$.
Since $\eta\leq\gl[C]_{\theta}$ it holds that $1-e^{-\eta^{2}/2}\geq C_{1}\eta^{2}$
for $C_{1}=\frac{1-e^{-\gl[C]_{\theta}^{2}/2}}{\gl[C]_{\theta}^{2}}>0$.
Then, Lemma \ref{lem: weight simple properties}, item \ref{enu: weight iteration simple properties - iteration}
implies that if $\eta>\frac{2}{C_{1}}\cdot\frac{\omega_{1}}{\rho_{*}}$
then 
\[
h_{-}(0)\geq\rho_{*}\left[1-e^{-\eta^{2}/2}\right]-\eta\omega_{1}\geq C_{1}\rho_{*}\eta^{2}-\eta\omega_{1}>0\,.
\]
Hence the iteration $\rho_{t+1}=h_{-}(\rho_{t})$ with $\rho_{0}=0$
will converge to $\rho_{-}$. Analogous claims hold for the upper
envelope for which clearly $h_{+}(0)>0$. We next bound the errors
$\rho_{*}-\rho_{-}$ and $\rho_{+}-\rho_{*}$ and the convergence
times of the envelops. For the lower envelope, the truncation is inconsequential.
Lemma \ref{lem: weight simple properties}, item \ref{enu: weight iteration simple properties - derivative from zero to rho_star}
implies that $h_{-}'(\rho)\leq\min\{e^{-1},1-\frac{\eta^{2}}{12}\}$
for all $\rho\in[0,\rho_{*}]$, and so 
\begin{align*}
\rho_{-} & =h_{-}(\rho_{-})\\
 & =h(\rho_{-})-\eta\omega_{1}\\
 & =h(\rho_{*})-\int_{\rho_{-}}^{\rho_{*}}h'(\rho)\d\rho-\eta\omega_{1}\\
 & =\rho_{*}-\int_{\rho_{-}}^{\rho_{*}}h'(\rho)\d\rho-\eta\omega_{1}\\
 & \geq\rho_{*}-\min\{e^{-1},1-\frac{\eta^{2}}{12}\}\cdot(\rho_{*}-\rho_{-})-\eta\omega_{1}\,.
\end{align*}
Thus, the error is at most 
\[
\rho_{*}-\rho_{-}\leq\frac{\eta\omega_{1}}{1-\max\{e^{-1},1-\frac{\eta^{2}}{12}\}}\leq\max\left\{ 12\frac{\omega_{1}}{\eta},2\gl[C]_{\theta}\omega_{1}\right\} \leq12\gl[C]_{\theta}^{2}\cdot\frac{\omega_{1}}{\eta}\,.
\]
Further note that as $\rho_{*}>\frac{2}{C_{1}}\cdot\frac{\omega_{1}}{\eta}$
was assumed, this also implies that $\rho_{-}>\frac{\rho_{*}}{2}$.
We now turn to the convergence time. Again, since $h_{-}'(\rho)\leq\min\{e^{-1},1-\frac{\eta^{2}}{12}\}<1$
for all $\rho\in(0,\rho_{-})$, Proposition \ref{prop: Properties of one dimensional iterations},
item \ref{enu: general convergence in one-dim - convergence time of contraction},
implies that $|\rho_{t}-\rho_{-}|\leq\frac{\omega_{1}}{\eta}$ for
all 
\[
t\geq\frac{1}{1-\max\{e^{-1},1-\frac{\eta^{2}}{12}\}}\log\left[\frac{\omega_{1}}{\eta\cdot\rho_{-}}\right]\,.
\]
The r.h.s. is at most $12\gl[C]_{\theta}^{2}\eta^{-2}\log(C_{1})$.

For the upper envelope, we use again Lemma \ref{lem: weight simple properties},
item \ref{enu: weight iteration simple properties - derivative from zero to rho_star}
implies that $h'(\rho_{*})\leq\max\{e^{-1},1-\frac{\eta^{2}}{12}\}$.
Furthermore, by the truncation operation $\rho\leq\gl[C]_{\rho}$
and Lemma \ref{lem: weight simple properties}, item \ref{enu: weight iteration simple properties - second derivative at theta_star}
imply that $h''(\rho)\leq\gl[C]''_{h}\eta^{2}$. Since $\rho_{+}-\rho_{*}\to0$
as $n\to\infty$ there exists $n_{0}$ such that $\frac{\gl[C]''_{h}}{2}(\rho_{+}-\rho_{*})\leq\frac{1}{24}$
and so by Taylor expansion, for any $\rho\in[\rho_{*},\rho_{+}]$
\begin{align*}
h'(\rho) & \leq h'(\rho_{*})+\gl[C]''_{h}(\rho-\rho_{*})\eta^{2}\\
 & \leq\max\{e^{-1},1-\frac{\eta^{2}}{24}\}\,.
\end{align*}
With this bound on the first derivative, the analysis is similar to
the one made for the lower envelope. 
\end{proof}

\appendix

\section{A proof for the concentration inequality of Section \ref{subsec:Concentration-of-the}
\label{sec: concentration proof}}

The proof of Theorem \ref{thm: Error concentration} follows \cite[Proof of Theorem 4]{wu2019EM},
where here, uniform convergence of the relative error should be assured
for all possible $\beta_{\rho}$, and uniform convergence of the error
is also established for the weight iteration. For legibility of the
proof, we summarize the required bounds on the moments and the tail
bounds in the following lemma. 
\begin{lem}
\label{lem: concentration}Let $X\sim P_{\theta_{*},\rho_{*}}$ for
arbitrary $\rho_{*}\in(-1,1)$ and $\theta_{*}\in\mathbb{R}^{d}$.
There exists an absolute constant $c>0$ and $n_{0}\in\mathbb{N}$
such that the following holds: 
\begin{enumerate}
\item Population moments: 
\[
\E\|X\|^{2}=d+\|\theta_{*}\|^{2}\,,
\]
and 
\[
\E\|X\|^{3}\leq c(\|\theta_{*}\|+\sqrt{d})^{3}\,.
\]
\item Concentration of empirical moments: 
\[
\P\left[\E_{n}\left[\|X\|^{2}\right]>2\|\theta_{*}\|^{2}+10d\right]\leq e^{-dn}
\]
and for all $n>n_{0}$ 
\[
\P\left[\E_{n}\left[\|X\|^{3}\right]>4\|\theta_{*}\|^{3}+16d^{3/2}+16n^{3/2}\right]\leq e^{-cn}\,.
\]
\item Concentration of projections: For any $u,v\in\mathbb{S}^{d-1}$ and
$b>0$ and $n>0$ 
\[
\P\left[\E\left[\langle u,X\rangle\right]-\E_{n}\left[\langle u,X\rangle\right]>\sqrt{\left(1+\|\theta_{*}\|^{2}\right)\frac{bd\log n}{n}}\right]\leq2\exp\left(-cbd\log n\right)\,,
\]
and for any $n\geq bd\log n$ 
\[
\P\left[\left|\E\left[\langle u,X\rangle\langle v,X\rangle\right]-\E_{n}\left[\langle u,X\rangle\langle v,X\rangle\right]\right|>(1+\|\theta_{*}\|^{2})\sqrt{\frac{bd\log n}{n}}\right]\leq2\exp\left(-cbd\log n\right)\,.
\]
\item Concentration of empirical EM iterations at a single point: For any
$u,v\in\mathbb{S}^{d-1}$, $b>0$ and $n\geq bd\log n$ 
\[
\P\left[\left|\langle u,f_{n}(\theta,\rho)\rangle-\langle u,f(\theta,\rho)\rangle\right|\geq(\|\theta\|+\beta_{\rho})\cdot(1+\|\theta_{*}\|^{2})\sqrt{\frac{bd\log n}{n}}\right]\leq2\exp\left(-cbd\log n\right)\,.
\]
\end{enumerate}
\end{lem}

\begin{proof}
~ 
\begin{enumerate}
\item The second moment follows from direct computation. For the third moment,
we use $\|X\|\leq\|\theta_{*}\|+\|Z\|$ where $Z\sim N(0,I_{d})$.
By \cite[Theorem 3.1.1]{vershynin2018high} $\|\|Z\|-\sqrt{d}\|_{\psi_{2}}\lesssim1$
and so also $\|\|Z\|\|_{\psi_{2}}\lesssim\sqrt{d}.$ Hence, $\|\|X\|\|_{\psi_{2}}\lesssim\|\theta_{*}\|+\sqrt{d}.$
The results then follows the moment property of the sub-gaussian $\|X\|$
\cite[Proposition 2.5.2]{vershynin2018high}. 
\item Since $\E_{n}\left[\|X\|^{2}\right]\leq2\|\theta_{*}\|^{2}+2\E_{n}\left[\|Z\|^{2}\right]$,
where $n\E_{n}\left[\|Z\|^{2}\right]\sim\chi_{dn}^{2}$ using the
$\chi^{2}$ tail bound in (\ref{eq: chi-square tail bound}) (Appendix
\ref{subsec:Usful-results}) it holds that 
\[
\P\left[\E_{n}\left[\|Z\|^{2}\right]\geq5d\right]\leq e^{-dn}\,.
\]
Hence, 
\[
\P\left[\E_{n}\left[\|X\|^{2}\right]>2\|\theta_{*}\|^{2}+10d\right]\leq e^{-dn}\,.
\]
For the third moment, 
\[
\E_{n}\left[\|X\|^{3}\right]\leq4\|\theta_{*}\|^{3}+4\E_{n}\left[\|Z\|^{3}\right]\leq4\|\theta_{*}\|^{3}+4\left(\max_{i\in[n]}\|Z_{i}\|\right)^{3}\,.
\]
Since $\|\|Z_{i}\|-\sqrt{d}\|_{\psi_{2}}\lesssim1$ we have $\P[\|Z_{i}\|-\sqrt{d}>\sqrt{n}]\leq e^{-c_{0}n}$
for some $c_{0}>0$. By the union bound, there exists $n_{0}(c_{0})$
and $c_{1}>0$ such that for all $n>n_{0}$ 
\[
\P\left[\max_{i\in[n]}\|Z_{i}\|\geq\sqrt{d}+\sqrt{n}\right]\leq ne^{-c_{0}n}\leq e^{-c_{1}n}\,.
\]
Thus, $\E_{n}\left[\|Z\|^{3}\right]\leq(\sqrt{d}+\sqrt{n})^{3}\leq4d^{3/2}+4n^{3/2}$
with probability larger than $1-e^{-c_{1}n}$. 
\item We note that $\|\langle u,X\rangle\|_{\psi_{2}}\leq\sqrt{1+\|\theta_{*}\|^{2}}$
for any $u\in\mathbb{S}^{n-1}$, and so the first claim follows from
sub-gaussian concentration \cite[Proposition 2.6.1]{vershynin2018high}.
Next, from \cite[Lemma 2.7.7]{vershynin2018high} 
\[
\left\Vert \langle u,X\rangle\cdot\langle v,X\rangle\right\Vert _{\psi_{1}}\leq\left\Vert \langle u,X\rangle\right\Vert _{\psi_{2}}\left\Vert \langle v,X\rangle\right\Vert _{\psi_{2}}=1+\|\theta_{*}\|^{2}
\]
and Bernstein's inequality \cite[Corollary 2.8.3]{vershynin2018high}
implies the required inequality for any $b>0$ such that $n\geq bd\log n$. 
\item For the mean iteration using the fact that product of sub-gaussian
is sub-exponential \cite[Lemma 2.7.7]{vershynin2018high} (twice)
\begin{align*}
\left\Vert \langle u,X\rangle\cdot\tanh\left(\|\theta\|\langle\hat{\theta},X\rangle+\beta_{\rho}\right)\right\Vert _{\psi_{1}} & \leq\|\theta\|\left\Vert \langle u,X\rangle\cdot\langle\hat{\theta},X\rangle\right\Vert _{\psi_{1}}+\beta_{\rho}\left\Vert \langle u,X\rangle\right\Vert _{\psi_{1}}\\
 & \leq\|\theta\|\cdot\left\Vert \langle u,X\rangle\right\Vert _{\psi_{2}}\left\Vert \langle\hat{\theta},X\rangle\right\Vert _{\psi_{2}}+\beta_{\rho}\left\Vert \langle u,X\rangle\right\Vert _{\psi_{2}}\\
 & \leq(\|\theta\|+\beta_{\rho})\cdot(1+\|\theta_{*}\|^{2})\,.
\end{align*}
Bernstein's inequality \cite[Corollary 2.8.3]{vershynin2018high}
implies the required inequality for any $b>0$ such that $n\geq bd\log n$. 
\end{enumerate}
Define for some arbitrary nonnegative constants $\{C_{j}>0\}_{j\in[3]}$
the events 
\[
{\cal E}_{n}^{(1)}(C_{1})\dfn\left\{ f_{n}(\theta,\rho)\in\mathbb{B}^{d}(C_{1}),\;\;\forall\theta\in\mathbb{R}^{d},\;\forall\rho\in\mathbb{B}(\gl[C]_{\rho})\right\} \,,
\]
\[
{\cal E}_{n}^{(2)}(C_{2},C_{1})\dfn\left\{ \|f_{n}(\theta,\rho)-f(\theta,\rho)\|\leq C_{2}\left(\|\theta\|+\beta_{\rho}\right)\cdot\sqrt{\frac{d\log n}{n}},\;\;\forall\theta\in\mathbb{B}^{d}(C_{1}),\;\forall\rho\in\mathbb{B}(\gl[C]_{\rho})\right\} \,,
\]
\[
{\cal E}_{n}^{(3)}(C_{3})\dfn\left\{ \left|h_{n}(\rho,\theta)-h(\rho,\theta)\right|\leq C_{3}\|\theta\|\cdot\sqrt{\frac{\log n}{n}},\;\;\forall\theta\in\mathbb{R}^{d},\;\forall\rho\in(-1,1)\right\} \,.
\]
To prove the theorem we will show that there exist constants $c,C,C_{3}>0$
which depend on $(\gl[C]_{\theta},\gl[C]_{\rho})$ such that for all
$n\geq Cd\log n$ 
\[
\P\left[{\cal E}_{n}^{(1)}(C_{1})\cap{\cal E}_{n}^{(2)}(C_{2},C_{1})\cap{\cal E}_{n}^{(3)}(C_{3})\right]\geq1-\frac{1}{n^{cd}}
\]
with $C_{1}\dfn5(\sqrt{d}+\gl[C]_{\theta})$ and $C_{2}=C\left(1+\gl[C]_{\theta}^{2}\right)$.
The proof is then completed using the relation $\gl[\underline{C}]_{\beta}\rho\leq|\beta_{\rho}|\leq\gl[\overline{C}]_{\beta}\rho$
(with $\gl[\underline{C}]_{\beta}$ and $\gl[\overline{C}]_{\beta}$
depending on $\gl[C]_{\rho}$).

For ${\cal E}_{n}^{(1)}(C_{1})$ we note that 
\[
\|f_{n}(\theta,\rho)\|\leq\|\E_{n}[X\tanh(\langle\theta,X\rangle+\beta)]\|\leq\E_{n}[\|X\|]\leq\sqrt{\E_{n}[\|X\|^{2}]}
\]
and then use Lemma \ref{lem: concentration}.

For ${\cal E}_{n}^{(2)}(C_{2},C_{1})$, let $\epsilon\leq\frac{1}{2}$
be given, and let ${\cal C}\subset\mathbb{S}^{d-1}$ be an $\epsilon$-net
of $\mathbb{S}^{d-1}$ in Euclidean distance, whose size satisfies
$|{\cal C}|\leq(\frac{3}{\epsilon})^{d}$ (whose existence is assured
from \cite[Corollary 4.2.13]{vershynin2018high}). By a standard argument
(e.g., \cite[Exercise 4.4.2]{vershynin2018high}) 
\[
\|f_{n}(\theta,\rho)-f(\theta,\rho)\|\leq2\cdot\max_{u\in{\cal C}}\langle u,f_{n}(\theta,\rho)-f(\theta,\rho)\rangle\,.
\]
Furthermore, any $\hat{\theta}\in\mathbb{S}^{d-1}$ may be approximated
by $v\in{\cal C}$ such that $\|\theta\|\cdot\|v-\hat{\theta}\|\leq\epsilon\|\theta\|$.
As $\tanh$ is $1$-Lipschitz 
\[
\left|\E[\langle u,X\rangle\tanh(\|\theta\|\langle\hat{\theta},X\rangle+\beta_{\rho})]-\E[\langle u,X\rangle\tanh(\|\theta\|\langle v,X\rangle+\beta_{\rho})]\right|\leq\epsilon\|\theta\|\cdot\E\|X\|^{2}\,.
\]
Repeating the same argument for the empirical iteration we get 
\begin{align*}
\|f_{n}(\theta,\rho)-f(\theta,\rho)\| & \leq2\cdot\max_{(u,v)\in{\cal C}^{2}}\langle u,f_{n}(\|\theta\|\cdot v,\rho)-f(\|\theta\|\cdot v,\rho)\rangle+\epsilon\|\theta\|\cdot\left(\E\|X\|^{2}+\E_{n}\|X\|^{2}\right)\\
 & \dfn\max_{(u,v)\in{\cal C}^{2}}\Phi(u,v,\|\theta\|,\rho).
\end{align*}
Define the sets ${\cal A}=\mathbb{B}(\epsilon)\times\mathbb{B}(\epsilon)$
and $\overline{{\cal A}}=\mathbb{B}(C_{1})\times\mathbb{B}(C_{1})\backslash{\cal A}$.
By the union bound, the required probability is bounded as: 
\begin{align}
\P\left[{\cal E}_{n}^{(2)}(C_{2},C_{1})\right] & \leq\sum_{(u,v)\in{\cal C}^{2}}\P\left[\exists(\|\theta\|,\rho)\in{\cal A}\colon\Phi(u,v,\|\theta\|,\rho)>C_{2}\left(\|\theta\|+\beta_{\rho}\right)\cdot\sqrt{\frac{d}{n}\log n}\right]\label{eq: sample concentration first probability}\\
 & \phantom{===}+\P\left[\exists(\|\theta\|,\rho)\in\overline{{\cal A}}\colon\Phi(u,v,\|\theta\|,\rho)>C_{2}\left(\|\theta\|+\beta_{\rho}\right)\cdot\sqrt{\frac{d}{n}\log n}\right]\,.\label{eq: sample concentration second probability}
\end{align}
The probability pertaining to the set ${\cal A}$ in (\ref{eq: sample concentration first probability})
is analyzed as follows. Since $\tanh'\leq1$ and $|\tanh''|\leq1$
Taylor expansion of $\tanh$ around $0$ implies 
\begin{align*}
 & \left|\E\left[\langle u,X\rangle\tanh\left(\|\theta\|\langle v,X\rangle+\beta_{\rho}\right)\right]-\E\left[\langle u,X\rangle\left(\|\theta\|\langle v,X\rangle+\beta_{\rho}\right)\right]\right|\\
 & \leq\|\theta\|^{2}\cdot\E\left[\left|\langle u,X\rangle\right|\langle v,X\rangle^{2}\right]+\beta_{\rho}^{2}\cdot\E\left[\left|\langle u,X\rangle\right|\right]\,.
\end{align*}
Repeating the same argument for the empirical iteration, and then
using the triangle and Cauchy-Schwartz inequalities, we obtain 
\begin{align*}
\Phi(u,v,\|\theta\|,\rho) & \leq\|\theta\|\left|\E\left[\langle u,X\rangle\langle v,X\rangle\right]-\E_{n}\left[\langle u,X\rangle\langle v,X\rangle\right]\right|+\beta_{\rho}\left|\E\left[\langle u,X\rangle\right]-\E_{n}\left[\langle u,X\rangle\right]\right|\\
 & \hphantom{=}+\|\theta\|^{2}\E\left[\|X\|^{3}\right]+\beta_{\rho}^{2}\sqrt{\E\left[\|X\|^{2}\right]}+\|\theta\|^{2}\E_{n}\left[\|X\|^{3}\right]+\beta_{\rho}^{2}\sqrt{\E_{n}\left[\|X\|^{2}\right]}\\
 & \hphantom{=}+\epsilon\|\theta\|\cdot\left(\E\|X\|^{2}+\E_{n}\|X\|^{2}\right)\,.
\end{align*}
By Lemma \ref{lem: concentration}, for any given $(u,v,\|\theta\|,\beta_{\rho})\in{\cal C}^{2}\times{\cal A}$,
as long as $n\geq bd\log n$, there exists absolute constants $\{c_{i}\}$
such that 
\begin{multline*}
\P\left[\frac{\Phi(u,v,\|\theta\|,\rho)}{\|\theta\|+\beta_{\rho}}>c_{1}\left(1+\|\theta_{*}\|^{2}\right)\sqrt{\frac{bd\log n}{n}}+\epsilon\cdot\left(\|\theta_{*}\|^{3}+n^{3/2}\right)\right]\\
\leq4\exp\left(-c_{2}bd\log n\right)+\exp(-c_{2}n)+\exp(-dn)\leq\exp(-c_{3}bd\log n)\,.
\end{multline*}
We will choose $\epsilon\leq\frac{c_{4}}{n^{2}}$ with sufficiently
small $c_{4}$ and $b$ to be sufficiently large so that the probability
in (\ref{eq: sample concentration first probability}) is bounded
by $\exp(-c_{4}bd\log n)$ for $C_{2}=\left(1+\gl[C]_{\theta}^{2}\right)\sqrt{\frac{bd}{n}\log n}.$

The probability of ${\cal \overline{A}}$ in (\ref{eq: sample concentration second probability})
is analyzed as follows. Let ${\cal R}_{\beta}$ be an $\epsilon^{2}$-net
of $[-\gl[C]_{\beta},\gl[C]_{\beta}]$ of size $2\gl[C]_{\beta}\cdot\epsilon^{-2}$,
and let ${\cal R}_{\theta}$ be an $\epsilon^{2}$-net of $[0,C_{1}]$
of size $C_{1}\cdot\epsilon^{-2}$. As $\tanh$ is $1$-Lipschitz,
and by the triangle and Cauchy-Schwartz inequalities, for any $(u,v)\in{\cal C}^{2}$
and $(\|\theta\|,\beta_{\rho})\in\overline{{\cal A}}$ there exists
$(s,\gamma)\in{\cal R}_{\beta}\times{\cal R}_{\theta}$ such that
\begin{align*}
 & \left|\E[\langle u,X\rangle\tanh(\|\theta\|\langle v,X\rangle+\beta_{\rho})]-\E[\langle u,X\rangle\tanh(s\langle v,X\rangle+\gamma)]\right|\\
 & \leq\epsilon^{2}\left(\E\left[\|X\|^{2}\right]+\sqrt{\E\|X\|^{2}}\right)\\
 & \leq\epsilon(\|\theta\|+\beta_{\rho})\left(\E\left[\|X\|^{2}\right]+\sqrt{\E\|X\|^{2}}\right)
\end{align*}
where the first term in the r.h.s. (resp. second) corresponds to the
approximation of $\|\theta\|$ with $s$ (resp. $\beta$ with $\gamma$),
and the second inequality is since $(\|\theta\|,\rho)\in\overline{{\cal A}}$.
Repeating the same argument for the empirical iteration, we deduce
\begin{align*}
\Phi(u,v,\|\theta\|,\rho) & \leq\max_{(s,\gamma)\in{\cal R}_{\theta}\times{\cal R}_{\beta}}\left|\E\left[\langle u,X\rangle\cdot\tanh\left(s\langle v,X\rangle+\gamma\right)\right]-\E_{n}\left[\langle u,X\rangle\cdot\tanh\left(s\langle v,X\rangle+\gamma\right)\right]\right|\\
 & +2\epsilon(\|\theta\|+\beta_{\rho})\left(\E\left[\|X\|^{2}\right]+\E_{n}\left[\|X\|^{2}\right]+\sqrt{\E\|X\|^{2}}+\sqrt{\E_{n}\left[\|X\|^{2}\right]}\right)\\
 & \dfn\Psi(u,v,s,\rho_{\gamma})\,.
\end{align*}
By Lemma \ref{lem: concentration}, for any given $(u,v,s,\gamma)\in{\cal C}^{2}\times{\cal R}_{\theta}\times{\cal R}_{\beta}$,
and any $b>0$ such that $n\geq bd\log n$ there exists absolute constants
$\{c_{i}\}$ 
\begin{multline}
\P\left[\Psi(u,v,s,\rho_{\gamma})>(s+\gamma)\cdot(1+\|\theta_{*}\|^{2})\sqrt{\frac{bd\log n}{n}}+c_{1}\epsilon(s+\gamma)(d+\|\theta_{*}\|^{2})\right]\\
\leq2\exp\left(-c_{2}bd\log n\right)+\exp(-dn)\leq\exp(-c_{3}bd\log n)\,.\label{eq: sample concentration second probability - per norm and inv temp}
\end{multline}
We will choose $\epsilon\leq c_{4}\sqrt{\frac{d\log n}{n}}$ for sufficiently
small $c_{4}$ so that the probability in (\ref{eq: sample concentration second probability - per norm and inv temp})
is bounded by $\exp(-c_{3}bd\log n)$ for 
\[
C_{2}=c_{5}\cdot(\|\theta\|+\beta_{\rho})\cdot(1+\|\theta_{*}\|^{2})\sqrt{\frac{bd\log n}{n}}\,.
\]
By a union bound over ${\cal R}_{\theta}\times{\cal R}_{\beta}$ of
size $2\gl[C]_{\beta}C_{1}\epsilon^{-4}$, the probability in (\ref{eq: sample concentration second probability})
is upper bounded by $\exp(-c_{6}bd\log n).$ The proof is then completed
by another union bound over ${\cal C}^{2}$ whose size is $(\frac{3}{\epsilon})^{2d}$,
and taking $b$ to be large enough.

We next turn to the analysis for ${\cal E}_{n}^{(3)}(C_{3},C_{1})$,
which deals with the error of the weight iteration. Since this is,
in essence, a one-dimensional iteration, the analysis is somewhat
simpler. Since $\tanh$ is $1$-Lipschitz, for all $\theta\in\mathbb{R}^{d}$
\[
\left|\tanh(\|\theta\|u+\beta_{\rho})-\tanh(\|\theta\|v+\beta_{\rho})\right|\leq\|\theta\|\cdot|u-v|\,.
\]
Let $X^{(n)}$ be a random variable distributed according to the empirical
distribution of $\{X_{i}\}_{i=1}^{n}$. Hence, by coupling, 
\begin{align}
\left|h_{n}(\rho,\theta)-h(\rho,\theta)\right| & \leq\left|\E\left[\tanh(\|\theta\|\langle\hat{\theta},X\rangle+\beta_{\rho})-\tanh(\|\theta\|\langle\hat{\theta},X^{(n)}\rangle+\beta_{\rho})\right]\right|\nonumber \\
 & \leq\E\left[\left|\tanh(\|\theta\|\langle\hat{\theta},X\rangle+\beta_{\rho})-\tanh(\|\theta\|\langle\hat{\theta},X^{(n)}\rangle+\beta_{\rho})\right|\right]\nonumber \\
 & \leq\|\theta\|\E\left[\left|\langle\hat{\theta},X\rangle-\langle\hat{\theta},X^{(n)}\rangle\right|\right]\nonumber \\
 & \leq\|\theta\|\cdot W_{1}(\nu,\nu_{n})\label{eq: empirical error in weight iteration by Wasserstein}
\end{align}
where $W_{1}$ is the first order Wasserstein distance, $\nu={\cal L}(\langle\hat{\theta},X\rangle)$
and $\nu_{n}$ is the empirical law of $\{\langle\hat{\theta},X_{i}\rangle\}{}_{i=1}^{n}$.
Now, $\langle\hat{\theta},X\rangle\sim(1-\delta_{*})N(\langle\hat{\theta},\theta_{*}\rangle,1)+\delta_{*}N(-\langle\hat{\theta},\theta_{*}\rangle,1)$
and so $\|\langle\hat{\theta},X\rangle\|_{\psi_{2}}\leq\sqrt{1+\|\theta_{*}\|^{2}}$
for any $\hat{\theta}\in\mathbb{S}^{n-1}$. The concentration inequality
in \cite[Theorem 2, Case (1)]{fournier2015rate} with the choices
$d=1$ (the dimension of $\langle\hat{\theta},X\rangle$), $p=1$
(Wasserstein distance order) and $\alpha=2$ (for the $\psi_{\alpha}$
condition $\E[e^{\gamma|X|^{\alpha}}]<\infty)$ implies that for $x_{0}>0$
there exists $c,C>0$ such that 
\begin{equation}
\P\left[W_{1}(\nu,\nu_{n})>\sqrt{\frac{\log n}{n}}\right]\leq C\cdot\exp\left(-c\log n\right)\,.\label{eq: weight iteration error Wasserstein distance}
\end{equation}
The bound (\ref{eq: empirical error in weight iteration by Wasserstein})
and (\ref{eq: weight iteration error Wasserstein distance}) imply
that ${\cal E}_{n}^{(3)}(C_{3})$ has high probability as stated in
the theorem. 
\end{proof}

\section{Minimax rates \label{sec:Minimax-rates}}
\begin{thm}
\label{thm: mean estimation minimax}For any $d\geq2$, $n\in\mathbb{N}$
and $\eta\geq0$, let $\tilde{\theta}$ be any estimator of $\theta_{*}$
based on $\underline{X}=(X_{1},\ldots,X_{n})\stackrel{\tiny\mathrm{i.i.d.}}{\sim}P_{\theta_{*},\rho_{*}}$.
Then, for $d\leq n$ 
\begin{align}
\sup_{\tilde{\theta}(\rho_{*})}\inf_{\|\theta_{*}\|=\eta}\E_{\theta_{*},\rho_{*}}\left[\ell(\tilde{\theta},\theta_{*})\right] & \asymp\begin{cases}
\eta, & \eta\leq\frac{1}{\rho_{*}}\sqrt{\frac{d}{n}}\\
\frac{1}{\rho_{*}}\sqrt{\frac{d}{n}}, & \frac{1}{\rho_{*}}\sqrt{\frac{d}{n}}<\eta<\rho_{*}\\
\frac{1}{\eta}\sqrt{\frac{d}{n}}, & \rho_{*}<\eta<1\\
\sqrt{\frac{d}{n}}, & \eta>1
\end{cases}\label{eq: mean estimation minimax rate large rho}
\end{align}
if $\rho_{*}\geq(\frac{d}{n})^{1/4}$ and 
\begin{align*}
\sup_{\tilde{\theta}(\rho_{*})}\inf_{\|\theta_{*}\|=\eta}\E_{\theta_{*},\rho_{*}}\left[\ell(\tilde{\theta},\theta_{*})\right] & \asymp\begin{cases}
\eta, & \eta\leq\left(\frac{d}{n}\right)^{1/4}\\
\frac{1}{\eta}\sqrt{\frac{d}{n}}, & \left(\frac{d}{n}\right)^{1/4}<\eta<1\\
\sqrt{\frac{d}{n}}, & \eta>1
\end{cases}\,.
\end{align*}
if $\rho_{*}\leq(\frac{d}{n})^{1/4}$. 
\end{thm}

\begin{proof}
~

\paragraph*{Upper bounds}

The error rates in all cases except for the second case in (\ref{eq: mean estimation minimax rate large rho})
were shown to be achieved by a spectral method \cite[Appendix B]{wu2019EM}.
Specifically, in case $\rho_{*}\leq(\frac{d}{n})^{1/4}$ then the
knowledge of the weight can be completely ignored by the estimator.
Furthermore, the same method achieves $\frac{1}{\eta}\sqrt{\frac{d}{n}}$
in the third case of (\ref{eq: mean estimation minimax rate large rho}).
We next show that an error rate of $\frac{1}{\rho_{*}}\sqrt{\frac{d}{n}}$
is also achievable by the estimator $\tilde{\theta}(\rho_{*})=\frac{1}{\rho_{*}}\E_{n}[X]$.
Indeed, let $X=S\theta_{*}+Z$ as in (\ref{eq: Guassian mixture model - random variables}).
Then, 
\begin{align*}
\E_{\theta_{*},\rho_{*}}\left[\ell(\tilde{\theta}(\rho_{*}),\theta_{*})\right] & \leq\E\|\tilde{\theta}(\rho_{*})-\theta_{*}\|\\
 & \leq\frac{\|\theta_{*}\|}{\rho_{*}}\E\left[\left|\E_{n}[S]-\rho_{*}\right|\right]+\frac{1}{\rho_{*}}\E\left[\left\Vert \E_{n}[Z]\right\Vert \right]\\
 & \lesssim\frac{\|\theta_{*}\|}{\rho_{*}\sqrt{n}}+\frac{1}{\rho_{*}}\sqrt{\frac{d}{n}}\\
 & \leq\frac{1}{\rho_{*}}\sqrt{\frac{d}{n}}
\end{align*}
where the penultimate asymptotic inequality follows from: (a) For
the first term, as $\|\E_{n}[\I\{S=-1\}]-\delta_{*}\|_{\psi_{2}}\leq\frac{c_{1}}{\sqrt{n}}$
for some universal constant $c_{1}>0$ \cite[Example 2.5.8. and Proposition 2.6.1]{vershynin2018high},
and so 
\[
\E\left[\left\Vert \E_{n}[S]-\rho_{*}\right\Vert \right]=2\E\left[\left|\E_{n}[\I\{S=-1\}]-\delta_{*})\right|\right]\leq\frac{c_{2}}{\sqrt{n}}
\]
for some universal constant $c_{2}>0$ \cite[Proposition 2.5.2]{vershynin2018high}.
(b) For the second term, similarly, $\|\|\E_{n}[Z]\|\|_{\psi_{2}}\leq c_{3}\sqrt{\frac{d}{n}}$
for some universal constant $c_{3}>0$ as in Lemma \ref{lem: concentration},
and using \cite[Proposition 2.6.1]{vershynin2018high}.

\paragraph*{Lower bounds}

The proof follows \cite[Appendix B]{wu2019EM}, which uses Fano's
method \cite{yang1999information} for all cases which are not lower
bounded by the $\ell_{2}$ error-rate of the standard Gaussian location
model $\min\{\eta,\sqrt{\frac{d}{n}}\}$. Thus, we mainly highlight
the main difference and omit all other details. First note that if
$|\rho_{*}|\geq\frac{1}{2}$ (say), then the lower bound (\ref{eq: mean estimation minimax rate large rho})
us again equivalent to the $\ell_{2}$ error-rate of the standard
Gaussian location model, and thus no proof is required. Thus, we may
henceforth only consider the case $|\rho_{*}|\leq\frac{1}{2}$. The
lower bound in \cite{wu2019EM} is based on a lemma which is here
generalized from $\rho=0$ to any $\rho\in(-1,1)$ as follows: 
\end{proof}
\begin{lem}[{{Generalization of \cite[Lemma 27]{wu2019EM}}}]
\label{lem:minimax rate -bounding KL divergence}Let $0\leq\eta\leq1$
and $|\rho|\leq\frac{1}{2}$. Then there exists a universal constant
$C$ such that for any $d\geq1$ and $u,v\in\mathbb{S}^{d-1}$ 
\[
\dkl(P_{\eta\cdot u,\rho}||P_{\eta\cdot v,\rho})\leq C\cdot\ell^{2}(u,v)\cdot\eta^{2}(\eta^{2}+\rho^{2})\,.
\]
\end{lem}

\begin{proof}
By symmetry, it suffice to prove 
\[
\dkl(P_{\eta\cdot u,\rho}||P_{\eta\cdot v,\rho})\leq C\cdot\|\hat{\theta}_{1}-\hat{\theta}_{2}\|^{2}\cdot\eta^{2}(\eta^{2}+\rho^{2})\,,
\]
and by rotational invariance of the normal distribution it can be
assumed that $v=e_{1}=(1,0,\ldots,0)$. Let $\lambda=\max\left\{ 1-u_{1},\|u_{\perp}\|\right\} <1$
where $u_{\perp}=(u_{2},\ldots,u_{d})$ (and similar notation will
be used for any $d$-dimensional vector). Further, let $Q$ be the
distribution of $X=(X_{1},\ldots X_{d})\in\mathbb{R}^{d}$ under $\theta_{*}=\eta v=\eta e_{1}$,
to wit $Q=Q_{X_{1},\ldots X_{d}}=P_{\eta,\rho_{*}}\otimes N(0,I_{d-1})$
(which is a product distribution), and let $P$ be the corresponding
distribution under $\theta_{*}=\eta u$. From the chain rule of the
KL divergence 
\[
\dkl(P_{\eta\cdot u,\rho}||P_{\eta\cdot v,\rho})=\dkl(P_{X_{1}}||Q_{X_{1}})+\E_{P_{X_{1}}}\left[\dkl(P_{X_{\perp}\mid X_{1}}||N(0,I_{d-1}))\right]\dfn(\text{I})+(\text{II}).
\]
We bound the two KL divergence terms using the corresponding chi-square
divergence.

\paragraph*{Bounding $(\text{I})$}

In one dimension, 
\begin{align*}
p_{\eta,\rho}(x) & =e^{-\eta^{2}/2}\varphi(x)\left[\left(\frac{1+\rho}{2}\right)e^{\theta x}+\left(\frac{1-\rho}{2}\right)e^{-\theta x}\right]\\
 & =e^{-\eta^{2}/2}\varphi(x)\left[\cosh(\eta x)+\rho\sinh(\eta x)\right]\,.
\end{align*}
Hence, for $\epsilon=\eta\lambda$ 
\begin{align}
(\text{I}) & =\dkl(P_{X_{1}}||Q_{X_{1}})\nonumber \\
 & \leq\dchis(P_{\eta-\epsilon,\rho}||Q_{\eta,\rho})\nonumber \\
 & =e^{\eta^{2}/2}\cdot\int\varphi(x)\frac{\left[e^{-(\eta-\epsilon)^{2}/2}\left(\cosh\left((\eta-\epsilon)x\right)+\rho\sinh\left((\eta-\epsilon)x\right)\right)-e^{-\eta^{2}/2}\left(\cosh\left(\eta x\right)+\rho\sinh\left(\eta x\right)\right)\right]^{2}}{\cosh(\eta x)+\rho\sinh(\eta x)}\d x\nonumber \\
 & \trre[\leq,a]\sqrt{\frac{4e}{3}}\cdot\int\varphi(x)\left[e^{-(\eta-\epsilon)^{2}/2}\left(\cosh\left((\eta-\epsilon)x\right)+\rho\sinh\left((\eta-\epsilon)x\right)\right)-e^{-\eta^{2}/2}\left(\cosh\left(\eta x\right)+\rho\sinh\left(\eta x\right)\right)\right]^{2}\d x\nonumber \\
 & \trre[=,b]\sqrt{\frac{4e}{3}}\cdot e^{-(\eta-\epsilon)^{2}}\int\varphi(x)\left[\cosh^{2}\left((\eta-\epsilon)x\right)+\rho^{2}\sinh^{2}\left((\eta-\epsilon)x\right)\right]\d x\nonumber \\
 & \hphantom{-}-\sqrt{\frac{4e}{3}}\cdot2e^{-(\eta-\epsilon)^{2}/2-\eta^{2}/2}\int\varphi(x)\left[\cosh\left((\eta-\epsilon)x\right)\cosh\left(\eta x\right)+\rho^{2}\sinh\left((\eta-\epsilon)x\right)\sinh\left(\eta x\right)\right]\d x\nonumber \\
 & \hphantom{-}+\sqrt{\frac{4e}{3}}\cdot e^{-\eta^{2}/2}\int\varphi(x)\left[\cosh^{2}\left(\eta x\right)+\rho^{2}\sinh^{2}\left(\eta x\right)\right]^{2}\d x\nonumber \\
 & \trre[=,c]\sqrt{\frac{4e}{3}}\cdot\left[\cosh\left((\eta-\epsilon)^{2}\right)+\cosh\left(\eta^{2}\right)-2\cosh\left(\eta(\eta-\epsilon)\right)\right]\nonumber \\
 & \hphantom{=}+\sqrt{\frac{4e}{3}}\left[\sinh\left((\eta-\epsilon)^{2}\right)+\sinh\left(\eta^{2}\right)-2\sinh\left(\eta(\eta-\epsilon)\right)\right]\rho^{2}\nonumber \\
 & \leq C_{1}\epsilon^{2}(\eta^{2}+\rho^{2})=C_{1}\lambda^{2}\eta^{2}(\eta^{2}+\rho^{2})\,,\label{eq: minimax lemma bound on I}
\end{align}
where $(a)$ is since by the inequality of arithmetic and geometric
means $\cosh(t)+\rho\sinh(t)=\left(\frac{1+\rho}{2}\right)e^{\theta x}+\left(\frac{1-\rho}{2}\right)e^{-\theta x}\geq\sqrt{1-\rho^{2}}$
and using $0<\eta<1$ and $|\rho|<\frac{1}{2}$; $(b)$ is obtained
by expanding the square, and noting $\sinh$ is odd and that as $\varphi(x)\propto e^{-x^{2}/2}$
is an even function, $\int\varphi(x)f(x)\d x=0$ for any odd function
$f$; $(c)$ is obtained by the identities 
\[
\int\varphi(x)\cosh(\eta x)^{2}\d x=e^{\eta^{2}}\cosh(\eta^{2}),\;\int\varphi(x)\cosh(\eta x)^{2}\d x=e^{\eta^{2}}\sinh(\eta^{2})\,,
\]
\[
\int\varphi(x)\cosh(\eta_{1}x)\cosh(\eta_{2}x)\d x=\frac{1}{2}e^{\frac{(\eta_{1}+\eta_{2})^{2}}{2}}+\frac{1}{2}e^{\frac{(\eta_{1}-\eta_{2})^{2}}{2}}\,,
\]
\[
\int\varphi(x)\sinh(\eta_{1}x)\sinh(\eta_{2}x)\d x=\frac{1}{2}e^{\frac{(\eta_{1}+\eta_{2})^{2}}{2}}-\frac{1}{2}e^{\frac{(\eta_{1}-\eta_{2})^{2}}{2}}\,;
\]
(d) is by Taylor expansion of $\cosh$ and $\sinh$ around $\eta^{2}$,
since $|\epsilon|\leq\sqrt{2}\eta\leq\sqrt{2}$ and where $C_{1}>0$
is a universal constant.

\paragraph*{Bounding $(\text{II})$}

The proof follows the one in \cite{wu2019EM} up until almost the
very last step. Recall that under $P$ one can write $X=R_{i}+Z_{i}$
for $i\in[d]$ where $R_{i}=S\cdot\eta u_{i}$ where $S\in\{\pm1\}$
and $\P[S=-1]=\delta_{*}$. Then, 
\begin{align}
(\text{II}) & =\E_{P_{X_{1}}}\left[\dkl(P_{X_{\perp}\mid X_{1}}||N(0,I_{d-1}))\right]\nonumber \\
 & \trre[\leq,a]\E\left[\dchis(P_{X_{\perp}\mid X_{1}}||N(0,I_{d-1}))\right]\nonumber \\
 & \trre[\leq,b]\eta^{2}\cdot\sum_{i=2}^{d}u_{i}^{2}\E_{P_{X_{1}}}\left[\E^{2}[R\mid X_{1}]\right]+C_{2}(\eta\lambda)^{4}\nonumber \\
 & \trre[=,c]\eta^{2}\cdot\sum_{i=2}^{d}u_{i}^{2}\E_{P_{X_{1}}}\left[\tanh^{2}(u_{1}X_{1}+\beta_{\rho})\right]+C_{2}(\eta\lambda)^{4}\nonumber \\
 & \trre[\leq,d]\eta^{2}\cdot\sum_{i=2}^{d}u_{i}^{2}\E_{P_{X_{1}}}\left[2(u_{1}^{2}\eta^{2}X_{1}^{2}+\beta_{\rho}^{2})\right]+C_{2}(\eta\lambda)^{4}\nonumber \\
 & \trre[\leq,e]4\eta^{4}\lambda^{2}+2C_{3}\eta^{2}\rho^{2}\lambda^{2}+C_{2}(\eta\lambda)^{4}\nonumber \\
 & \trre[\leq,f]C_{4}\eta^{2}(\eta^{2}+\rho^{2})\lambda^{2}\,,\label{eq: minimax lemma bound on II}
\end{align}
where $(a)$ is by bounding the KL divergence using the chi-square
divergence; $(b)$ stems from the Ingster-Suslina identity \cite{ingster2012nonparametric}
along with Taylor expansion (see details in \cite[Appendix B]{wu2019EM});
(c) follows from $\E_{\eta,\rho}[S\mid X_{1}]=\tanh(\eta X_{1}+\beta_{\rho})$
(see (\ref{eq: expected value of the sign given sample})); $(d)$
follows from $\tanh^{2}(x)\leq x^{2}$; $(e)$ follows from $|u_{1}|\leq1$,
$\|u_{\perp}\|\leq\lambda$, $\E_{P_{X_{1}}}[X_{1}^{2}]=1+\eta^{2}\leq2$,
and $\beta_{\rho}\leq C_{3}\rho$ for all $|\rho|\leq\frac{1}{2}$
and $C_{3}>0$ is a universal constant; $(f)$ holds for a universal
constant $C_{4}>0$ since $\lambda<1$.

Combining (\ref{eq: minimax lemma bound on I}) and (\ref{eq: minimax lemma bound on II})
we complete the proof of the lemma.

For completeness, we outline the proof of the lower bound using Fano's
method. The method states that if there exists a set of $M$ parameters
$\Theta_{M}=\{\theta_{1},\ldots,\theta_{M}\}$ such that $I(\theta;\underline{X})\lesssim\log M$
and $\|\theta_{m}-\theta_{m'}\|\geq\epsilon\eta$ for all $m,m'\in[M],m'\neq m$
then the lower bound is of the order $\epsilon\eta$. This is shown
by bounding the mutual information with the KL radius of $\Theta_{M}$
as $I(\theta;\underline{X})\lesssim n\max_{m\in[m]}\dkl(P_{\theta_{m},\rho}||P_{\theta_{0},\rho})$
for some $\theta_{0}$. As constructed in \cite[Appendix B]{wu2019EM},
there exists a set $\{\theta_{0}\}\cup\Theta_{M}$ with $M\geq e^{C_{0}d}$
for some $C_{0}$, and a small constant $c_{0}>0$ such that: (a)
$\|\theta_{m}\|=\eta$ for all $m\in0\cup[m]$; (b) $\|\theta_{m}-\theta_{m'}\|\geq c_{0}\epsilon\eta$
for all $m,m'\in[M],m'\neq m$; (c) $\|\theta_{m}-\theta_{0}\|\leq2c_{0}\epsilon\eta$
for all $m\in[m]$. By Lemma \ref{lem:minimax rate -bounding KL divergence}
\[
\frac{I(\theta;\underline{X})}{\log M}\asymp\frac{I(\theta;\underline{X})}{d}\lesssim\frac{n}{d}\max_{m\in[m]}\dkl(P_{\theta_{m},\rho}||P_{\theta_{0},\rho})\lesssim\frac{n}{d}\epsilon\cdot\eta^{2}(\max\{\eta,\rho\})^{2}
\]
and so choosing $\epsilon=\min\{1,\frac{1}{\eta\max\{\eta,\rho\}}\sqrt{\frac{d}{n}}\}$
yields a minimax lower bound of rate $\min\{\eta,\frac{1}{\eta}\sqrt{\frac{d}{n}},\frac{1}{\rho}\sqrt{\frac{d}{n}}\}$. 
\end{proof}
\begin{thm}
\label{thm: weight estimation minimax}For any $d,n\in\mathbb{N}$
let $\tilde{\rho}$ be any estimator of $\rho_{*}$ based on $\underline{X}=(X_{1},\ldots,X_{n})\stackrel{\tiny\mathrm{i.i.d.}}{\sim}P_{\theta_{*},\rho_{*}}$.
Then, 
\begin{align*}
\sup_{\tilde{\rho}(\theta_{*})}\inf_{\rho_{*}\in\mathbb{B}(\overline{\rho})}\E_{\theta_{*},\rho_{*}}\left[\ell(\tilde{\rho},\rho_{*})\right] & \asymp\begin{cases}
\overline{\rho}, & \|\theta_{*}\|\leq\frac{\overline{\rho}}{\sqrt{n}}\\
\frac{1}{\|\theta_{*}\|\sqrt{n}}, & \frac{\overline{\rho}}{\sqrt{n}}<\|\theta_{*}\|<1\,\\
\frac{1}{\sqrt{n}}, & \|\theta_{*}\|>1
\end{cases}.
\end{align*}
\end{thm}

\begin{proof}
Given the measurements $\{X_{i}\}_{i=1}^{n}$ the projections $\{\langle\hat{\theta}_{*},X_{i}\rangle\}_{i=1}^{n}$
are sufficient statistics for the estimation of $\rho_{*}$. Hence
we may assume that $d=1$, and we may write $X_{i}=S_{i}\theta_{*}+Z_{i}\in\mathbb{R}^{d}$
for $i\in[n]$ where $S_{i}\in\{\pm1\}$ and $\P[S_{i}=-1]=\delta_{*}$

\paragraph*{Upper bound}

The first case can be achieved by the trivial estimator $\tilde{\rho}=0$.
For the other two cases, as in the proof of Theorem \ref{thm: mean estimation minimax},
the estimator $\tilde{\rho}(\theta_{*})=\frac{1}{\|\theta_{*}\|}\langle\hat{\theta}_{*},\E_{n}[X]\rangle$
can be shown to achieve an error rate of $\max\{\frac{1}{\sqrt{n}},\frac{1}{\|\theta_{*}\|\sqrt{n}}\}$
where the first term stems from the empirical error of $\E_{n}[S]$,
and the second term is due to the additive error $Z$.

\paragraph*{Lower bound}

If $\|\theta_{*}\|>1$ we may bound the error rate of the given estimator
by the error rate of an estimator which known the noise sequence $\{Z_{i}\}_{i=1}^{n}$,
which, equivalently, has direct access to $\{S_{i}\}_{i=1}^{n}$.
This is a simple Bernoulli model and the error rate is $\frac{1}{\sqrt{n}}$.
We thus next assume that $\|\theta_{*}\|=\eta<1$. As the calculation
in the bound of term $(\text{I})$ in Lemma \ref{lem:minimax rate -bounding KL divergence},
\begin{align*}
\dkl(P_{\eta,\rho}||P_{\eta,0}) & \leq\dchis(P_{\eta,\rho}||P_{\eta,0})\\
 & =e^{-\eta^{2}/2}\int\varphi(x)\frac{\rho^{2}\cdot\sinh^{2}(\eta x)}{\cosh(\eta x)}\d x\\
 & \leq e^{-\eta^{2}/2}\int\varphi(x)\rho^{2}\cdot\sinh^{2}(\eta x)\d x\\
 & =e^{\eta^{2}/2}\sinh(\eta^{2})\rho^{2}\\
 & \leq C\eta^{2}\rho^{2}\,,
\end{align*}
for some $C>0$ using $\sinh(t)\leq C|t|$ for $t\leq1$. Le-Cam's
two point argument with $\rho=c_{0}\min\{\overline{\rho},\frac{1}{\eta\sqrt{n}}\}$
and $c_{0}>0$ small enough then results a minimax error rate of $\frac{1}{\eta\sqrt{n}}$. 
\end{proof}

\section{Miscellaneous \label{sec:Miscellaneous}}

\subsection{Useful results\label{subsec:Usful-results}}

We collect here several useful results which are repeatedly used throughout
the paper: 
\begin{itemize}
\item Relations for inverse temperature parameter: For $\beta\dfn\frac{1}{2}\log\frac{1-\delta}{\delta}$
it holds that $\tanh(\beta)=1-2\delta$, $\cosh(\beta)=\frac{1}{\sqrt{4\delta(1-\delta)}}$,
and $\frac{\d\beta_{\delta}}{\d\delta}=-\frac{1}{2\delta(1-\delta)}$
and $\frac{\d\beta_{\rho}}{\d\rho}=\frac{1}{1-\rho^{2}}$. 
\item Change of measure: Let $V\sim(1-\delta)\cdot N(\theta,1)+(1-\delta)\cdot N(-\theta,1)$
and let $Z\sim N(0,1)$. Then, for any integrable function $f$ 
\begin{align}
\E\left[f(V)\right] & =e^{-\theta^{2}/2}\cdot\E\left[f(Z)\cdot\cosh(\theta Z+\beta_{\delta})\right]\nonumber \\
 & =e^{-\theta^{2}/2}\cdot\E\left[f(Z)\cdot\left((1-\delta)e^{\theta Z}+\delta e^{-\theta Z}\right)\right]\,.\label{eq: change of measure}
\end{align}
\item For $U\sim N(\eta,1)$ \cite[Lemma 2]{pmlr-v65-daskalakis17b} 
\begin{equation}
\E\left[\tanh(U\theta)\right]\geq1-e^{-\eta\theta/2}\,,\label{eq: lower bound on expected tanh}
\end{equation}
\item $(a+b)^{k}\leq2^{k-1}(a^{k}+b^{k})$ for $k\geq1$. 
\item Chi-square tail bound: \cite[Remark 2.11]{boucheron2013concentration}
\begin{equation}
\P\left[\chi_{k}^{2}\geq2k+3t\right]\leq\P\left[\chi_{k}^{2}-k\geq2\sqrt{kt}k+2t\right]\leq e^{-t}\label{eq: chi-square tail bound}
\end{equation}
\item Stein's identity: Let $U\sim N(\theta,\sigma^{2})$ and $Z\sim N(0,1)$.
Let $f$ be a differentiable function such that $\E|f'(Z)|<\infty$.
Then, 
\begin{equation}
\E\left[f(U)(U-\theta)\right]=\sigma^{2}\E\left[f'(U)\right].\label{eq: Stein lemma}
\end{equation}
\item Let $V\sim(1-\delta)N(\theta,1)+\delta N(-\theta,1)$ with $\theta>0$
and $\delta<\frac{1}{2}$. Then $\P[V=v\mid|V|=v]>\P[V=-v\mid|V|=v]$
for any $v>0$. This follows from Chebyshev's sum inequality 
\begin{equation}
\frac{\P[V=v\mid|V|=v]}{\P[V=-v\mid|V|=v]}=\frac{(1-\delta)\varphi(\eta-v)+\delta\varphi(\eta+v)}{(1-\delta)\varphi(\eta+v)+\delta\varphi(\eta-v)}>1\label{eq: GM positive vs negative probability}
\end{equation}
since $1-\delta>\delta$ and $\varphi(\eta+v)<\varphi(\eta-v)$. 
\item Gaussian average of odd function.\textbf{ }Let $f(u)$ be an odd function
which is positive on $\mathbb{R}_{+}$ and negative on $\mathbb{R}_{-}$.
Let $U$ be a continuous random variable such that $\P[U=u\mid|U|=u]\geq\P[U=-u\mid|U|=u]$.
Then, $\E[f(U)]\geq0$. This is satisfied for $U\sim N(\eta,\sigma^{2})$
with $\eta>0$. 
\end{itemize}

\subsection{Convergence properties of one-dimensional iterations}
\begin{prop}[Convergence properties of one-dimensional iterations]
\label{prop: Properties of one dimensional iterations}Let $\theta\mapsto h(\theta)$
be an analytic monotonically increasing function, and $h(\theta)-\theta$
is not identically $0$. Let $\theta_{0}$ be given, and suppose that
either $\sup_{\theta>\theta_{0}}h(\theta)<\infty$ or $\lim_{\theta\to\infty}h'(\theta)<1$.
Let $h_{+}(\theta)$ be another function which satisfies the same
properties as $h(\cdot)$. 
\begin{enumerate}
\item \label{enu: general convergence in one-dim - existence of a fixed point}If
$h(\theta_{0})>\theta_{0}$ for some $\theta_{0}$ then $h(\theta)$
has at least a single fixed point in $(\theta_{0},\infty)$. 
\item \label{enu: general convergence in one-dim - alternating slope}Let
$\{\tilde{\theta}_{k}\}$ be an enumeration of the fixed points of
$h(\theta)$. For all $k\geq1$, if $h'(\tilde{\theta}_{k})<1$ then
$h'(\tilde{\theta}_{k+1})\geq1$ and if $h'(\tilde{\theta}_{k})>1$
then $h'(\tilde{\theta}_{k+1})\leq1$. 
\item Assume that $h(\theta)$ is strictly concave on $[\theta_{0},\infty)$.
If $h(\theta_{0})>\theta_{0}$ then $h(\theta)$ has a single fixed
point in $(\theta_{0},\infty)$. If $h(\theta_{0})\leq\theta_{0}$
then $h(\theta)$ has at most two fixed points in $(\theta_{0},\infty)$. 
\item \label{enu: general convergence in one-dim - convergence of monotonic}Consider
the iteration $\theta_{t+1}=h(\theta_{t})$. If $\theta_{1}=h(\theta_{0})>\theta_{0}$
(resp. $\theta_{1}<\theta_{0}$) then $\{\theta_{t+1}\}$ is monotonically
increasing (resp. decreasing) and converges to a fixed point $\theta_{\infty}$.
It holds that $h'(\theta_{\infty})\leq1$ (resp. $h'(\theta_{\infty})\geq1$). 
\item \label{enu: general convergence in one-dim - convergence dominance}In
addition, consider the iteration $\theta_{t+1}^{+}=h_{+}(\theta_{t}^{+})$
such that $\theta_{0}=\theta_{0}^{+}$, and suppose that $h_{+}(\theta)>h(\theta)$
on $[\theta_{0},\infty)$. $\theta_{t}^{+}\geq\theta_{t}$ for all
$t\geq1$, and this holds specifically in the limit $t\to\infty$.
Hence, if, in addition, $\lim_{t\to\infty}\theta_{t}=\lim_{t\to\infty}\theta_{t}^{+}=\theta_{\infty}$
then $\theta_{\infty}-\theta_{t}^{+}\leq\theta_{\infty}-\theta_{t}$,
i.e., the convergence of $\{\theta_{t}^{+}\}$ to the fixed point
is faster than that of $\{\theta_{t}\}$. 
\item \label{enu: general convergence in one-dim - convergence time of contraction}Convergence
rate of a contraction: If $\max_{\theta\in[\theta_{0},\theta_{\infty}]}h'(\theta)=\zeta<1$
then $\theta_{\infty}-\theta_{t}\leq(\theta_{\infty}-\theta_{0})\cdot\zeta^{t}$,
and $\theta_{\infty}-\theta_{t}\leq c$ for all $t\geq\frac{1}{1-\zeta}\cdot\log\frac{c}{\theta_{\infty}-\theta_{0}}$
(assuming $c\geq\theta_{\infty}-\theta_{0}$, otherwise $t\geq1$
suffice). 
\item \label{enu: general convergence in one-dim - convergence time of contraction plus constant}Suppose
that $0\leq h(\theta)\leq(1-a)\theta+b$ for $a\in(0,1)$. Then $h(\theta_{t})\leq\frac{2b}{a}$
for all $t\geq\frac{1}{a}\log\frac{a}{b}$ 
\end{enumerate}
\end{prop}

\begin{proof}
~ 
\begin{enumerate}
\item Under both conditions, there exists $\text{\ensuremath{\theta_{1}} such that \ensuremath{h(\theta_{1})<\theta_{1}}}$.
The claim follows from the intermediate value theorem for the function
$h(\theta)-\theta$. 
\item First, we note that such an enumeration is possible since $h(\theta)-\theta$
is analytic and not a constant, and so its zeros in $\mathbb{R}_{+}$
are isolated. Assume w.l.o.g. that $h'(\tilde{\theta}_{1})>1$. Thus,
$h(\theta)>\theta$ for all $\theta\in(\tilde{\theta}_{1},\tilde{\theta}_{2})$
and so $h'(\tilde{\theta}_{2})=\lim_{t\to0}\frac{h(\tilde{\theta}_{2})-h(\tilde{\theta}_{2}-t)}{t}=\frac{\tilde{\theta}_{2}-h(\tilde{\theta}_{2}-t)}{t}\leq1$.
The analogous property is proved similarly. 
\item Let $\theta_{1}$ be the minimal fixed point which is larger than
$\theta_{0}$. If $h(\theta_{0})>\theta_{0}$ then we must have $h'(\theta_{1})<1$
and by concavity $h'(\theta)<1$ for all $\theta\geq\theta_{1}$.
Thus there are no fixed points in $(\theta_{1},\infty)$. If $h(\theta_{0})\leq\theta_{0}$
then assume by contradiction that there are more than or three fixed
points $\{\tilde{\theta}_{k}\}$. By strict concavity, it is not possible
that $h'(\tilde{\theta}_{k})=1$ for any of the fixed point since
this fixed point would be unique. By the previous item, the signs
$h'(\tilde{\theta}_{k})-1$ are alternating. Thus, there exists $k$
such that $h'(\tilde{\theta}_{k})>1$ and $h'(\tilde{\theta}_{k+1})<1$.
Strict concavity implies that no fixed points are possible in $(-\infty,\tilde{\theta}_{k})$,
$(\tilde{\theta}_{k},\tilde{\theta}_{k+1})$ and $(\tilde{\theta}_{k},\infty)$.
So $h(\theta)$ cannot be more than two fixed points in $(\theta_{0},\infty)$. 
\item Assume $\theta_{1}>\theta_{0}$. Since $h$ is increasing, then $\theta_{2}=h(\theta_{1})>h(\theta_{0})=\theta_{1}$.
By induction, $\{\theta_{t}\}$ is an increasing an bounded sequence,
and thus has a limit $\tilde{\theta}$. Since $h$ is continuous $\tilde{\theta}=\lim_{t\to\infty}\theta_{t+1}=\lim_{t\to\infty}h(\theta_{t})=h(\lim_{t\to\infty}\theta_{t})=h(\tilde{\theta})$,
and so $\tilde{\theta}$ is a fixed point. The proof for $\theta_{1}<\theta_{0}$
is similar. 
\item By induction $\theta_{t+1}^{+}=h^{+}(\theta_{t}^{+})\geq h^{+}(\theta_{t})\geq h(\theta_{t})=\theta_{t+1}$. 
\item By induction and $|\theta_{t+1}-\theta_{\infty}|=|h(\theta_{t})-h(\theta_{\infty})|\leq\zeta|\theta_{t}-\theta_{\infty}|$.
To achieve $\theta_{\infty}-\theta_{t}\leq c$ we may require $t\geq\frac{\log\frac{\theta_{\infty}-\theta_{0}}{c}}{\log(1/\zeta)}$,
and the claim holds since $\frac{1}{\log(1/\zeta)}\leq\frac{1}{1-\zeta}$. 
\item By induction $\theta_{t+1}\leq(1-a)^{t}+\sum_{j=1}^{t}b(1-a)^{j}\leq(1-a)^{t}+\frac{b}{a}.$
Using $-a\geq\log(1-a)$ we have $(1-a)^{t}\leq\frac{b}{a}$ if $t\geq\frac{1}{a}\log\frac{a}{b}$. 
\end{enumerate}
\end{proof}

\subsection{Totally positive kernels and variation diminishing property\label{subsec:Totally-positive-kernels}}

Let $A,B\subseteq\mathbb{R}$. A kernel $K:A\times B\mapsto\mathbb{R}$
is said to be \emph{totally positive of order }$k$, $\TP_{k}$ if
for all $m\in[k]$ and all $x_{1}<\cdots<x_{m}$ and $y_{1}<\cdots<y_{m}$
(with $x_{i}\in A$ and $y_{i}\in B$ for $i\in[k]$) it holds that
\[
K\left(\begin{array}{c}
x_{1},\cdots,x_{m}\\
y_{1},\cdots,y_{m}
\end{array}\right)=\det\left[\begin{array}{ccc}
K(x_{1},y_{1}) & \cdots & K(x_{1},y_{m})\\
\vdots &  & \vdots\\
K(x_{m},y_{1}) & \cdots & K(x_{m},y_{m})
\end{array}\right]\geq0\,.
\]
If $K$ is $\TP_{k}$ for all $k\in\mathbb{N}$ then the kernel is
said to be totally positive (resp. strictly totally positive), which
is written $\TP_{\infty}$.

An important consequence of totally positive property is its \emph{variation
diminishing property}. The \emph{number of zero-crossings }of a function
$f\colon B\mapsto\mathbb{R}$ is the supremum of the numbers of sign
changes in sequences of the form $f(x_{1}),\ldots f(x_{m})$, for
$m\in\mathbb{N}$, $x_{i}\in B$ for all $i\in[m]$ and $x_{1}<\cdots<x_{m}$,
where zero values in the sequence are discarded. The following is
a result by Karlin: 
\begin{thm}[{{Variation diminishing property of totally positive kernels \cite[Theorem A.5 p. 759]{marshall1979inequalities}}}]
\label{thm: variation diminishing property}Let $A,B\subseteq\mathbb{R}$,
and let $K\colon A\times B\mapsto\mathbb{R}$ be Borel-measurable
and $\TP_{k}$. Let $\sigma$ be a regular $\sigma$-finite measure
on $B$, and let $f\colon B\mapsto\mathbb{R}$ be a bounded measurable
function such that 
\[
g(x)=\int_{B}K(x,y)f(y)\d\sigma(y)
\]
converges absolutely. If $f$ changes sign at most $j\leq k-1$ times
on $B$, then $g$ changes signs at most $j$ times on $A$. 
\end{thm}

\begin{prop}
\label{prop: Gaussian variation diminishing}Let $f\colon\mathbb{R}\mapsto\mathbb{R}$
be a bounded and measurable function. If $f$ has at most $j$ sign-changes
on $\mathbb{R}$, then $g\colon\mathbb{R}\mapsto\mathbb{R}$ defined
by $g(\eta)=\E_{U\sim N(\eta,1)}[f(U)]$ has at most $j$ signs-changes
on $\mathbb{R}.$ 
\end{prop}

\begin{proof}
The Gaussian kernel $K(x,y)=e^{-(x-y)^{2}}$ for $A=B=\mathbb{R}$
is $\TP_{\infty}$ \cite[Theorem A.6.B p. 759]{marshall1979inequalities}.
We use Theorem \ref{thm: variation diminishing property} and 
\begin{align}
g(\eta)=\E\left[f(U)\right] & =\int\varphi(u-\eta)f(u)\d u=\frac{1}{\sqrt{2\pi}}\int e^{-(u-\eta)^{2}/2}f(u)\d u=\frac{1}{\sqrt{\pi}}\int e^{-(\tilde{u}-\tilde{\eta})^{2}}f(\tilde{u})\d\tilde{u}=g(\sqrt{2}\tilde{\eta})
\end{align}
with $\tilde{u}=\frac{u}{\sqrt{2}}$ and $\tilde{\eta}=\frac{\eta}{\sqrt{2}}$
as well as the observation that $f(\frac{u}{\sqrt{2}})$ (resp. $g(\sqrt{2}\eta)$)
has the same zero-crossings as $f(u)$ (resp. $g(\eta)$). 
\end{proof}
\bibliographystyle{plain}
\bibliography{HM_GM}

\end{document}